\documentclass{dis}

\usepackage{tikz} 
\usetikzlibrary{arrows,chains,matrix,positioning,scopes}
\makeatletter
\tikzset{join/.code=\tikzset{after node path={%
\ifx\tikzchainprevious\pgfutil@empty\else(\tikzchainprevious)%
edge[every join]#1(\tikzchaincurrent)\fi}}}
\makeatother
\tikzset{>=stealth',every on chain/.append style={join},
         every join/.style={->}}

\usepackage{amsmath,amssymb,amsfonts} 
\usepackage{float,enumerate}
\usepackage[format=plain,justification=centering]{caption}

\usepackage[T1]{fontenc}

\usepackage{graphicx}
\usepackage{textcomp}

 \usepackage[lined,ruled]{algorithm2e}

\usepackage{amsthm}

\renewcommand{\Pr}{\mbox{\rm Pr}}
\newcommand{\A}{{\mathcal A}}
\newcommand{\Le}{{\mathcal L}}

\newcommand{\Borel}{{\mathcal B}}
\newcommand{\D}{{\mathfrak D}}

\newcommand{\R}{{\mathbb R}}
\newcommand{\N}{{\mathbb N}}
\newcommand{\Z}{{\mathbb Z}}

\newcommand{\spec}{\mathop{{\bf s}\rm pec}}
\renewcommand{\rho}{{\varrho}}
\def\esssup{\mathop{\rm ess\,sup}}
\def\dist{\mathop{{\bf d}\rm ist}}
\def\Lip{\mathop{\rm Lip}}
\def\min{{\rm min}}  
 
\def\dint{{\rm d}}
 
\def\a{{\alpha }} 
 
\def\e{{\varepsilon}}
\def\eps{{\varepsilon}}  \def\F{\mathcal F} 
\def\d{{\delta}} 
\def\phi{{\varphi}}  

\newcommand{\Ps}{\mathcal P} 
\newcommand{\comp}{{\rm {\bf c}omp}} 
\newcommand{\vol}{\operatorname{{\bf v}ol}}
\renewcommand{\rho}{{\varrho}}
\newcommand{\ball}{{B^d}}

\newcommand{\fc}{\mathcal F_C}

\newcommand{\vtn}{A^{\mathrm{simple}}_n}

\newcommand{\abs}[1]{\left\vert #1 \right\vert} 
\newcommand{\norm}[2]{\left\Vert #1  \right\Vert _{#2}} 
 
\newcommand{\set}[1]{\left\{#1\right\}}
\newcommand{\expect}{\mathbf E}

\newcommand{\scalar}[2]{\left\langle #1,#2\right\rangle}

\theoremstyle{plain}
\newtheorem{theorem}{Theorem}[chapter]
\newtheorem{lemma}[theorem]{Lemma}
\newtheorem{prop}[theorem]{Proposition}
\newtheorem{coro}[theorem]{Corollary}

\newtheoremstyle{citing}{}{}{\itshape}{}{\bfseries}{.}%
  { }{\thmnote{#3}}
\theoremstyle{citing}
\newtheorem{cit}{}

\theoremstyle{definition} 
\newtheorem{defin}[theorem]{Definition}
\newtheorem{remark}[theorem]{Remark}

\begin{document}

\mathclass{Primary 65C05; Secondary 60J10, 60J05, 60J22, 65C20, 37A25, 62F25, 65Y20.}
\thanks{
I want to thank Aicke Hinrichs, Thomas M\"uller-Gronbach, Erich Novak, Klaus Ritter 
and Mario Ullrich for several suggestions and helpful hints. 
The research was supported by DFG Priority Program 1324 and the DFG Research Training Group 1523.}
\abbrevauthors{Daniel Rudolf}
\abbrevtitle{Explicit error bounds for Markov chain Monte Carlo}

\title{Explicit error bounds for Markov chain Monte Carlo}

\author{Daniel Rudolf}
\address{Institute of Mathematics\\ University Jena\\
Ernst-Abbe-Platz 2\\ D-07743 Jena, Germany\\
E-mail: daniel.rudolf@uni-jena.de}

\maketitledis

\tableofcontents
\begin{abstract}
We prove explicit, i.e. non-asymptotic, error bounds for Markov
chain Monte Carlo methods.
The problem is to
compute the expectation of a function $f$ with respect to a measure $\pi$. 
Different convergence properties of Markov chains
imply different error bounds. For uniformly ergodic and reversible Markov chains 
we prove a lower and an upper error bound with respect to $\norm{f}{2}$.
If there exists an $L_2$-spectral gap, which is a weaker convergence property than uniform ergodicity, then 
we show an upper error bound with respect to $\norm{f}{p}$ for $p>2$. 
Usually a burn-in period is an efficient way to tune the algorithm.
We provide and justify a recipe how to choose the burn-in period. 
The error bounds are applied to the problem 
of the integration with respect to a possibly unnormalized density.
More precise, we consider the integration with respect to log-concave densities and the integration over convex bodies.
By the use of the Metropolis algorithm based on a ball walk and the hit-and-run algorithm it is shown that both problems 
are polynomial tractable.
\end{abstract}
\makeabstract



\chapter{Introduction and results}
In numerous applications one wants to compute the expectation of a function 
$f\colon D \to\R$ with respect to a probability measure $\pi$ defined on a measurable space $(D,\D)$. 
The goal is to approximate
\begin{equation}  \label{int}
S(f)=\int_D f(x) \, \pi(\dint x),
\end{equation}
where we assume that it is not possible to sample $\pi$ 
directly with reasonable cost.
In other words, we assume that there is no random number generator
which generates a sample with respect to $\pi$ reasonably fast.
This might happen if the available information on $\pi$ is 
incomplete or one has a complicated measurable space. 
However, many applications have in common that one knows enough about $\pi$
to design a Markov chain
which approximates the desired distribution.
Hence we assume that we cannot sample $\pi$ directly,
but we can run a Markov chain to get close to $\pi$.

Let us briefly illustrate such problems:
\begin{itemize}

\item Let $A\subset \R^d$ be an arbitrary \emph{convex 
body}\footnote{A convex body is a bounded convex set with non-empty interior.}. 
Suppose that we can sample the uniform distribution on 
\[
A\cap \ell \quad \text{for an arbitrary line} \quad \ell.
\] 
The goal is to simulate the uniform distribution on $A$, say $\mu_A$.
For a complicated $A$ it might be impossible 
to generate a uniformly distributed sample with reasonable cost.
But the hit-and-run algorithm (see Section~\ref{int_conv}) 
provides a Markov chain which has the
limit distribution $\mu_A$.
\item
Let $D\subset \R^d$ be a convex body. 
Suppose that $f\colon D\to \R$ is an integrable function with respect to $\pi_\rho$,
where $\rho$ is an unnormalized positive density and 
\[
\pi_\rho(A) = \frac{\int_A \rho(x) \,\dint x}{\int_D \rho(x) \,\dint x}, \quad A\subset D.
\] 
The goal is to approximate
\[
S(f,\rho)= \int_D f(x)\,\pi_\rho(\dint x) =\frac{\int_D f(x) \rho(x)\, \dint x}{\int_D \rho(x)\, \dint x}.
\]
By the Metropolis algorithm based on the ball walk (see Section~\ref{sec_S_f_rho}) one can construct
a Markov chain which has the limit distribution $\pi_\rho$.
It might be impossible to sample $\pi_\rho$ directly, 
in particular if $\rho$ is a complicated density function.
\end{itemize}
One can ask the following questions.
How does the error of numerical integration based on Markov chains behave? 
And, how long does the Markov chain need to get close to the limit distribution?\\

The thesis deals with the first question and,
 because of the close relation, touches briefly the second one. 
The Markov chain Monte Carlo method for approximating the expectation 
plays a crucial role in computer science, in statistical physics, in statistics,
and in financial mathematics, see e.g. \cite{MC_MC_pract,Mar,liu,mc_revolution,BroGelJonMen}.
Suppose that the function $f \colon D \to \R$ is given by an oracle 
which provides function values of $f$. The goal is to approximate $S(f)$.
The integral simplifies to a sum if the state space $D$ is finite, such that
\begin{equation}  \label{sum}
S(f)=\sum_{x\in D} f(x) \, \pi(x).
\end{equation}
We assume that the distribution $\pi$ can be simulated by a Markov chain $(X_n)_{n\in\N}$
with transition kernel $K$ and initial distribution $\nu$.
The distribution $\pi$ is the limit distribution, in particular it is stationary, i.e.
\[
  \pi(A) = \int_D K(x,A)\,\pi(\dint x),\quad A\in \D.
\]
Under weak assumptions on the Markov chain we obtain 
that after sufficiently many steps $m \geq n_0$, the distribution of $X_{m}$ is close to $\pi$.
The number $n_0$ determines the number of steps to get close to $\pi$, it is called the burn-in or the warm up period.
Afterwards, we approximate $S(f)$ by 
\begin{equation*}  
	S_{n,n_0}(f)=\frac{1}{n} \sum_{j=1}^n f(X_{j+n_0}).
\end{equation*} 
It is well known that an ergodic theorem\footnote{Suppose that $(D,\D)$ is countably generated. 
Let the Markov chain $(X_n)_{n\in\N}$ be $\phi$-irreducible ($\phi$ is a non-trivial $\sigma$-finite measure,  
for all $A\in\D$ and for all $x\in D$ there exists an $n\in\N$ such that $\phi(A)>0$ implies $K^n(x,A)>0$).
We assume that $\pi$ is a stationary distribution. Furthermore for all $A\in\D$ and for all $x\in D$ 
we have $\Pr(X_n\in A \text{ infinitely often} \mid X_1=x)=1$. Then  
$\lim_{n\to \infty} S_{n,n_0} (f) = S(f)$ almost surely. 
For a proof of the fact see \cite[Theorem~17.1.7, p.~427]{tweed}.
For a simple approach of a similar ergodic theorem we refer to \cite{AsGly}.
For a central limit theorem and fixed-width asymptotics of Markov chain Monte Carlo see \cite{Gey,Jon_et_al}.
} holds
which says that
\[
\lim_{n\to \infty} S_{n,n_0} (f) = \lim_{n\to \infty} \frac{1}{n} \sum_{j=1}^n f(X_{j+n_0})  = \int_D f(x)\, \pi(\dint x)  =S(f) \quad \text{almost surely.}
\]
This means that the algorithm is well defined 
but does not imply an error bound. 
It is a qualitative rather than a quantitative result.
We study the mean square error of $S_{n,n_0}$. 
For a function $f$, integrable with respect to $\pi$, 
it is given by
\begin{equation*}
	\label{error_ind}
  e_\nu(S_{n,n_0},f)=\left( \expect_{\nu,K} \abs{S_{n,n_0}(f)-S(f)}^2 \right)^{1/2},
\end{equation*}
where $\expect_{\nu,K}$ denotes the expectation of a Markov chain 
with transition kernel $K$ and initial distribution $\nu$.

The main topic of the thesis is the presentation of old and new explicit 
error bounds for computing the expectation
by Markov chain Monte Carlo.
These bounds are in terms of the $\norm{\cdot}{p}$-norm of the integrand $f$,
\[
  \norm{f}{p}  
  = \begin{cases}
      \left(\int_D \abs{f(x)}^p \, \pi(\dint x)\right)^{1/p}, & p\in[2,\infty),\\
      \pi\text{-}\esssup_{x\in D} \abs{f(x)}, & p=\infty.
    \end{cases}
\]
The kernel $K$ of the Markov chain determines the Markov operator
 \[
    Pf(x) = \int_D f(y)\, K(x,\dint y) , \quad x\in D,
  \]
 and $S(f)=\int_D f(x)\,\pi({\rm d}x)$
 can be considered as operator mapping into the constant functions.
 If $P$ is self-adjoint acting on $L_2$ then the Markov chain is called reversible.
 The asymptotic error is completely known if the underlying Markov chain is reversible, 
the initial distribution
  has a bounded density with respect to $\pi$ 
and one has $\norm{P^j-S}{L_1 \to L_1} \leq M \a^j $ for an $\a \in[0,1)$ and $M<\infty$, 
see 
Corollary~\ref{asymp_err_coro_gen}.
 One obtains
 \begin{equation} \label{asymp_err1}
   	\lim_{n\to \infty}  n\cdot 
	\sup_{\norm{f}{2} \leq 1} e_\nu(S_{n,n_0},f)^2
	= \frac{1+ \Lambda}{1-\Lambda} \leq \frac{2}{1-\Lambda},
 \end{equation}
and
 \begin{equation} \label{asymp_err2}
      \lim_{n_0 \to \infty} \sup_{\norm{f}{2}\leq 1} e_\nu(S_{n,n_0},f)^2 
  = \frac{1+\Lambda}{n(1-\Lambda)}-\frac{2\Lambda(1-\Lambda^n)}{n^2(1-\Lambda)^2} \leq \frac{2}{n(1-\Lambda)},
 \end{equation}
where $\Lambda=\sup\set{\a\mid\a\in\spec(P-S)}$.
Similar asymptotic estimates are shown in \cite{sokal,mathe1,bremaud,mathe2,RoRo}.
However, we want to have explicit error bounds.
The desired error estimate should behave asymptotically 
as described in \eqref{asymp_err1} and \eqref{asymp_err2}.
For $\Lambda$ close to $1$ the right hand sides of the equalities of the asymptotic error 
can be very well estimated 
by $\frac{2}{1-\Lambda}$ and $\frac{2}{n(1-\Lambda)}$.
The main goal is to prove non-asymptotic, explicit error bounds with respect to $\norm{f}{p}$ of the form
\[
	\sup_{\norm{f}{p} \leq 1} 
	e_\nu(S_{n,n_0},f)^2 \leq \frac{2}{n(1-\Lambda)}+ \frac{C_{\nu,p}\; \gamma^{n_0}}{n^2(1-\gamma)^2},
\]
where $C_{\nu,p}$ and $\gamma<1$ should be known.
If the initial distribution $\nu$ of the Markov chain is the stationary one, say $\pi$, 
then the influence of the initial part resulting from $\nu$ should vanish, i.e. $C_{\nu,p}=0$.
We give more details in the following.

  First we consider the special case where the state space is finite.
  Let the cardinality of $D$ be astronomically large, 
  say for example $\abs{D}=10^{30}$, such that an exact computation 
  of the sum \eqref{sum} might be practically impossible.
  Suppose that we have a Markov chain with transition matrix $P$ and initial distribution $\nu$.
  All definitions, such as stationarity, irreducibility, aperiodicity and all relevant facts of Markov chains 
  on finite state spaces are provided in Section~\ref{MC_finite}.
  The Markov chain is reversible if 
  the transition matrix $P=(p(x,y))_{x,y\in D}$ fulfills for a probability measure $\pi$ that
  \[
    \pi(x) p(x,y)=\pi(y) p(y,x) ,\quad x,y\in D.
  \]
  If the Markov chain is reversible, then let us define
  \[
      \beta=\norm{P-S}{\ell_2 \to \ell_2} = \max\set{\beta_1,\abs{\beta_{\abs{D}-1}}},
  \]
  where $\beta_1$ is the second largest and $\beta_{\abs{D}-1}$ the smallest eigenvalue of $P$.
  We consider reversible and ergodic Markov chains, 
  i.e. $\beta$, the second largest absolute eigenvalue of $P$, is 
  less than $1$. 
  Hence also $\beta_1$, the second largest eigenvalue of $P$, is less than $1$. 
  Section~\ref{sec_error_bound} contains the first error estimate.
  The explicit error bound is developed with respect to the $\ell_2$-norm of the integrand $f\in \R^D$.
  For 
  \[
    C=\sqrt{\norm{\frac{1}{\pi}}{\infty}} \norm{\frac{\nu}{\pi}-1}{2}
  \]
  we obtain in Theorem~\ref{main_fin} that
  \begin{equation} \label{intro_err}
    \sup_{\norm{f}{2} \leq 1}  e_\nu(S_{n,n_0},f)^2 \leq \frac{2 }{n(1-\beta_1)}+\frac{2C\beta^{n_0}}{n^2(1-\beta)^2}.
  \end{equation}
  Obviously $C$ is $0$ if $\nu$ is $\pi$ and the asymptotic estimates of 
\eqref{asymp_err1} and \eqref{asymp_err2} hold true.
However, the factor 
  $\norm{\frac{1}{\pi}}{\infty}$ is unsatisfactory for an extension to general state spaces.
  Furthermore we also provide a lower bound of the error, see Remark~\ref{rem_low_bound}.
  In Section~\ref{burn_in_fin} we suggest a choice of the burn-in.
  The main result is as follows.
\begin{cit}[Theorem~\ref{main_coro}]
Suppose that 
  \[
    n_0 = \max\set{\left\lceil \frac{\log\left(\sqrt{\norm{\frac{1}{\pi}}{\infty}} 
	  \norm{\frac{\nu}{\pi}-1}{2}\right)}{\log(\beta^{-1})} \right\rceil,0}.
  \]
Then 
\[
\frac{1+\beta_1}{n(1-\beta_1)}- \frac{4}{n^2(1-\beta)^2} \leq
\sup_{\norm{f}{2}\leq1}e_\nu(S_{n,n_0},f)^2
		\leq \frac{2}{n(1-\beta)} 
				+				 
				 \frac{2}{n^2(1-\beta)^2}.
\]  
\end{cit}
  The suggestion of the burn-in is optimal 
  in the following sense.
For $\eta>0$ let the number of steps $N=n+n_0$ of the Markov chain be large enough, let $\beta=\beta_1$ and assume that $C$ and $\beta$ obey an additional less restrictive 
condition.  
  Then the burn-in $n_{\text{opt}}$, 
  which minimizes the upper error bound of \eqref{intro_err}, 
  satisfies $n_{\text{opt}} \in [n_0,(1+\eta)n_0]$.

  In many examples an estimate for $\beta$ is available.
  In Section~\ref{toy_fin} we consider some illustrating examples 
  where all eigenvalues and eigenvectors are known,
  so that the exact error is computable.
  Then we compare the lower and upper estimates with the exact error.
  It turns out that the estimates are sharp depending on the available information of the eigenvalues.
  Similar estimates can be found in \cite{aldous} and \cite{niemiro_poka}.
  However, the suggestion of the burn-in and the lower bound seem to be new.

  After the study of Markov chains on finite state spaces 
  let us introduce the general state space setting.
  Assume that the measurable space $(D,\D)$ is given. 
  Then the desired expectation becomes an integral, see \eqref{int}.
 Suppose we have a Markov chain with transition kernel $K$ and initial distribution $\nu$. 
 Let us recall that the transition kernel $K$ defines the Markov operator 
  \[
    Pf(x) = \int_D f(y)\, K(x,\dint y), \quad x\in D,
  \]
  and $S(f)=\int_D f(x)\,\pi(\dint x)$
  can be considered as operator mapping into the constant functions. It is well known that reversibility of $K$ is equivalent
  to self-adjointness of $P$ acting on $L_2$. 
  In Section~\ref{MC_cont} we provide all definitions such as stationarity and reversibility in detail. 
  Furthermore it contains all relevant convergence properties of Markov chains.
  Mainly the two convergence properties of 
  Definition~\ref{L_p_exp_def} and Definition~\ref{def_spectral_gap} are essential:
  \begin{itemize}
   \item  Let $\a\in[0,1)$ and $M<\infty$.
	  The Markov chain is called $L_1$-exponentially convergent with $(\a,M)$
	  if 
          \[
	    \norm{P^j-S}{L_1\to L_1} \leq M\a^j , \quad j\in\N.
	  \]
	  For reversible Markov chains $L_1$-exponential convergence with $(\a,2M)$ 
	  is equivalent to $\pi$-a.e. uniform ergodicity with $(\a,M)$,
          see Proposition~\ref{uni_impl_geo}. 
   \item  The Markov operator has an $L_2$-spectral gap if 
	  \[
	    \beta=\norm{P-S}{L_2 \to L_2} <1,
	  \]
	  where the gap is given by $1-\beta$. 
	  The existence of an $L_2$-spectral gap implies an exponential convergence of $P^j$ to $S$
	  with respect to the $L_2$-operator norm for $j\to\infty$. 
  \end{itemize}
 Section~\ref{err_bound_gen} contains the error estimates for $S_{n,n_0}$.
 We explain the main results in the following.  
 Let $\Lambda$ be the largest element 
 of the spectrum of $P-S$ acting on $L_2$, i.e.  
  \[
   \Lambda=\sup\set{\a\mid\a\in\spec(P-S)}. 
  \]
 Suppose that the Markov chain is reversible and $L_1$-exponentially convergent with  $(\a,M)$.
 Furthermore assume that there exists a bounded density $\frac{d\nu}{d\pi}$ of the initial distribution $\nu$ with respect to $\pi$.
 For 
  \[
    C=M \norm{\frac{d\nu}{d\pi}-1}{\infty}
  \]
  we show in Theorem~\ref{main_thm_unif} that the error obeys
  \[
    \sup_{\norm{f}{2} \leq 1}  e_\nu(S_{n,n_0},f)^2 \leq \frac{2 }{n(1-\Lambda)}+\frac{2C\a^{n_0}}{n^2(1-\a)^2}.
  \]
  Note that the error bound is of the same form as for finite state spaces except for the fact that $\a$ of the
  $L_1$-exponential convergence appears. If the transition kernel is reversible one has 
  $\Lambda \leq \beta$ and in Proposition~\ref{uni_impl_geo} it is shown that $\beta\leq \a$.
  Hence one can further estimate the leading term of the error bound by using $(1-\Lambda)^{-1} \leq (1-\a)^{-1}$.
  Then a reasonable choice of the burn-in can be obtained by the same arguments as for finite state spaces.
  In Section~\ref{burn_in_sec} we also justify the choice of the burn-in by numerical experiments, which confirm the
  theoretical result. 
  \begin{cit}[Theorem~\ref{main_coro_gen}\,\eqref{erstens_burn_in}]
    Suppose that we have a Markov chain which
    is reversible with respect to $\pi$ and 
    $L_1$-exponentially convergent with $(\a,M)$. Let
\[
n_0= \max\set{\left\lceil \frac{\log(M \norm{\frac{d\nu}{d\pi}-1}{\infty})}{\log(\a^{-1})}\right\rceil,0}.
\]
Then 
\begin{align*}
\sup_{\norm{f}{2}\leq1}e_\nu(S_{n,n_0},f)^2
		& 	
		    \leq \frac{2}{n(1-\a)} 
				+				 
				 \frac{2}{n^2(1-\a)^2}.
\end{align*}
\end{cit}
  The condition that the Markov chain is $L_1$-exponentially convergent with $(\a,M)$ is rather restrictive. 
  This motivates the study of Markov chains which fulfill a weaker convergence property, 
  namely we assume that there is an $L_2$-spectral gap.
  Let us provide the main result.
\begin{cit}[Theorem~\ref{main_coro_gen}\,\eqref{zweitens_burn_in}]
  Suppose that 
  we have a Markov chain with Markov operator $P$ which has 
  an $L_2$-spectral gap, $1-\beta>0$. 
  For $p\in(2,\infty]$ let $n_0(p)$ be the smallest natural number (including zero) 
  which is greater than or equal to 
  \[
    \frac{1}{\log(\beta^{-1})} \cdot
    \begin{cases}
	\frac{p}{2(p-2)}\log\left(\frac{32p}{p-2} \norm{\frac{d\nu}{d\pi}-1}{\frac{p}{p-2}}\right), & p\in(2,4),\\
	\log\left(64 \norm{\frac{d\nu}{d\pi}-1}{2}\right) , & p\in[4,\infty]. 
    \end{cases}
  \]
  Then
  \[
    \sup_{\norm{f}{p}\leq1}e_\nu(S_{n,n_0(p)},f)^2
	\leq \frac{2}{n(1-\beta)} +\frac{2}{n^2(1-\beta)^2}.
  \]
\end{cit}
  For further details let us refer to Section~\ref{burn_in_sec}.
  There we justify the
   choice for the burn-in $n_0(p)$ by numerical experiments.
  Briefly summarized, 
  by weakening the convergence property we get an estimate of the error for a smaller class of functions.
  Namely, we prove an error bound for integrands $f$ which satisfy $\norm{f}{p}<\infty$ where $p>2$.

  The last chapter deals with applications.
  The problem of integration with respect to log-concave densities is the following.
  For a function $f\colon D\to \R$ and a convex body $D\subset\R^d$ 
  the goal is to approximate
  \[
    S(f,\rho)= \frac{\int_D f(x) \rho(x)\, \dint x}{\int_D \rho(x)\, \dint x},
  \]
  where $\rho$ is an unnormalized density. 
  The problem is linear in $f$ but not in $\rho$.
  Suppose that the domain $D$ is the $d$-dimensional Euclidean unit ball $\ball$. 
  Furthermore assume that $\rho$ is log-concave and 
  $\log \rho$ is Lipschitz continuous with Lipschitz constant $\text{\rm L}$.
  Hence we consider the class of densities
  \[
    \mathcal{R}^\text{\rm L}(\ball) = \set{\rho>0\mid  \rho\; 
		\text{is log-concave},\, \abs{ \log \rho(x) - \log \rho(y)  } \leq \text{\rm L} \norm{x-y}{\text{\rm E}} },
  \]
  where $\norm{\cdot}{\text{\rm E}}$ denotes the Euclidean norm.
  We analyze the Metropolis algorithm based on a $\delta$ ball walk,
  see Algorithm~\ref{mcmc_metropolis} on page \pageref{mcmc_metropolis} and 
  for the Procedure~\ref{ball_walk} see page \pageref{ball_walk}. 
  The algorithm generates the desired sample. 
  The sample, say $(X^\d_{n_0+1},\dots,X^\d_{n_0+n})$, is used to compute
  \[
    S_{n,n_0}^\d(f,\rho)=\frac{1}{n}\sum_{j=1}^n f(X^\d_{j+n_0}). 
  \]
  For an adapted $\d=\min\set{(d+1)^{-1},\text{\rm L}^{-1}}$ Math{\'e} and Novak showed in \cite{novak}
  that the Markov chain which is defined by the Metropolis algorithm based on a $\d$ ball walk
  has an $L_2$-spectral gap. This result is used to get an explicit error bound.
  We state the result for the unit ball and for simplicity we consider integrands $f$ with $\norm{f}{3}\leq1$.
  For
  \[
	n_0 
	\asymp  
	d \,\text{\rm L}  \max\set{d,\text{\rm L}^2} 
  \]
  the error obeys
  \[
  \sup_{\norm{f}{3}\leq1,\,\rho\in  \mathcal{R}^\text{\rm L}(\ball)} e(S_{n,n_0}^\d,(f,\rho)) 
  \prec \sqrt{\frac{d}{n}} \max\set{\sqrt{d},\text{\rm L}} + 
	\frac{d}{n} \max\set{d,\text{\rm L}^2},
 \]
  where $d\in\N$ and $\text{\rm L} \geq0$.\footnote{We use the notation $\prec$ and 
  $\asymp$
  as follows. 
  Let $(a_n)_{n\in\N}$ and $(b_n)_{n\in\N}$ be positive sequences. 
  We write $a_n \prec b_n$ if there exists an absolute constant $c$ such that $a_n \leq c\, b_n$ for all $n\in\N$. 
  We write $a_n \asymp b_n$ if $a_n \prec b_n$ and $b_n \prec a_n$.
}
  The geometry of the unit ball is essential for the estimate of the $L_2$-spectral gap of \cite{novak}, 
  since the ball walk 
  might get stuck on domains which have corners. 
  However, the results of Section~\ref{sec_S_f_rho} are slightly more general.
  There we treat balls with arbitrary radius $r>0$ and 
  the result is with respect to $\norm{f}{p}$ for $p>2$. 
  We refer to Theorem~\ref{err_S_f_rho} for the details.
  The number of function evaluations to obtain an error smaller than $\e$ is polynomially
  bounded in the dimension $d$ and the Lipschitz constant $\text{\rm L}$. 
  Hence the problem of integration with respect to a log-concave density 
  is tractable, see Novak and Wo{\'z}niakowski \cite{trac_vol1,trac_vol2}.

  The problem of integration on a convex body is as follows.
  Let  $A\subset \R^d$ be a convex body. The goal is to compute
  \begin{equation*}  
	S(f,A)
	    = \frac{1}{\vol_d(A)} \int_A f(x)\, \dint x, 
  \end{equation*}
  where $\vol_d(A)$ denotes the $d$-dimensional volume of $A$.
  In other words the goal is to approximate the expectation of $f$
  with respect to the uniform distribution, say $\mu_A$, on $A$. 
  The problem is linear in $f$ but not in $A$.
  Let $\ball \subset A \subset r\ball$ where $r\ball$ is the Euclidean ball with radius $r$ around $0$. 
  We assume that there is an oracle $\text{Or}_A(\ell)$ which returns a uniform 
  distributed state on $A\cap \ell$ for an arbitrary line $\ell$.
  Hence we consider state spaces from the class
  \[
    \mathcal{S}_d(r)=\set{ A\subset \R^d\; \text{convex} \mid \ball\subset A \subset r\ball}
  \]
  and we assume that $\text{Or}_A(\ell)$ is available for any $A\in \mathcal{S}_d(r)$.
  We analyze the hit-and-run algorithm, see Algorithm~\ref{mcmc_hit_and_run} 
  on page \pageref{mcmc_hit_and_run} and for the Procedure~\ref{hit_and_run} see page \pageref{hit_and_run}. 
  It generates the desired sample, say $(X_{n_0+1}^\text{har},\dots,X_{n+n_0}^\text{har})$.
  Afterwards we compute
  \[
    S_{n,n_0}^\text{har}(f,A)=\frac{1}{n} \sum_{j=1}^n f(X_{j+n_0}^\text{har}).
  \]
  The Markov chain generated by the hit-and-run algorithm has the right stationary distribution,
  see Lemma~\ref{hit_and_run_reversible} or \cite{hit_smith}.
  A result of Lov{\'a}sz and Vempala presented in \cite{hit_and_run_from_corner} provides an estimate of $1-\beta$.
  Hence there exists an $L_2$-spectral gap and we can apply the error bound
  of Theorem~\ref{main_coro_gen}\,\eqref{zweitens_burn_in}. 
  For simplicity suppose that $\norm{f}{3}\leq 1$.
  For 
  \[
    n_0 
    \asymp
    d^3 r^2 \log(r) 
  \]
  the error obeys
  \[
   \sup_{\norm{f}{3}\leq 1,\,A\in\mathcal{S}_d(r)} e(S_{n,n_0}^\text{har},(f,A)) 
    \prec  \frac{dr}{\sqrt{n}} + \frac{d^2 r^2}{n}.
  \]
  For the general result with respect to $\norm{f}{p}$ with $p>2$ we refer to Theorem~\ref{err_S_f_A}.
  The
  number of function evaluations to obtain an error smaller than $\e$ is polynomially
  bounded in the dimension $d$ and the radius $r$. 
  Hence the problem of integration on a convex body 
  is tractable, see \cite{trac_vol1,trac_vol2}.

\chapter{Finite state spaces} \label{ch_fin}

In the following we study the mean square error of Markov chain Monte Carlo methods 
on finite state spaces.
In Section~\ref{MC_finite} the basic definitions and properties of Markov chains on finite state spaces 
are stated. 
The estimate of the mean square error is shown in Section~\ref{sec_error_bound}.
We suggest and justify a recipe how to choose the burn-in.
Afterwards the error bound is applied to illustrating examples and
finally we discuss how the results fit into the published literature.  

\section{Markov chains}  \label{MC_finite}
In this section the basics of Markov chains on finite state spaces are provided.
Let $D$ be a finite set and $\Ps(D)$ the power set of $D$ such that the measurable state space $(D,\Ps(D))$ is given.


\begin{defin}[Markov chain]
A sequence of random variables $(X_n)_{n\in\N}$ 
on a probability space $(\Omega,\F,\Pr)$ mapping into $(D,\Ps(D))$
is called a \emph{Markov chain with transition matrix} $P=(p(x,y))_{x,y \in D}$ 
if for all $n\in\N$, all $x,y \in D$ and for all $x_1,\dots,x_{n-1}$ with
\[
  \Pr(X_1=x_1,\dots,X_{n-1}=x_{n-1},X_n=x)>0
\]
 one has
\begin{align*}
\Pr(X_{n+1}=y \mid X_1=x_1,\dots,X_{n-1}=x_{n-1},X_n=x) = \Pr(X_{n+1}=y \mid X_n=x)
= p(x,y).
\end{align*}
\end{defin}
All entries of the transition matrix $P$ are non-negative and the rows sum up to $1$. 
For $x,y\in D$ the value $p(x,y)$ is the probability of jumping from state $x$ 
to state $y$ in a single step of the chain.
The distribution 
\[
  \nu(x)=\Pr(X_1=x),\quad x\in D,
\]
is called the \emph{initial distribution}.\\


Suppose that we have a transition matrix $P$ and a probability measure $\nu$.
Any transition matrix $P$ has a \emph{random mapping representation}, see \cite[Proposition~1.5, p.~7]{peres}.
A random mapping representation of $P$ on state space $D$ is a function 
$\Phi\colon D\times [0,1] \to D$, which satisfies
\[
  \Pr(\Phi(x,Z)=y)=p(x,y),\quad x,y\in D,
\] 
where $Z:(\Omega, \F,\Pr) \to ([0,1],\Borel([0,1]))$ is a uniformly distributed random variable, where $\Borel([0,1])$
denotes the
Borel $\sigma$-algebra.
Then a Markov chain can be constructed as follows. 
If $(Z_n)_{n\in\N}$ is a sequence of i.i.d. random variables with uniform distribution,
and $X_1$ has distribution $\nu$, then it is easy to see that $(X_n)_{n\in \N}$ defined by
\[
  X_{n}=\Phi(X_{n-1},Z_n), \quad n\geq 2,
\]
is a Markov chain with transition matrix $P$ and initial distribution $\nu$.\\

In the following assume that we have a Markov chain $(X_n)_{n\in\N}$ with transition matrix $P$
and initial distribution $\nu$.
The expectation $\expect_{\nu,P}$ is taken with respect to 
the joint distribution of $(X_n)_{n\in\N}$, say $W_{\nu,P}$, 
which is defined on $(D^{\N},\sigma(\A))$ where
\begin{align*}
D^{\N} & = \set{ \omega=(x_1,x_2,\dots) \mid x_i \in D \text{ for all } i\geq1 } \quad \text{and} \\
\A     & =\bigcup_{k\in\N} \set{A_1\times A_2 \times \dots \times A_k \times D \times \dots \mid A_i\in \Ps(D),\; i=1,\dots,k } .
\end{align*} 
For $k\in \N$ one has with $A_1\times \dots \times A_k \subset D^{k}$ that
\begin{align*}  
W_{\nu,P}(A_1\times \dots \times A_k \times D \times \cdots) 
& = \sum_{x_1\in A_1} \dots \sum_{x_k\in A_k} \Pr(X_1=x_1,\dots,X_k=x_k).
\end{align*}
If the initial distribution $\nu$ is $\delta_x$, the point mass at $x\in D$, we say that the Markov chain starts at $x$. 
By 
\[
 Pf(x) =\sum_{y\in D} f(y)\, p(x,y)=\sum_{y\in D} f(y)\, \Pr (X_2=y \mid X_1=x)= \expect_{\delta_x,P}[ f(X_2) ]
\] 
one has the expectation of $f\in\R^{D}$ after a single step of the chain which starts at $x\in D$. 
The probability to get from $x$ to $y$ in $k\geq2$ steps is
\[
\Pr(X_{k+1}=y \mid X_1=x) = \sum_{x_2\in D} \sum_{x_3 \in D} \dots \sum_{x_{k}\in D} p(x,x_2) p(x_2,x_3) \dots p(x_k,y) = p^k(x,y),
\]
where $P^k=(p^k(x,y))_{x,y \in D}$ denotes the $k$th power of $P$. 
Then 
\[
P^kf(x)=\sum_{y\in D} f(y)\, p^k(x,y) = \sum_{y\in D} f(y)\, \Pr (X_{k+1}=y \mid X_1=x) 
= \expect_{\delta_x,P}[ f(X_{k+1})]
\]
is the expectation after $k$ steps of the Markov chain which starts at $x$.
Similarly we consider the application of $P$ to a probability measure $\nu$, i.e. 
\[
\nu P(x)= \sum_{y\in D}  p(y,x)\,\nu(y) = \sum_{y\in D} \Pr(X_2=x \mid X_1=y)\, \nu(y)
= \Pr(X_2=x).
\] 
This is the distribution which arises after a single transition where the initial state is chosen by $\nu$. 
The distribution which arises after $k\geq1$ steps is given by 
\[
\nu P^k(x)= \sum_{y\in D}  p^k(y,x)\,\nu(y) 
= \sum_{y\in D} \Pr(X_{k+1}=x \mid X_1=y)\, \nu(y)
 = \Pr(X_{k+1}=x).
\]


In the following we present properties of transition matrices. 


\begin{defin}[\index{irreducibility}irreducibility, \index{aperiodicity}aperiodicity, \index{periodicity} periodicity]
A transition matrix $P$ is called \emph{irreducible} if for all $x,y\in D$ there exists a $k\in\N$ such that
\[
p^k(x,y)>0,\quad \mbox{where} \quad P^k=(p^k(x,y))_{x,y\in D}.
\]
A transition matrix $P$ is called \emph{aperiodic} if we have for all $x\in D$ that
\[
\mbox{{\bf g}cd}(\set{k\in\N \mid p^k(x,x)>0})=1,
\]
where {\bf g}cd denotes the greatest common divisor. If $P$ is not aperiodic we call it \emph{periodic}.
\end{defin}
If a transition matrix is irreducible, then 
one can reach every state $y$ from every state $x$ in finitely many steps. 
Aperiodicity ensures that the number of steps to return 
 to an arbitrary state is not in $\set{m,2m,3m,\dots}$ for $m>1$.

\begin{defin}[stationarity\index{stationarity}]
Let $\pi$ be a probability measure on $D$. 
Then $\pi$ is called a \emph{stationary distribution} of a transition matrix $P$ if
\[
\pi P (x) = \pi(x), \quad x\in D.
\]
\end{defin}
If the initial distribution of a Markov chain with transition matrix $P$ is a stationary one, say $\pi$, 
then after a single transition 
the same distribution as the initial one appears, i.e. 
\[
 \Pr(X_1=x)=\pi(x)=\pi P(x)=\Pr(X_2=x), \quad x\in D.
\]

 
\begin{defin}[reversibility\index{reversibility}]
Let $\pi$ be a probability measure on $D$.
A transition matrix $P$ is called \emph{reversible with respect to} $\pi$ if 
\[
\pi(x)p(x,y)=\pi(y)p(y,x), \quad x,y\in D.
\]
\end{defin}
If a transition matrix $P$ is reversible with respect to a probability measure $\pi$, 
then $\pi$ is a stationary distribution (see \cite[Proposition~1.19, p.~14]{peres}).
If the initial distribution of a Markov chain with transition matrix $P$ is $\pi$, 
then reversibility with respect to $\pi$ is equivalent to 
\[
\Pr(X_1=x,X_2=y)=\Pr(X_1=y,X_2=x), \quad x,y \in D.
\]

\begin{defin}[lazy version\index{laziness}]
Let $P$ be a transition matrix and let $I$ be the identity matrix. 
Then we call 
\[
  \widetilde{P}=\frac{1}{2}(I+P),
\]
the \emph{lazy version of} $P$.
\end{defin}
Let $\pi$ be a stationary distribution of a transition matrix $P$, then $\pi$
is also stationary with respect to $\widetilde{P}$.
If $P$ is irreducible, reversible with respect to $\pi$ and periodic, 
then one can pass over to the lazy version $\widetilde{P}$ 
and obtains that $\widetilde{P}$ is irreducible, reversible with respect to $\pi$ 
and aperiodic.\\

A Markov chain is called irreducible, periodic, aperiodic and reversible with respect to $\pi$
if the corresponding transition matrix is 
irreducible, periodic, aperiodic and reversible with respect to $\pi$, respectively.\\

Let us state some well known implications of the different properties. 
For proofs or more details see \cite{bremaud,stroock_book,peres}. 
For every transition matrix there exists a stationary distribution 
and if the matrix is irreducible 
then there exists a unique stationary distribution, which is positive 
(\cite[Proposition~1.14, p.~12 and Corollary~1.17, p.~14]{peres}).
Note that if $\xi$ is an eigenvalue of a transition matrix $P$, 
then $\abs{\xi}\leq1$ (\cite[Lemma~12.1(i), p.~153]{peres}). 
Furthermore, for irreducible transition matrices $1$ is a simple eigenvalue (\cite[Lemma~12.1(ii), p.~153]{peres}). 
If the Markov chain is aperiodic and irreducible, 
then $-1$ is not an eigenvalue of $P$ (\cite[Lemma~12.1(iii), p.~153]{peres} or 
\cite[Theorem~5.1.14, p.~113]{stroock_book}). 
These eigenvalue results are also known as results of the Perron-Frobenius Theorem, see \cite{Sen}.\\

In the following we always assume that the Markov chains are irreducible, aperiodic 
and reversible with respect to a probability measure $\pi$. 
Hence $\pi$ is the stationary distribution.
Aperiodicity is not essential. 
For a Markov chain with periodic transition matrix $P$ and initial distribution $\nu$ 
we may consider
a \emph{lazy Markov chain}, i.e. a chain with aperiodic transition matrix $\widetilde{P}$, 
the lazy version of $P$,
and initial distribution $\nu$.\\

Let us define a weighted inner product for $f,g\in\R^{D}$ by
\[
\scalar{f}{g}=\sum_{x\in D} f(x)g(x)\,\pi(x)
\]
and let $\norm{f}{2}=\scalar{f}{f}^{1/2}$. 
By considering the inner product it is easy to see that reversibility is equivalent to $P$ being self-adjoint.  
Applying the spectral theorem for self-adjoint transition matrices 
and the fact that the Markov chain is irreducible one obtains that $P$ has real eigenvalues 
\[
1=\beta_0>\beta_1\geq \beta_2 \geq \dots \geq \beta_{\abs{D}-1} \geq -1.
\]
If the transition matrix is aperiodic, then $\beta_{\abs{D}-1} > -1$.
There exists a basis of orthonormal eigenfunctions (vectors) 
$\set{u_0,u_1,\dots,u_{\abs{D}-1}}$, i.e.
for $i,j\in\set{0,\dots,\abs{D}-1}$ one has
\[
Pu_i=\beta_i u_i, \qquad \scalar{u_i}{u_j} = \delta_{ij} = \begin{cases} 1, & i=j,\\
					            0, & i\not=j.
				            \end{cases}
\]
Clearly, $u_0(x)=1$ for all $x\in D$ and $S(u_i)=\scalar{u_i}{u_0}=0$ for 
$i\in\set{1,\dots,\abs{D}-1}$.
By the spectral structure of the transition matrix one has
\begin{equation} \label{P_k_spec}
P^k=(p^k(x,y))_{x,y \in D} = \sum_{i=0}^{\abs{D}-1} \beta_i^k\, \left(\,    u_i(x) u_i(y)\, \pi(y)   \,\right)_{x,y \in D},
\end{equation}  
see \cite[p. 203]{bremaud} for details.\\

For $ p\in[1,\infty]$ let 
\[
\norm{f}{p}= \begin{cases}
(\sum_{x\in D} \abs{f(x)}^p \pi(x) )^{1/p}, & p\in[1,\infty),\\
\sup_{x\in D}\abs{f(x)}, & p=\infty. 
\end{cases}
\]
The weighted vector space $\ell_p=\ell_p(D,\pi)$ is defined by the normed space $(\R^D,\norm{\cdot}{p})$.
Furthermore let
\[
\ell^0_p=\ell^0_p(D,\pi)=\set{f\in \ell_p\mid S(f)=0}.
\] 
Then 
\[
\ell_2= \ell_2^0 \oplus (\ell_2^0)^\bot,\quad \mbox{with} \quad
(\ell_2^0)^\bot=\set{f \in \R^D\mid f \equiv c,\; c\in \R}= \mbox{Eig}(P,1),
\]
where $\mbox{Eig}(P,1)$ is the eigenspace of $P$ with respect to the eigenvalue $1$.
Note that for the next well known result it is not assumed that the transition matrix is
reversible with respect to $\pi$.
\begin{lemma} \label{1_EW}
 Let $p\in[1,\infty]$ and $f\in \R^D$.
 For any transition matrix $P$ with stationary distribution $\pi$ one obtains 
 \[
 \norm{Pf}{p} \leq  \norm{f}{p} \quad \text{and} \quad \norm{P}{\ell_p\to \ell_p}=1.
 \]  
\end{lemma}
\begin{proof}
By the Jensen inequality $(\text{J})$\footnote{
Let $(D, \D, \mu)$ be a probability space. 
For any convex function $h\colon \R \to \R$  and 
for any function $f\colon D \to \R$ that is integrable with respect to $\mu$, the Jensen inequality is
$
h(\int_D f\, \dint \mu ) \leq \int_D  (h \circ f)\, \dint \mu . 
$
} 
and stationarity (stat.) one has
\begin{align*}
\sum_{x\in D} \abs{Pf(x)}^p \pi(x) & \leq  \sum_{x\in D} \left(\sum_{y\in D} \abs{f(y)} p(x,y) \right)^p \pi(x) \\
 & \underset{(\text{J})}{\leq} \sum_{x\in D} \sum_{y\in D}  \abs{f(y)}^p p(x,y)\, \pi(x)\underset{(\text{stat.})}{=} \sum_{x\in D} \abs{f(x)}^p \pi(x).
\end{align*}
If $p=\infty$ then
\begin{align*}
\norm{Pf}{\infty}=\sup_{x\in D} \abs{Pf(x)} \leq \sup_{x\in D} \sum_{y\in D} \abs{ f(y)} p(x,y) 
 = \norm{f}{\infty}. 
\end{align*}
Since $\norm{Pf}{p} \leq  \norm{f}{p}$ and $Pu_0=u_0$ 
with $\norm{u_0}{p}=1$ we have $\norm{P}{\ell_p\to\ell_p}=1$.
\end{proof}



Let us briefly explain how to quantify the difference of two distributions.  
For any measure $\nu$ let 
\[
\norm{\frac{\nu}{\pi}}{2} = \left( \sum_{x\in D} \left(\frac{\nu(x)}{\pi(x)} \right)^2 \pi(x) \right)^{1/2}.
\]
If $\nu$ is a probability measure on $D$, then 
the quantity $\norm{\frac{\nu}{\pi}-1}{2}$ is related to the $\chi^2$-contrast, defined as follows.

\begin{defin}[\index{$\chi^2$-contrast}$\chi^2$-contrast]
The \emph{$\chi^2$-contrast} of a distribution $\nu$ and a positive distribution $\mu$ is defined by
\[
\chi^2(\nu,\mu)=\sum_{x \in D} \frac{(\nu(x)-\mu(x))^2}{\mu(x)}.
\]
\end{defin}
The $\chi^2$-contrast is not symmetric and therefore no distance.
By a simple calculation one obtains
\[
 \chi^2 (\nu,\pi) = \norm{\frac{\nu}{\pi}-1}{2}^2 . 
\] 

The functional $S$ can be interpreted as operator which maps into the constant functions,
consequently one can see
\[
S=(\,\pi(y)\,)_{x,y\in D}
\]
as a matrix.
Furthermore let 
\[
\beta=\max \set{\beta_1, \abs{\beta_{\abs{D}-1}}}
\]
be the second largest absolute value of the eigenvalues of $P$. 
Now we state a property of the matrix $P-S$.
\begin{lemma} \label{P_S_fin}
Let $P$ be a reversible transition matrix with respect to $\pi$. 
Then 
\begin{equation}
\label{P_S_2_ident}
   \norm{P^n-S}{\ell_2 \to \ell_2} = \norm{P^n}{\ell_2^0 \to \ell_2^0} = \beta^n,\quad n\in\N.
\end{equation}
\end{lemma}
\begin{proof}
The self-adjointness of $P$ implies 
$\norm{P}{\ell_2^0 \to \ell_2^0}=\max\set{\beta_1,\abs{\beta_{\abs{D}-1}}}=\beta$, 
consequently $\norm{P^n}{\ell_2^0 \to \ell_2^0}  =  \beta^n$.
By
\begin{align*}
\norm{P^n-S}{\ell_2\to \ell_2} 	& =\sup_{\norm{f}{2}\leq1} \norm{(P^n-S)f}{2} = \sup_{\norm{f}{2}\leq1} \norm{P^n(f-S(f))}{2} \\
			& \leq 
				\sup_{\norm{g}{2}\leq1,\; S(g)=0} \norm{P^n g}{2} 
  		= \norm{P^n}{\ell_2^0 \to \ell_2^0} 
\end{align*}	
and													
\begin{align*}
\norm{P^n}{\ell_2^0 \to \ell_2^0} & = \sup_{\norm{g}{2}\leq1,\; S(g)=0} \norm{P^n g}{2}
	= \sup_{\norm{g}{2}\leq1,\; S(g)=0} \norm{P^n g-S(g)}{2}\\
		&	\leq \sup_{\norm{f}{2}\leq1} \norm{(P^n-S)f}{2} = \norm{P^n-S}{\ell_2\to \ell_2}
\end{align*}
claim \eqref{P_S_2_ident} is shown. 
\end{proof}
%
%
%
This section is finished by stating a well known fact which shows that 
$\nu P^k$ converges to $\pi$ for increasing $k$
exponentially fast if $\beta<1$.

\begin{coro}
Let $P$ be a transition matrix and $\nu$ a probability measure on $D$.
Let $P$ be reversible with respect to $\pi$. 
Then
\[
 \norm{\frac{\nu P^k}{\pi} -1}{2} 
\leq \beta^k \norm{\frac{\nu}{\pi}-1}{2}, \quad k\in\N.
\]
\end{coro}
\proof
The assertion is proven by
\begin{align*}
 \norm{\frac{\nu P^k}{\pi}-1}{2} \underset{\text{(rev.)}}{=} \norm{P^k\left(\frac{\nu}{\pi}\right)-1}{2} 
		= \norm{P^k\left(\frac{\nu}{\pi}-1\right)}{2}  
		&\leq 
		\beta^k \norm{\frac{\nu}{\pi}-1}{2}. \sq
\end{align*}

\section{Error bounds}  \label{sec_error_bound}

In this section explicit error bounds are proven. 
Let us repeat the idea of Markov chain Monte Carlo.
Suppose we have a Markov chain $(X_n)_{n\in\N}$ with transition matrix $P$ and initial distribution $\nu$,
where $\pi$ is a stationary distribution, and we compute
\[
  S_{n,n_0}(f)=\frac{1}{n}\sum_{j=1}^n f(X_{j+n_0})
\]
as approximation for $S(f)=\sum_{x\in D} f(x)\, \pi(x)$. 
The error is measured in the mean square sense, i.e.
\[
   e_\nu(S_{n,n_0},f)=\left( \expect_{\nu,P} \abs{S_{n,n_0}(f)-S(f)}^2 \right)^{1/2}.
\]
Now let us present a helpful result. 
\begin{lemma}
Let $(X_n)_{n\in\N}$ be a Markov chain with transition matrix $P$ and initial distribution $\nu$. 
Then for $i,j\in \N$ with $j\leq i$ it follows that
\begin{equation}
  \label{err_help}
	\expect_{\nu,P}[{f(X_i)f(X_j)}]=\sum_{x\in D} P^j(f P^{i-j}f)(x)\, \nu(x).
\end{equation}
Moreover, if $\pi$ is a stationary distribution and $\nu=\pi$ then
 \begin{equation}
   \label{err_help2}
	\expect_{\pi,P}[{f(X_i)f(X_j)}]=\scalar{f}{ P^{i-j}f}.
\end{equation}
\end{lemma}
\begin{proof}
The calculation
\begin{align*}
\expect_{\nu,P}[{f(X_i)f(X_j)}] & 
= \sum_{x_1 \in D} \dots \sum_{x_i \in D}  f(x_j)f(x_i)\, p(x_{i-1},x_i) \cdots p(x_1,x_2) \nu(x_1)\\
& = \sum_{x_1 \in D} \cdots \sum_{x_j \in D} f(x_j) P^{i-j}f(x_j)\, p(x_{j-1},x_j)\cdots p(x_1,x_2) \nu(x_1)\\
& = \sum_{x\in D} P^j (f P^{i-j}f) (x)\, \nu(x)
\end{align*}
proves \eqref{err_help} and by using $\pi P(x)=\pi(x)$ one can see \eqref{err_help2}. 
\end{proof}
In the following a special case of the method $S_{n,n_0}$ is considered. 
In this case the initial distribution is a stationary one, thus, 
the distribution after a single transition does not change. Hence 
it is not necessary 
to do any burn-in, i.e. $n_0=0$. 
Afterwards the error representation of the special case is set in relation 
to the error where the initial distribution might differ from a stationary one.
The techniques which are used are adapted from \cite{expl_error} and \cite{Ru11}.
 
In the following $S_{n,0}$ is always denoted as $S_n$. 
Let us start with a result
stated in \cite[Proposition~2.1, p.~3]{bassetti_dia}.
\begin{prop} \label{expl_stat}
Let $f\in \R^D$ and let $(X_n)_{n\in\N}$ be a Markov chain with transition matrix $P$ 
and initial distribution $\pi$. Let $P$ be reversible with respect to $\pi$.
Then
\begin{equation} \label{err_present}
e_\pi(S_n,f)^2 = \frac{1}{n^2} \sum_{k=1}^{\abs{D}-1} 
a_k^2\, W(n,\beta_k),
\end{equation}
where
\begin{equation*}
 	a_k=\scalar{f}{u_k}\quad\mbox{and}\quad
 	W(n,\beta_k)=\frac{n(1-\beta_k^2)-2\beta_k(1-\beta_k^{n})}{(1-\beta_k)^2}.
\end{equation*}
\end{prop}
\proof
 Let us consider $g=f-S(f)\in\R^D$.
 The error obeys
 \begin{align*}
      e_\pi(S_n,f)^2 &	=\expect_{\pi,P}\abs{\frac{1}{n}\sum_{j=1}^n g(X_j )}^2
     			=\frac{1}{n^2} \expect_{\pi,P} \abs{ \sum_{j=1}^n g(X_j)}^2 \\
		       &=\frac{1}{n^2} \sum_{j=1}^n \expect_{\pi,P}\,[g(X_j)^2]
     			 + \frac{2}{n^2} \sum_{j=1}^{n-1} \sum_{i=j+1}^n \expect_{\pi,P}\,[g(X_j)g(X_i)].
 \end{align*}
 By using the orthonormal basis $\set{u_0,u_1\dots,u_{\abs{D}-1} }$ we have 
 $g(x)=\sum_{k=1}^{\abs{D}-1} a_k u_k(x)$. 
 For $j\leq i$ one obtains
 \begin{align*}
  \expect_{\pi,P}\,[g(X_i)g(X_j)]	
    &\;= \sum_{k=1}^{\abs{D}-1} \sum_{l=1}^{\abs{D}-1} a_ka_l\, \expect_{\pi,P}\, [u_k(X_i)u_l(X_j)]\\ 
    &\underset{\eqref{err_help2}}{=} 
    \sum_{k=1}^{\abs{D}-1} \sum_{l=1}^{\abs{D}-1} a_ka_l\, \scalar{u_k}{P^{i-j}u_l}\\
    &\;= \sum_{k=1}^{\abs{D}-1} \sum_{l=1}^{\abs{D}-1} a_ka_l\, \beta_l^{i-j}\scalar{u_k}{u_l}
     = \sum_{k=1}^{\abs{D}-1} a_k^2\; \beta_k^{i-j}.
\end{align*} 		
 The last two equalities follow from the orthonormality of the basis of the eigenvectors.
 Altogether this gives
 \begin{align*}
 		 e_\pi(S_n,f)^2 & =\frac{1}{n^2} \sum_{k=1}^{\abs{D}-1} a_k^2 
 		\left[ n + 2 \sum_{j=1}^{n-1} \sum_{i=j+1}^n \beta_k^{i-j} \right]\\
 	&	=\frac{1}{n^2} \sum_{k=1}^{\abs{D}-1} a_k^2 
 		\left[ n + 2 \frac{(n-1)\beta_k-n\beta_k^2+\beta^{n+1}_k}{(1-\beta_k)^2}\right]
 		=\frac{1}{n^2} \sum_{k=1}^{\abs{D}-1} a_k^2\, W(n,\beta_k). \sq
  \end{align*} 

Let us consider $W(n,\beta_k)$ to simplify and interpret Proposition~\ref{expl_stat}.

\begin{lemma} \label{lemm_mono_inc}
For all $n\in\N$ and $k\in\set{1,\dots,\abs{D}-1}$ it follows that
\begin{align}  \label{mono_inc}
W(n,\beta_k) & \leq W(n,\beta_1)\leq \frac{2n}{1-\beta_1}.
\end{align}
\end{lemma}

\begin{proof}
We will show that the mapping $x \mapsto W(n,x)$ is increasing on $[-1,1)$, 
so that $W(n,\beta_k) \leq W(n,\beta_1)$. 
For $i\in\set{0,\dots,n}$ we have
\[
x^{n-i}\leq1\quad \Longleftrightarrow \quad (1-x^i)\,x^{n-i} 
\leq 1-x^i \quad\Longleftrightarrow\quad x^{n-i}+x^i \leq 1+x^n.
\]
This implies
\[
x^j+x^{j+1}+x^{n-j-1}+x^{n-j} \leq 2(1+x^n), \quad j\in \set{0,\dots,n-1}
\]
and
\[
(1+x)\sum_{j=0}^{n-1} x^j =\frac{1}{2} \sum_{j=0}^{n-1} (x^j+x^{j+1}+x^{n-j-1}+x^{n-j}) \leq n(1+x^n).
\]
Now
\[
\frac{d \,W}{dx}(n,x)=-2\frac{(1+x)\sum_{j=0}^{n-1} x^j-n(1+x^n)}{(1-x)^2} \geq 0
\]
and the first inequality of the assertion is proven. By
\[
W(n,x)\leq 	
\begin{cases}
   \frac{n(1+x)-2xn}{1-x}\leq \frac{2n}{1-x},\quad 	& x\in[-1,0],\\
   \frac{n(1+x)}{1-x}\leq \frac{2n}{1-x},\quad		& x\in(0,1),
\end{cases}
\]

everything is shown.
\end{proof}

An explicit formula of the error is established
if the initial state is chosen by a stationary distribution. 
Let us consider the maximal error of $S_n$ for $f$ which satisfy $\norm{f}{2}\leq 1$.

\begin{coro}  \label{prop_stat}
Let $(X_n)_{n\in\N}$ be a Markov chain with transition matrix $P$ and initial distribution $\pi$. 
Let $P$ be reversible with respect to $\pi$.
Then
			\begin{equation}  \label{err_class_eq}
				\sup_{\norm{f}{2}\leq1}e_\pi(S_n,f)^2 = 
												\frac{1+\beta_1}{n(1-\beta_1)}-\frac{2\beta_1(1-\beta_1^n)}{n^2(1-\beta_1)^2}
												\leq \frac{2}{n(1-\beta_1)}.
			\end{equation}
\end{coro} 

\begin{proof}
The individual error of $f$ is 
\begin{align*}
 e_\pi(S_n,f)^2	& 	\underset{\eqref{err_present}}{=} \frac{1}{n^2} \sum_{k=1}^{\abs{D}-1} a_k^2\, W(n,\beta_k)
								\leq 		\frac{\norm{f}{2}^2}{n^2} \max_{k=1,\dots, \abs{D}-1}  W(n,\beta_k) \\
\displaybreak
	&
			\underset{\eqref{mono_inc}}{=} \frac{\norm{f}{2}^2}{n^2} W(n,\beta_1)
								=\frac{1+\beta_1}{n(1-\beta_1)}\norm{f}{2}^2-\frac{2\beta_1(1-\beta_1^n)}{n^2(1-\beta_1)^2}\norm{f}{2}^2,
\end{align*}
where $a_k$ is chosen as in Proposition~\ref{expl_stat} and therefore 
$\sum_{k=1}^{\abs{D}-1} a_k^2\leq\norm{f}{2}^2$.
From the preceding analysis of the individual error 
we have for $\norm{f}{2}\leq1$ the right upper 
error bound. 
Now we consider $f=u_1$, where $\norm{u_1}{2}=1$. By applying \eqref{err_present} we have
\[
e_\pi(S_n,u_1)^2=\frac{1+\beta_1}{n(1-\beta_1)}-\frac{2\beta_1(1-\beta_1^n)}{n^2(1-\beta_1)^2}.
\] 
Thus the equality of \eqref{err_class_eq} is proven and by \eqref{mono_inc} the inequality is shown.
\end{proof}


In Corollary~\ref{prop_stat}
an explicit error bound with respect to $\norm{\cdot}{2}$ is shown. 
Notice that the first part of \eqref{err_class_eq} is an equality, which means 
that the integration error is completely known if 
the initial distribution is stationary.\\

Suppose that the distribution $\pi$ can be simulated directly, i.e. we can apply a Monte Carlo method with an i.i.d. sample.
Then an i.i.d. sequence $(X_n)_{n\in\N}$, where every $X_n$ is distributed with respect to $\pi$, 
is a Markov chain with transition matrix $S=(\,\pi(y)\,)_{x,y\in D}$ 
and initial distribution $\pi$. In this setting one has 
\[
e_\pi(S_n,f)^2=\frac{1}{n} \norm{f-S(f)}{2}^2.
\]
This corresponds to $\beta_i=0$ for all $i>0$.
In some artificial cases other Markov chain Monte Carlo methods can do better.
For example
if there is a Markov chain where $\beta_i<0$ and the target is to approximate $S(u_i)$ or if all eigenvalues are smaller than zero. 
A simple transition matrix which satisfies this eigenvalue condition is given by
\[
P=\begin{pmatrix}
		0 & \frac{1}{\abs{D}-1} & \cdots & \frac{1}{\abs{D}-1}\\
		\frac{1}{\abs{D}-1} & 0 & \ddots & \vdots\\
		\vdots 	& \ddots & \ddots  & \frac{1}{\abs{D}-1} \\
	\frac{1}{\abs{D}-1} &\cdots & \frac{1}{\abs{D}-1} & 0 
	\end{pmatrix},
\]
where $\pi(x)=1/\abs{D}$ for all $x\in D$, see \cite[Remark~3, p.~617]{eigen_neg}.
It turns out that $\beta_1=\cdots=\beta_{\abs{D}-1}=-\frac{1}{\abs{D}-1}$. 
For large $\abs{D}$ it is unfortunately not possible to construct a transition matrix where $\beta_1$ is close 
to $-1$. 
\begin{prop}
Let $P$ be 
an irreducible  
 transition matrix. 
Then
\[
\beta_1 \geq -\frac{1}{\abs{D}-1}.
\]
\end{prop} 
\proof
Since $\beta_0=1$ one has
\[
1+ \sum_{i=1}^{\abs{D}-1} \beta_i =\sum_{i=0}^{\abs{D}-1} \beta_i= \mbox{trace}(P)= \sum_{x\in D} p(x,x)\geq 0.
\]
Then
\[
-1\leq \sum_{i=1}^{\abs{D}-1} \beta_i \leq (\abs{D}-1) \beta_1. \sq
\]

The error estimates under the assumption that the initial distribution is a 
stationary one seem to be restrictive.
If we could sample $\pi$ directly we would approximate $S(f)$ by Monte Carlo with an i.i.d. sample.
However, even if it is possible 
it might happen that the sampling procedure is computationally expensive. 
It can be reasonable to generate only the initial state by sampling from $\pi$ and afterwards
run a Markov chain with stationary distribution $\pi$. 
Perfect sampling might be helpful for the construction of such direct sampling procedures, see
\cite{PrWi,Haeg}.\\

In the following we consider the case, where the initial distribution is not necessarily stationary.
Let $\nu$ be a distribution on $D$ and $k\in\N$.
Then we define
\[
d_k(x)= \sum_{y\in D} \frac{\nu(y)}{\pi(y)} (p^k(x,y)-\pi(y)) 
      = P^k (\frac{\nu}{\pi})(x)-1 = (P^k-S)(\frac{\nu}{\pi}-1)(x),\quad x\in D.
\]
If $P$ is reversible with respect to $\pi$, then we obtain
\[
\norm{d_k}{2}=\norm{\frac{\nu P^k}{\pi} -1}{2},\quad k\in\N,
\]
thus $d_k$ determines the difference between $\nu P^k$ and the stationary distribution $\pi$.
Additionally by the spectral representation of $P^k$ (see \eqref{P_k_spec}) one obtains
\begin{equation}  \label{l_k_spec}
	d_k(x)=\sum_{i=1}^{\abs{D}-1} \beta_i^k  \sum_{y\in D} u_i(y) \nu(y)\, u_i(x) 
	      =\sum_{i=1}^{\abs{D}-1} \beta_i^k \scalar{\frac{\nu}{\pi}}{u_i} u_i(x), \quad x\in D .	
\end{equation}
The next statement gives
a relation between $e_\nu(S_{n,n_0},f)$ and  $e_\pi(S_n,f)$. 
\begin{prop} \label{fin_con_lem}
Let $f\in \R^D$ and $g=f-S(f)$.
Let $(X_n)_{n\in\N}$ be a Markov chain with transition matrix $P$ and initial distribution $\nu$.
Let $P$ be reversible with respect to $\pi$. 
Then
\begin{align}  \label{fin_con}
e_\nu(S_{n,n_0},f)^2 =  e_\pi(S_n,f)^2
+ \frac{1}{n^2}\sum_{j=1}^{n} L_{j+n_0}(g^2)
+ \frac{2}{n^2} \sum_{j=1}^{n-1} \sum_{k=j+1}^n L_{j+n_0}(gP^{k-j}g),
\end{align}
where
\[
L_i(h)=\scalar{d_i}{h}=\scalar{(P^i-S)(\frac{\nu}{\pi}-1)}{h},
\quad h\in \R^D,\, i\in\N.
\]
\end{prop}
\begin{proof}
It is easy to see that
\begin{align*}
& \expect_{\nu,P} \abs{S(f)-S_{n,n_0}(f)}^2
=\frac{1}{n^2} \sum_{j=1}^n \sum_{i=1}^n \expect_{\nu,P} [g(X_{n_0+j})g(X_{n_0+i})]\\
&\underset{\eqref{err_help}}{=} \frac{1}{n^2} \sum_{j=1}^n \sum_{x\in D} P^{n_0+j}(g^2)(x)\, \nu(x)+ \frac{2}{n^2} \sum_{j=1}^{n-1} \sum_{k=j+1}^n \sum_{x\in D} P^{n_0+j}(g P^{k-j}g)(x)\,\nu(x).
\end{align*}
Recall that reversibility with respect to $\pi$ is equivalent to self-adjointness (s-a) of $P$.
For every function $h\in \R^D$ and $i\in\N$ the following calculation holds
\begin{align*}
  &	\sum_{x\in D} (P^i h)(x)\, \nu(x)  = \scalar{P^i h}{\frac{\nu}{\pi}} 
		= \scalar{P^i h }{1} + \scalar{P^i h }{\frac{\nu}{\pi}-1} \\
&	\underset{\text{(s-a)}}{=} \scalar{P^i h }{1} + \scalar{P^i (\frac{\nu}{\pi}-1) }{h}
		= \scalar{P^i h }{1} + \scalar{(P^i-S) (\frac{\nu}{\pi}-1) }{h}\\
	&\;	= \sum_{x\in D} (P^i h)(x)\, \pi(x)+\scalar{d_i}{h}.
\end{align*}
Formula \eqref{fin_con} is shown by using the previous calculation
for $h=g^2$ and $h=g P^{k-j} g$. 
\end{proof}

\begin{coro}  \label{connection_u_i}
Under the same assumptions as in Proposition~\ref{fin_con_lem} we obtain for 
$i\in \set{1,\dots,\abs{D}-1}$
that
\begin{equation*} 
e_\nu(S_{n,n_0},u_i)^2=\frac{1+\beta_i}{n(1-\beta_i)}-\frac{2\beta_i(1-\beta_i^n)}{n^2(1-\beta_i)^2}
								+\frac{1}{n^2} \sum_{j=1}^n \left( \frac{1+\beta_i-2\beta_i^{n-j+1}}{1-\beta_i} \right) L_{j+n_0}(u_i^2),
\end{equation*}
where
\begin{equation}  \label{L_k_u}
 L_k(u_i^2)= \sum_{l=1}^{\abs{D}-1} \beta_l^k \scalar{\frac{\nu}{\pi}}{u_l} \scalar{u_l}{u_i^2}
					= \sum_{l=1}^{\abs{D}-1} \beta_l^k  \scalar{u_l}{u_i^2} \sum_{x\in D} u_l(x)\,\nu(x). 
\end{equation}
\end{coro}
\begin{proof}
By substituting 
\[
e_\pi(S_n,u_i)^2 = \frac{1+\beta_i}{n(1-\beta_i)}-\frac{2\beta_i(1-\beta_i^n)}{n^2(1-\beta_i)^2},
\]
and
\begin{align*}
		\sum_{j=1}^{n-1} \sum_{k=j+1}^n L_{j+n_0}(u_iP^{k-j}u_i) 
&= 	\sum_{j=1}^{n-1} \sum_{k=j+1}^n \beta_i^{k-j} L_{j+n_0}(u_i^2) 
=		\sum_{j=1}^{n-1} \frac{ L_{j+n_0}(u_i^2) ( 
		  \beta_i-\beta_i^{n-j+1})}{1-\beta_i}
\end{align*}
into \eqref{fin_con} one obtains the error formula. 
The equality of $L_k(u_i^2)$ is an implication of \eqref{l_k_spec}.   
\end{proof}

Equation \eqref{fin_con} and the result of Corollary~\ref{connection_u_i} 
are still exact error formulas. 
To get an upper bound for the error, we estimate the functional $L_k(\cdot)$. 
This estimate depends on the speed of convergence of $\nu P^k$ to $\pi$.

\begin{lemma}
Let  $h\in\R^{D}$, $k\in\N$ and recall that $\beta=\max\set{\beta_1,\abs{\beta_{\abs{D}-1}}}$. Then
\begin{equation} \label{2_2_fin}
\abs{L_k(h)} \leq \beta ^k \norm{\frac{\nu}{\pi}-1}{2} \norm{h}{2} \leq \beta ^k \norm{\frac{\nu}{\pi}-1}{2} \sqrt{\norm{\frac{1}{\pi}}{\infty}} \norm{h}{1}.
\end{equation}
\end{lemma}
\begin{proof}
After applying the Cauchy-Schwarz inequality (CS) to $L_k(h)=\scalar{ d_k }{h}$ one obtains 
\[
  \abs{L_k(h)} 	\underset{\text{(CS)}}{\leq} \norm{d_k}{2}\norm{h}{2}
\leq \norm{P^k-S}{\ell_2\to \ell_2} \norm{\frac{\nu}{\pi}-1}{2} \norm{h}{2}.
\]
By Lemma~\ref{P_S_fin} the first inequality is proven and the rest is shown
by using $\norm{h}{2}\leq \sqrt{\norm{\frac{1}{\pi}}{\infty}} \norm{h}{1}$.
\end{proof}


The last lemma ensures an exponential decay of $L_k(\cdot)$ for increasing $k\in\N$. 
This fact is used to show that there exists a constant $C_{\nu,\pi,\beta}$, 
which is independent of $n$ and $n_0$, such that
\[
  \abs{e_\nu(S_{n,n_0},f)^2-e_\pi(S_n,f)^2 } \leq 
C_{\nu,\pi,\beta}\,\norm{f}{2}^2\, \frac{\beta^{n_0}}{n^2} .
\]
An immediate consequence of the inequality is an explicit error bound.
The following two lemmas imply such an inequality and provide $C_{\nu,\pi,\beta}$ explicitly.
\begin{lemma} \label{err_prop_fin_2}
Let $(X_n)_{n\in \N}$ be a Markov chain with transition matrix $P$ and initial distribution $\nu$.
Let $P$ be reversible with respect to $\pi$.
Let $f\in \R^D$ 
and
\begin{align*}
U(\beta,n)   & = \sum_{j=1}^n \beta^j + 2 \sum_{j=1}^{n-1} \sum_{k=j+1}^n \beta^k.
\end{align*}
Then
\begin{align} \label{err_l_2} 
  \abs{e_\nu(S_{n,n_0},f)^2-e_\pi(S_n,f)^2} \leq U(\beta,n)
        \sqrt{\norm{\frac{1}{\pi}}{\infty}}\norm{\frac{\nu}{\pi}-1}{2} 
	\norm{f}{2}^2 \frac{\beta^{n_0} }{n^2}.
\end{align}
\end{lemma}

\begin{proof}	
  Let $g=f-S(f)$. 
  The equation \eqref{fin_con} implies
	\begin{align*} \label{exact_error_fin_2}  
	\abs{e_\nu(S_{n,n_0},f)^2-e_\pi(S_n,f)^2} \leq
 \frac{1}{n^2}\sum_{j=1}^{n} \abs{ L_{j+n_0}(g^2)}
+ \frac{2}{n^2} \sum_{j=1}^{n-1} \sum_{k=j+1}^n \abs{ L_{j+n_0}(gP^{k-j}g)}.
\end{align*}	
Then by \eqref{2_2_fin} 
one gets
 \begin{align*}
 		\abs{L_{j+n_0}(g^2)} & 
 		\leq \beta^{j+n_0}  \sqrt{\norm{\frac{1}{\pi}}{\infty}} \norm{\frac{\nu}{\pi}-1}{2}\;   \norm{g}{2}^2,\\
 		\abs{L_{j+n_0}(gP^{k-j}g)}& 
 		\leq \beta^{j+n_0} \sqrt{\norm{\frac{1}{\pi}}{\infty}} \norm{\frac{\nu}{\pi}-1}{2}\;   \norm{g P^{k-j} g}{1}.
 \end{align*}
By the Cauchy-Schwarz inequality (CS) and $\norm{P^{k-j}}{\ell_2^0 \to \ell_2^0}= \beta^{k-j}$ it follows that
\[
\norm{gP^{k-j}g}{1} \underset{\text{(CS)}}{\leq} \norm{g}{2} \norm{P^{k-j}g}{2} 
\leq \norm{g}{2}^2 \norm{P^{k-j}}{\ell_2^0 \to \ell_2^0}
=\beta^{k-j} \norm{g}{2}^2.
\]
Let $\e_0=   \sqrt{\norm{\frac{1}{\pi}}{\infty}} \norm{\frac{\nu}{\pi}-1}{2} \beta^{n_0}$. Then
 \begin{align*}
 		\sum_{j=1}^n & \abs{L_{j+n_0}(g^2)} +
 		2\sum_{j=1}^{n-1} \sum_{k=j+1}^n \abs{L_{j+n_0}(gP^{k-j}g)}\\
 		&\leq \e_0 \norm{g}{2}^2  \sum_{j=1}^n   \beta^{j}
 		+  2 \e_0 \norm{g}{2}^2 \sum_{j=1}^{n-1}  \sum_{k=j+1}^n \beta^{k}\\
 		 &=  \e_0 \norm{g}{2}^2 \left(  \sum_{j=1}^n   \beta^{j} +  2 \sum_{j=1}^{n-1} \sum_{k=j+1}^n \beta^{k} \right)\\
 		 & 
		  \leq U(\beta, n) \cdot \e_0 \norm{f}{2}^2.
 \end{align*}
The last inequality follows from $\norm{f-S(f)}{2}\leq \norm{f}{2}$.  
\end{proof}
Note that one can also get a similar estimate as in \eqref{err_l_2} with respect to $\norm{f}{4}$
by using the first inequality of \eqref{2_2_fin} instead of the second one. 
In the resulting estimate the factor $\sqrt{\norm{\frac{1}{\pi}}{\infty}}$ does not appear.

Let us consider  
$U(\beta,n)$. If $\beta<1$, then the mapping $n\mapsto U(\beta,n)$ is bounded.
\begin{lemma} \label{lemma_U_fin}
Let $\beta<1$. For all $n\in \N$ we have
\[
U(\beta,n)\leq \frac{2}{(1-\beta)^2}. 
\] 
\end{lemma}
\proof
 By the infinite geometric series one obtains
\begin{align*}
U(\beta,n)
	&	\leq 	\sum_{j=1}^n \beta^j + \frac{2 \beta}{1-\beta}\sum_{j=1}^{n-1} \beta^{j}
			  \leq  \left(\frac{1+\beta}{1-\beta}\right)\sum_{j=1}^n \beta^j
		\leq \frac{2}{(1-\beta)^2}. \sq
\end{align*}

From Lemma~\ref{err_prop_fin_2} and Lemma~\ref{lemma_U_fin} it follows that
\[
e_\nu(S_{n,n_0},f)^2 \leq e_\pi(S_n,f)^2
		+  \frac{2 \beta^{n_0} \sqrt{\norm{\frac{1}{\pi}}{\infty}} 
 		\norm{\frac{\nu}{\pi}-1}{2} }{n^2(1-\beta)^2} \norm{f}{2}^2.
\]
If the initial distribution $\nu$ is $\pi$ then the 
error can be represented as in Proposition~\ref{expl_stat} and bounded as in Corollary~\ref{prop_stat}.

The next theorem summarizes the main result of this section.
\begin{theorem} \label{main_fin}
Let $f\in\R^{D}$ and let $(X_n)_{n\in\N}$ be a Markov chain 
with transition matrix $P$ and initial distribution $\nu$. Let $P$ be reversible 
with respect to $\pi$ and let $\beta<1$.
Then
\begin{equation} \label{l2_err_fin}
	e_\nu(S_{n,n_0},f)^2 \leq \frac{2}{n(1-\beta_1)} \norm{f}{2}^2 
			+  \frac{2 \beta^{n_0} \sqrt{\norm{\frac{1}{\pi}}{\infty}} 
			\norm{\frac{\nu}{\pi}-1}{2} }{n^2(1-\beta)^2} \norm{f}{2}^2.
\end{equation}
For $a_k=\scalar{f}{u_k}$ one has 
\begin{align}  \label{asymp_fin}
 \lim_{n\to \infty} n\cdot e_\nu(S_{n,n_0},f)^2 &=\lim_{n\to \infty} n\cdot e_\pi(S_{n},f)^2 
= \sum_{k=1}^{\abs{D}-1} a_k^2\, \frac{1+\beta_k}{1-\beta_k}.
 \end{align}
\end{theorem}
\proof
By Lemma~\ref{err_prop_fin_2}, Corollary~\ref{prop_stat} and Lemma~\ref{lemma_U_fin} the estimate 
of \eqref{l2_err_fin} is proven.
By Lemma~\ref{err_prop_fin_2} and Lemma~\ref{lemma_U_fin} the first equality of \eqref{asymp_fin} holds. Then, by Proposition~\ref{expl_stat}
\[
\lim_{n\to\infty} n\cdot e_\pi(S_n,f)^2 = \sum_{k=1}^{\abs{D}-1} a_k^2\, \frac{1+\beta_k}{1-\beta_k}. \sq
\]

\begin{remark}
The error bound \eqref{l2_err_fin} can be interpreted as follows:
The burn-in $n_0$ is necessary to eliminate the influence of the initial distribution $\nu$, 
while $n$ must be large to decrease $e_\pi(S_n,f)$.
Unfortunately the dependence of the initial distribution on the estimate is 
disillusioning for an extension to general state spaces, 
because of the factor $\sqrt{\norm{\frac{1}{\pi}}{\infty}}$. 
One can avoid this factor if one considers error bounds 
with respect to $\norm{f}{p}$ with $p>2$, see Section~\ref{err_bound_gen}.
\end{remark}

Another consequence of Lemma~\ref{err_prop_fin_2} and Lemma~\ref{lemma_U_fin} is 
the following result concerning the asymptotic error for $\norm{f}{2}\leq1$.

\begin{coro}  \label{asymp_err_coro}
  Under the same assumptions as in Theorem~\ref{main_fin} it follows that
  \[
    \lim_{n\to \infty} n\cdot \sup_{\norm{f}{2}\leq 1} e_{\nu}(S_{n,n_0},f)^2 
    = \frac{1+\beta_1}{1-\beta_1}
  \]
  and
  \[
    \lim_{n_0 \to \infty}  \sup_{\norm{f}{2}\leq 1} e_{\nu}(S_{n,n_0},f)^2 
    = 	\frac{1+\beta_1}{n(1-\beta_1)}-\frac{2\beta_1(1-\beta_1^n)}{n^2(1-\beta_1)^2}.
  \]
\end{coro}
\begin{proof}
 Let us define
  \[
    c_{n,n_0}=\frac{2\beta^{n_0} \sqrt{\norm{\frac{1}{\pi}}{\infty}} 
      \norm{\frac{\nu}{\pi}-1}{2}}{n^2(1-\beta)^2}.
  \]
  One has $\lim_{n\to \infty}n\cdot c_{n,n_0}=0$ and $\lim_{n_0\to \infty} c_{n,n_0}=0$.
  For $\norm{f}{2}\leq1$ we obtain by Lemma~\ref{err_prop_fin_2} and Lemma~\ref{lemma_U_fin} that
  \[
    \abs{e_\nu(S_{n,n_0},f)^2-e_\pi(S_{n},f)^2} \leq c_{n,n_0}.
  \] 
  Hence
\begin{equation}  \label{eq_low_b}
  \sup_{\norm{f}{2}\leq 1}e_\pi(S_{n},f)^2 - c_{n,n_0} 
  \leq \sup_{\norm{f}{2}\leq 1}  e_\nu(S_{n,n_0},f)^2 
\leq \sup_{\norm{f}{2}\leq 1} e_\pi(S_{n},f)^2 + c_{n,n_0}.
\end{equation}
By Corollary~\ref{prop_stat} 
we have
\[
\sup_{\norm{f}{2}\leq1}e_\pi(S_n,f)^2
= 	\frac{1+\beta_1}{n(1-\beta_1)}-\frac{2\beta_1(1-\beta_1^n)}{n^2(1-\beta_1)^2} .
\]
By taking the limits in \eqref{eq_low_b} the assertions are proven.
\end{proof}

\begin{remark}
The number
\[
\tau=\frac{1+\beta_1}{1-\beta_1}
\]
is called an \index{autocorrelation time}\emph{autocorrelation time} of $P$, see \cite{sokal,mathe1}.
If one could sample from $\pi$ then $\beta_1=0$ so that $\tau=1$.
Hence $\tau$
is the factor of computing time 
which quantifies the asymptotic difference of Markov chain Monte Carlo compared to 
Monte Carlo with an i.i.d. sample from the distribution $\pi$.   
\end{remark}
\begin{remark}  \label{rem_low_bound}
Observe that one obtains from \eqref{eq_low_b} a lower error bound for $S_{n,n_0}$. 
We have
\begin{equation*} \label{low_est_fin}
\frac{1+\beta_1}{n(1-\beta_1)}-\frac{2}{n^2(1-\beta_1)^2}  - c_{n,n_0} 
\leq \sup_{\norm{f}{2}\leq1} e_\nu(S_{n,n_0},f)^2 
\leq  \frac{2}{n(1-\beta_1)} + c_{n,n_0}
\end{equation*}
with $c_{n,n_0}$ defined as in the proof of Corollary~\ref{asymp_err_coro}.
For a reasonable burn-in of the Markov chain the error can be effectively approximated by these estimates. 
We apply these estimates to illustrating examples, see Section~\ref{toy_fin}.
Now let us discuss which burn-in is reasonable.
\end{remark}

\section{Burn-in} \label{burn_in_fin}

Assume that computer resources for $N$ steps of the Markov chain are available, i.e. $N=n+n_0$.
The goal is to choose the burn-in $n_0$ and the number $n$ such that the upper error bound is as small as possible. 
There is obviously a trade-off between the choice of $n$ and $n_0$.
In the next statement the error for an explicitly given burn-in is stated. 
\begin{theorem}
\label{main_coro}
Suppose that 
\[
n_0= \max\set{\left\lceil \frac{\log(\sqrt{\norm{\frac{1}{\pi}}{\infty}} \norm{\frac{\nu}{\pi}-1}{2})}{\log(\beta^{-1})}\right\rceil,0}.
\]
Then 
\[
\frac{1+\beta_1}{n(1-\beta_1)}- \frac{4}{n^2(1-\beta)^2}  \leq \sup_{\norm{f}{2}\leq1}e_\nu(S_{n,n_0},f)^2
		\leq \frac{2}{n(1-\beta_1)} 
				+				 
				 \frac{2}{n^2(1-\beta)^2}.
\]
\end{theorem}
\begin{proof}
The assertion follows from Theorem~\ref{main_fin} and Remark~\ref{rem_low_bound}. 
\end{proof}
Note that $\log(\beta^{-1})=(1-\beta)+\sum_{j=2}^\infty \frac{(1-\beta)^j}{j !}$ 
and $\log(\beta^{-1})\geq1-\beta$.
One might use this observation to estimate the suggested burn-in. 
The choice of the burn-in of Theorem~\ref{main_coro} is justified by the following.

Let us define
\[
  C=\sqrt{\norm{\frac{1}{\pi}}{\infty}} \norm{\frac{\nu}{\pi}-1}{2}
\]
and assume that $\beta_1 = \beta$. 
If the assumption does not hold we may estimate the error bound of 
Theorem~\ref{main_fin} by using $(1-\beta_1)^{-1}\leq (1-\beta)^{-1}$.
For $\norm{f}{2}\leq1$ we want to minimize the error estimate 
\[
\text{est}(n,n_0)=\sqrt{\frac{2}{n(1-\beta)}+\frac{2C \beta^{n_0}}{n^2(1-\beta)^2}} 
\quad
\text{under the constraint that} 
\quad N=n+n_0.
\]

\begin{lemma}
 For $\eta>0$ let
 \begin{align}
C 		&	
			> \left(\frac{\log(\beta^{-1})}{1-\beta}\right)^{1/\eta}, \label{reasonable_C}\\
N		&	> (1+\eta)\,\frac{\log(C)}{\log(\beta^{-1})}+2[\log(\beta^{-1})-(1-\beta)]^{-1}. \label{reasonable_N}
\end{align}
Then there exists an 
\[
 n_{\text{opt}}\in\left[\frac{\log(C)}{\log(\beta^{-1})},(1+\eta)\frac{\log(C)}{\log(\beta^{-1})}\right],
\]
which minimizes the mapping  $n_0\mapsto \text{\rm est}(N-n_0,n_0)$.
\end{lemma}
If $\eta=10^{-3}$, then \eqref{reasonable_C} implies for $\beta=0.99$ that $C>152$ and 
for $C=10^{30}$ that $\beta>0.87$. 
Hence the assumptions are not restrictive, 
since $\beta$ is usually close to $1$, 
$C$ is 
large\footnote{The constant $C$ might depend exponentially on additional parameters, 
see the example ``Random walk on the hypercube'' in Section~\ref{toy_fin} or see Section~\ref{appl}.}  
and the computational resources $N$ should be sufficiently large.
\begin{proof}
 Let  
\[
 a=N-(1+\eta)\frac{\log(C)}{\log(\beta^{-1})} \quad \text{and}\quad b=N-\frac{\log(C)}{\log(\beta^{-1})}.
\] 
 Note that \eqref{reasonable_N} gives that $b>a>0$.
 It is enough to show that there exists an $m_{\text{opt}}\in[a,b]$ which minimizes $n \mapsto \text{est}^2(n)$
 given by
 \[
  \text{est}^2(n)= (\text{est}(n,N-n))^2=\frac{2}{n(1-\beta)}+\frac{2C \beta^{N-n}}{n^2(1-\beta)^2}.
 \]
 We have
\[
\text{est}^2(n)'= \frac{d}{dn}\,  \text{est}^2(n) =\frac{2}{n^2(1-\beta)}
	\left[ \frac{C\beta^{N-n}}{(1-\beta)}\left( \log(\beta^{-1}) - \frac{2}{n} \right)-1 \right].
\]
We will show for any $\tilde{a}\leq a$ and $\tilde{b} \geq b$ that
\[
 \text{est}^2(\tilde{a})'<0 \quad \text{and} \quad \text{est}^2(\tilde{b})'>0,
\]
consequently there exists an $m_\text{opt} \in [a,b]$ which minimizes $\text{est}^2(n)$. 
Let $\tilde{b}\geq b$. Then the inequality $\text{est}^2(\tilde{b})'>0$ follows by 
\eqref{reasonable_N} and
\begin{align*}
& \qquad\quad\; 
			N>\frac{\log(C)}{\log(\beta^{-1})}
			\frac{\left( \frac{2}{\log(C)} + \left( 1- \frac{1-\beta}{\log(\beta^{-1})} \right)  \right)}
			{\left( 1- \frac{1-\beta}{\log(\beta^{-1})} \right) }
			\\
& \Longleftrightarrow \quad 	
	N>\frac{(1-\frac{1-\beta}{\log(\beta^{-1})})\log(C)+2}{\log(\beta^{-1})-(1-\beta)}\\
& \Longleftrightarrow \quad	
				N(\log(\beta^{-1})-(1-\beta))+\log(C)\left(\frac{1-\beta}{\log(\beta^{-1})}-1\right)>2\\
&\Longleftrightarrow \quad		\left( N-\frac{\log(C)}{\log(\beta^{-1})} \right)\log(\beta^{-1})
												-			\left( N-\frac{\log(C)}{\log(\beta^{-1})} \right)(1-\beta)>2				\\
&\Longleftrightarrow \quad b \log(\beta^{-1}) - b (1-\beta)>2\\												
& \Longleftrightarrow \quad 
 \log(\beta^{-1})-\frac{2}{b}>(1-\beta) \\
& \Longleftrightarrow \quad	
	\frac{1}{1-\beta}\left( \log(\beta^{-1}) - \frac{2}{b} \right)-1
=	\frac{C\beta^{N-b}}{1-\beta}\left( \log(\beta^{-1}) - \frac{2}{b} \right)-1  
   > 0\\
& \;\Longrightarrow \quad 
\frac{C\beta^{N-\tilde{b}}}{1-\beta}\left( \log(\beta^{-1}) - \frac{2}{\tilde{b}} \right)-1  
   > 0.
\end{align*}
On the other hand for $\tilde{a}\leq a$ we obtain $\text{est}^2(\tilde{a})'<0$. 
This is shown by the following calculation. By
\eqref{reasonable_C} one has
\begin{align*}
& \qquad\quad\; C^{\eta} \underset{\eqref{reasonable_C}}{>} 
  \frac{\log(\beta^{-1})}{(1-\beta)}\\
& \Longleftrightarrow \quad \log(\beta^{-1})-(1-\beta) C^{\eta} < 0 \\
& \;\Longrightarrow \quad \log(\beta^{-1})-(1-\beta) C^{\eta} < \frac{2}{a}\\
& \Longleftrightarrow \quad \log(\beta^{-1})- \frac{2}{a} < (1-\beta)C^{\eta} \\
& \Longleftrightarrow \quad 
	\frac{C^{-\eta}}{(1-\beta)}\left( \log(\beta^{-1})-\frac{2}{a} \right)-1=
	\frac{C\beta^{N-a}}{(1-\beta)}\left( \log(\beta^{-1})-\frac{2}{a} \right)-1<0\\
& \;\Longrightarrow \quad
 \frac{C\beta^{N-\tilde{a}}}{(1-\beta)}\left( \log(\beta^{-1})-\frac{2}{\tilde{a}} \right)-1<0.
\end{align*}
Altogether this implies that there is an 
\[
 n_\text{opt}\in\left 
[\frac{\log(C)}{\log(\beta^{-1})},(1+\eta)\frac{\log(C)}{\log(\beta^{-1})}\right]
\]
which minimizes the mapping $n_0\mapsto \text{est}(N-n_0,n_0)$.
\end{proof}


If an error of at most $\e \in (0,1)$ is desired, then the suggested choice of the burn-in $n_0$ 
is independent of the precision $\e$, we choose 
\[
n_0= \max\set{\left\lceil \frac{\log(\sqrt{\norm{\frac{1}{\pi}}{\infty}} \norm{\frac{\nu}{\pi}-1}{2})}{\log(\beta^{-1})}\right\rceil,0}
\]
and
\[
n\geq \frac{ 1+\sqrt{1+4\e^2}}{(1-\beta)\e^2} \qquad \text{to achieve}\qquad e_\nu(S_{n,n_0},f) \leq \e.
\]  

\section{Examples}  \label{toy_fin}
The goal is to compare the upper bounds of Theorem~\ref{main_fin} and Theorem~\ref{main_coro} 
with the exact error for a given function $f\in  \R^D$. 
It is not known which $f$ with $\norm{f}{2}\leq 1$ maximizes
$
e_\nu(S_{n,n_0},f)^2.
$ 
But by 
Corollary~\ref{asymp_err_coro} one has
\[
\lim_{n_0\to \infty} \sup_{\norm{f}{2}\leq 1} e_\nu(S_{n,n_0},f)^2 = e_\pi(S_n,u_1)^2,
\]
where $u_1$ is the eigenfunction corresponding to $\beta_1$. 
This motivates the study of
the individual error for $u_1$,
which gives the maximal error for integrands $f$ with $\norm{f}{2}\leq 1$ if $n_0$ goes to infinity. 
In this section illustrating examples are considered, 
where the eigenvalues and the eigenfunctions are available. 
The Markov chains are very well studied in the literature, see
\cite{meise,saloff_coste, stroock_book,bassetti_dia, peres}. 
\subsection*{Random walk on a circle}
Let $T \geq 3$ be an odd natural number. Let $D=\Z_T$ be the underlying state space, 
where $\Z_T= \Z \mod T$ denotes the cyclic group of order $T$. 
The $T\times T$ transition matrix of the random walk is determined by
\[
p(x,y)=	\begin{cases}
		\frac{1}{2}, & \quad y=x\pm 1 \mod T, \\
		0, 	     & \quad \text{otherwise}.
	\end{cases}			
\]
The transition matrix is reversible with respect to the uniform distribution 
given by $\pi(x)=1/T$ for $x\in D$. 
Since $T$ is an odd number we obtain that the transition matrix is aperiodic, 
for even $T$ it would be periodic. 
The eigenvalues of the transition matrix are
\[
  \beta_0=1, \quad 
  \beta_{2j-1}=\beta_{2j}=\cos\left(\frac{2\pi j}{T}\right),\qquad j=1,\dots,\frac{T-1}{2},
\]
and the orthonormal eigenfunctions $\set{ u_0,u_1,\dots,u_{T-1}}$ are
\[
  u_0(x)=1, \quad u_{2j-1}(x)=\sqrt{2} \cos\left(2\pi\, \frac{jx}{T}\right),
	    \quad u_{2j}(x)=\sqrt{2} \sin\left(2\pi\, \frac{jx}{T}\right),
\]
where $j=1,\dots,\frac{T-1}{2}$ and $x\in D$.
Clearly $\beta=\abs{\beta_{T-1} }=\cos(\frac{\pi}{T})$, thus $\beta\neq \beta_1$.\\ 

Let us consider $f=u_1$. 
The initial distribution is chosen  
as $\nu=\delta_0$
, so that the initial state is $0\in D$.
By $(u_1)^2=u_0+\frac{1}{\sqrt{2}} u_3$ it is
\[
\scalar{u_i}{(u_1)^2}=\scalar{u_i}{u_0}+\frac{1}{\sqrt{2}} \scalar{u_i}{u_3} 
			= \begin{cases}
			1, & \quad i=0,\\
		        \frac{1}{\sqrt{2}}, &\quad i=3,\\
			0, & \quad \text{otherwise}.
			\end{cases}
\]
Hence by \eqref{L_k_u} we obtain
\begin{align*}
L_k((u_1)^2)
& = \sum_{i=1}^{T-1} \beta_i^k 
	\scalar{u_i}{(u_1)^2} u_i(0) 
=\beta_3^k.
\end{align*}
Additionally with $\beta_1=\cos(\frac{2\pi}{T})$ it is
\[
e_\pi(S_n,u_1)^2 
= \frac{1+\cos(\frac{2\pi}{T})}{n(1-\cos(\frac{2\pi}{T}))}
  -\frac{2\cos(\frac{2\pi}{T})(1-\cos^n(\frac{2\pi}{T}))}{n^2(1-\cos(\frac{2\pi}{T}))^2}.
\]
The exact error is determined by Corollary~\ref{connection_u_i} with 
 $\beta_3=\cos(\frac{4\pi}{T})$ so that 
\begin{align}  \label{circle_exact}
e_\nu(S_{n,n_0},u_1) & = \Bigl(  e_\pi(S_n,u_1)^2 \notag \\								
		& \qquad +\frac{1}{n^2} \sum_{j=1}^n \Bigl( \frac{1+\cos(\frac{2\pi}{T})-2\cos^{n-j+1}(\frac{2\pi}{T})}{1-\cos(\frac{2\pi}{T})} \Bigr) 
							\cos^{j+n_0}(\frac{4\pi}{T})) \Bigr)^{1/2}.
\end{align}
We apply Theorem~\ref{main_coro} to get a lower error bound and \eqref{l2_err_fin} of
 Theorem~\ref{main_fin} to get an upper error bound,
 since $\beta\neq \beta_1$.
Hence the burn-in is chosen as suggested in Theorem~\ref{main_coro}, i.e.    
\[
n_0=
		\left \lceil 
		\frac{1}{2}\frac{\log(T^2-T)}{\log(\cos^{-1}(\frac{\pi}{T}))}
		\right \rceil .
\]
Then 
\begin{equation} \label{circle_ub}
	e_\nu(S_{n,n_0},u_1)
	\leq \left( 
      \frac{2}{n(1-\cos(\frac{2\pi}{T}))}+\frac{2}{n^2(1-\cos(\frac{\pi}{T}))^2}
 	  \right)^{1/2}
\end{equation}
and
\begin{equation} \label{circle_lb}
   \left(
\frac{1+\cos(\frac{2\pi}{T})}{n(1-\cos(\frac{2\pi}{T}))}- \frac{4}{n^2(1-\cos(\frac{\pi}{T}))^2}
    \right)^{1/2} 
    \leq e_\nu(S_{n,n_0},u_1).
\end{equation}
We have an explicit exact error formula \eqref{circle_exact}, 
a lower error bound \eqref{circle_lb} and an upper error bound \eqref{circle_ub}. 

\begin{figure}[htb] 
  \begin{center}
    \includegraphics[height=9cm]{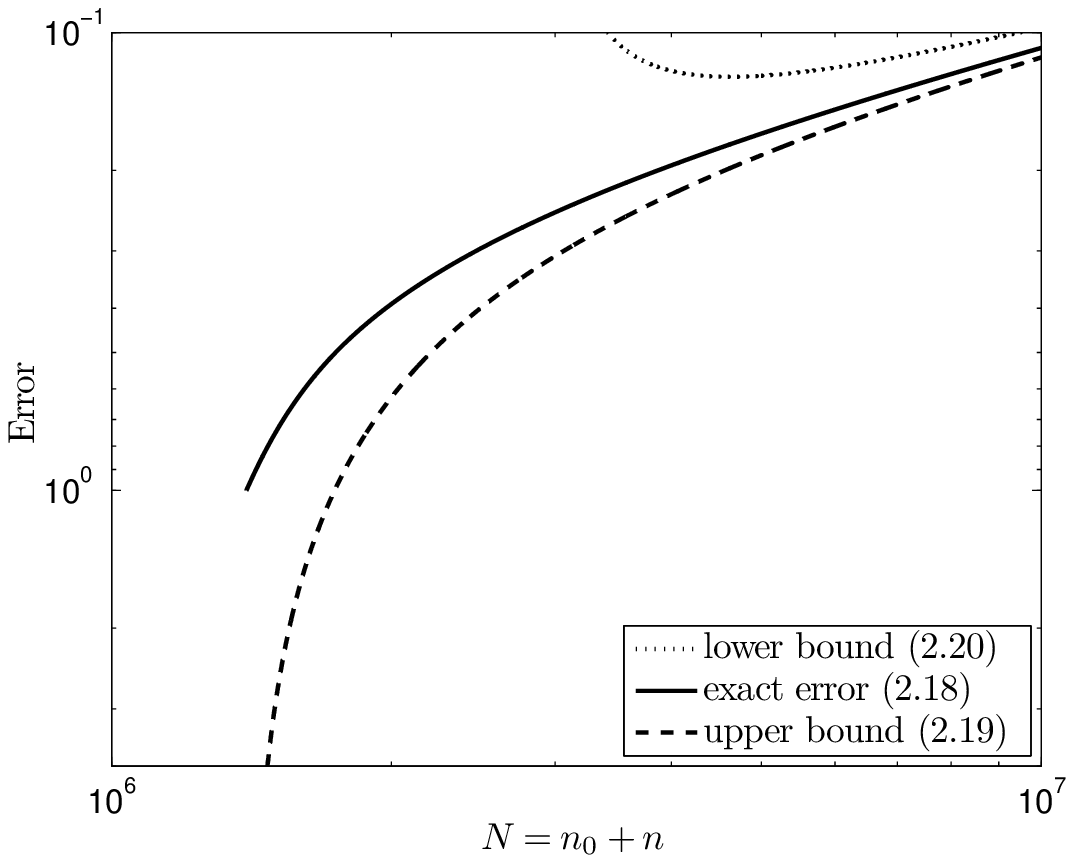}
    \caption{Random walk on a circle: Exact error and error bounds,
$T=999$ and $n_0=\left \lceil \frac{1}{2}\frac{\log(T^2-T)}{\log(\cos^{-1}(\frac{\pi}{T}))}\right \rceil $ = 1396699.}
    \label{n_opt_low_up_p999}
    \end{center}  
\end{figure} 

In Figure~\ref{n_opt_low_up_p999} the different bounds 
of \eqref{circle_ub}, \eqref{circle_lb} and the exact error of \eqref{circle_exact} are plotted for $T=999$. 
The curves start at $N = n_0$, since
the computational resources must be larger than the burn-in $n_0=1396699$. 
The lower error bound gives a non-trivial estimate if $N \geq n_0 + 1617911 = 3014610$, 
since for $n\geq \frac{4(1-\beta_1)}{(1+\beta_1)(1-\beta)^2}=1617911$ one obtains a lower bound larger as zero.
\subsection*{Random walk on the hypercube}
Let $d$ be a natural number.
Let $D=\set{0,1}^d$ be the state space and $\abs{\tilde{x}}=\sum_{i=1}^d \abs{\tilde{x}_i}$ for $\tilde{x}\in \set{-1,0,1}^d$.
The $2^d \times 2^d $ transition matrix is given by
\[
p(x,y)=
\begin{cases}
	\frac{1}{2}, 	& 	\quad x=y,\\
	\frac{1}{2d}, 	&	\quad \abs{x-y}=1,\\
	0,		&	\quad \text{otherwise}.
\end{cases}
\] 
The transition matrix is reversible with respect to $\pi(x)=2^{-d}$ for $x\in D$.
Furthermore, it is aperiodic and irreducible.
We use a different notation for the index of the eigenvalues and orthonormal eigenfunctions, 
for $z\in \set{0,1}^d$ one has 
 \[
\beta_z = 1-\frac{\abs{z}}{d} \quad \text{ and } \quad u_z (x) = (-1)^{\sum_{i=1}^d z_i x_i}, \quad x \in D.
\]
Set $[0]=(0,\dots,0)$ and set $[1]=(1,0,\dots,0)$ so that
\begin{align*}
 \beta_{[0]} &=1, \qquad\qquad \text{ and} \qquad u_{[0]}(x)=1,\quad x\in D,\\
 \beta_{[1]} &= 1- \frac{1}{d}, \qquad \text{ and } \qquad  u_{[1]}(x)=(-1)^{x_1},\quad x\in D.
\end{align*}

Obviously for all indizes $z\in \set{0,1}^d$ the eigenvalue $\beta_z \geq 0$ so that $\beta_{[1]}=\beta$.\\

Let us choose the initial state of the Markov chain deterministically in $(0,\dots,0)\in D$, i.e.
$\nu=\delta_{[0]}$. 
By $(u_{[1]})^2 = u_{[0]}$ one has for index $z\in\set{0,1}^d$ that
\[
\scalar{u_z}{(u_{[1]})^2}=\begin{cases}
				1, & \quad z=[0],\\ 
				0,	& \quad \text{otherwise}.
		      \end{cases}
\]
This implies
\[
L_k((u_{[1]})^2)=0,\quad k\in\N.
\]
The error of $S_n$, if the initial state is chosen by $\pi$, obeys
\[
e_\pi(S_n,u_{[1]})^2 
= \frac{2d-1}{n}-\frac{2(d^2-d)}{n^2} 
				\left( 1-\left(1-\frac{1}{d}\right)^n \right). 
\]
Then by Corollary~\ref{connection_u_i} it is
\begin{equation} \label{hyper_exact}
	e_\nu(S_{n,n_0},u_{[1]})=e_\pi(S_n,u_{[1]}).
\end{equation}
The burn-in and the error bounds are determined by Theorem~\ref{main_coro}. One obtains 
\[
n_0=\left\lceil \frac{1}{2} \frac{\log(2^{2d}-2^d)}{\log(1-\frac{1}{d})^{-1}}  \right\rceil
\]
such that 
\begin{equation}  \label{hyper_u}
 	e_\nu (S_{n,n_0},u_{[1]} ) \leq 
	\sqrt{\frac{2d}{n} + \frac{2d^2}{n^2}},
\end{equation}
and
\begin{equation}  \label{hyper_l}
    \sqrt{\frac{2d-1}{n} - \frac{4d^2}{n^2}}
\leq e_\nu (S_{n,n_0},u_{[1]} ).
\end{equation}
In Figure~\ref{n_opt_low_up_d50} for $d=50$ the exact error \eqref{hyper_exact}, 
the upper error bound \eqref{hyper_u} and the lower error bound \eqref{hyper_l} are plotted. 
It can be seen that after the burn-in the curves are close to each other. 
\begin{figure}[htb] 
  \begin{center}
    \includegraphics[height=9cm]{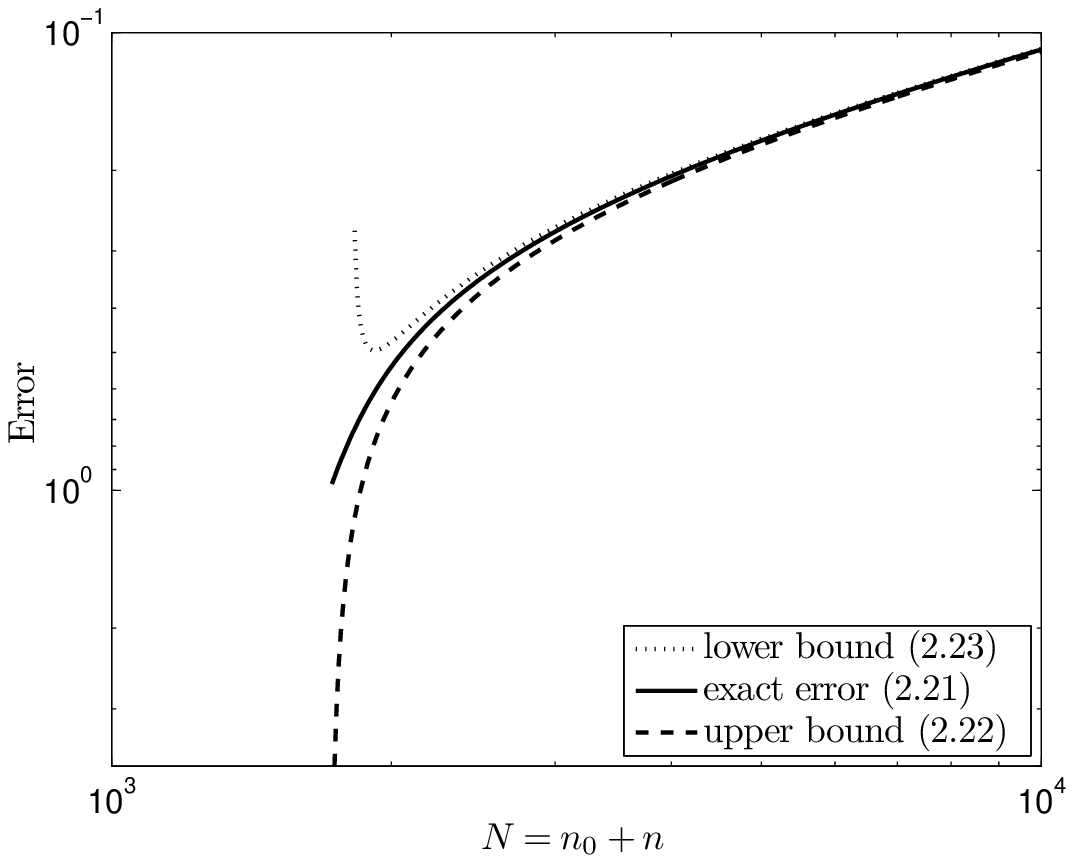}
    \caption{Random walk on the hypercube: 
    Exact error and error bounds, 
    $d=50$ and $n_0=\left\lceil \frac{1}{2} \frac{\log(2^{2d}-2^d)}{\log(1-\frac{1}{d})^{-1}}  \right\rceil
 = 1716.$}
    \label{n_opt_low_up_d50}
    \end{center}  
\end{figure} 
The error bounds are polynomial in $d$
which is of the magnitude of $\log(\abs{D})$.
%
%
\subsection*{Random walk on the star}
Let $T\geq 2$ be an even natural number. Let the state space $D=\set{0,1,\dots,T}$. 
The $(T+1)\times(T+1)$ transition matrix is given by
\[
p(x,y)=	\begin{cases}
		\theta, 	& \quad x=y=0,\\
		\frac{1-\theta}{T},	& \quad x=0,\;y \in D\setminus\set{0},\\
		1,				&	\quad x\in D\setminus\set{0},\;y=0,\\
		0, 		& \quad \text{otherwise},
	\end{cases}
\]
with a parameter $\theta \in (0,1)$.
The transition graph is star shaped since every state is connected solely with the center $0$.
The transition matrix is reversible with respect to $\pi$, for $x\in D$ given by
\[
\pi(x)=	\begin{cases}
	\frac{1}{2-\theta},		&  	\quad x=0,\\
	\frac{1-\theta}{T(2-\theta)},	&	\quad \text{otherwise}.
	\end{cases}
\]
One obtains $\beta_0=1$, $\beta_{T}  =\theta-1$ 
and for $x\in D$ one has
\[
u_0(x)=1, \qquad
    u_T(x)= \sqrt{1-\theta}
	\begin{cases}	
	\;\;\,1, & \quad x=0,\\
		\frac{1}{\theta-1}, &\quad \text{otherwise}.
	\end{cases}
\]
The eigenvalue $\beta_i=0$ for 
$i\in\set{1,\dots, T-1}$
is of multiplicity $T-1$. 
Without loss of generality we may assume that for any $x\in D$ one has
\[
u_1(x)=\begin{cases}
	 \quad  0,& \quad x=0,\\
	\;\;\,\sqrt{\frac{2-\theta}{1-\theta}},& \quad x=1,\dots,T/2,\\
	-\sqrt{\frac{2-\theta}{1-\theta}},& \quad x=T/2+1,\dots,T.
       \end{cases}
\]
The remaining eigenvectors $u_2,\dots,u_{T-1}$ are arbitrarily chosen such that we get an orthonormal basis
$\set{u_0,u_1,\dots,u_T}$.
One has an aperiodic and irreducible transition matrix where $\beta_1=0$ 
and $\beta=\max\set{ \beta_1,\abs{\beta_T} } = 1-\theta$.
We consider the error for $f=u_1$.
The initial state is given as the center of the star, i.e. $0$. 
Then $\nu=\delta_0$.
From $(u_1)^2= u_0 - \frac{1}{\sqrt{1-\theta}}\, u_T$
one gets
\[
\scalar{u_i}{(u_1)^2} = \begin{cases}
			 \quad 1, & \quad i=0,\\
			-\frac{1}{\sqrt{1-\theta}},	&	\quad i=T,\\
			\quad 0, & \quad \text{otherwise}.
			\end{cases}
\]
By \eqref{L_k_u} this implies
\begin{align*}
L_k((u_1)^2) &  
= \sum_{i=1}^T \beta_i^k \scalar{u_i}{(u_1)^2} u_i(0)
= -\beta_T^k 
= -(\theta-1)^{k}.
\end{align*}
The error where the Markov chain is initialized by the stationary distribution obeys
\[
e_\pi(S_n,u_1)^2 = \frac{1}{n}.
\]
Then by Corollary~\ref{connection_u_i} it follows that
\begin{align}
e_\nu( S_{n,n_0}, u_1)
\label{star_exact}
 = \left(\frac{1}{n}	- 
   \frac{(\theta-1)^{n_0+1}\left(  (\theta-1)^{n}-1 \right)}
    {(\theta-2)n^2}\right)^{1/2}.																					
\end{align}	
Recall that $\beta_1\neq \beta$. 
However, we only use the error bounds of Theorem~\ref{main_coro}.
The burn-in is chosen as 
\[
n_0=\left\lceil \frac{\log((2-\theta)T)}{2\log(1-\theta)^{-1}} \right\rceil.
\]
%
%
Then the upper bound is
\begin{equation} \label{star_u}
 e_{\nu} (S_{n,n_0}, u_1) \leq 
  \sqrt{\frac{2}{\theta n} + \frac{2}{\theta^2 n^2}},
\end{equation}
and the lower bound is given as 
\begin{equation} \label{star_l}
  \sqrt{    \frac{1}{n}-\frac{4}{\theta^2 n^2} }
\leq e_{\nu} (S_{n,n_0}, u_1).
\end{equation}
In Figure~\ref{n_opt_low_up_start_T100000_theta01} 
for $\theta=0.1$ and $T=10^5$ the exact error \eqref{star_exact}, 
the upper error bound \eqref{star_u} and the lower bound \eqref{star_l} are plotted. 
\begin{figure}[htb] 
  \begin{center}
    \includegraphics[height=9cm]{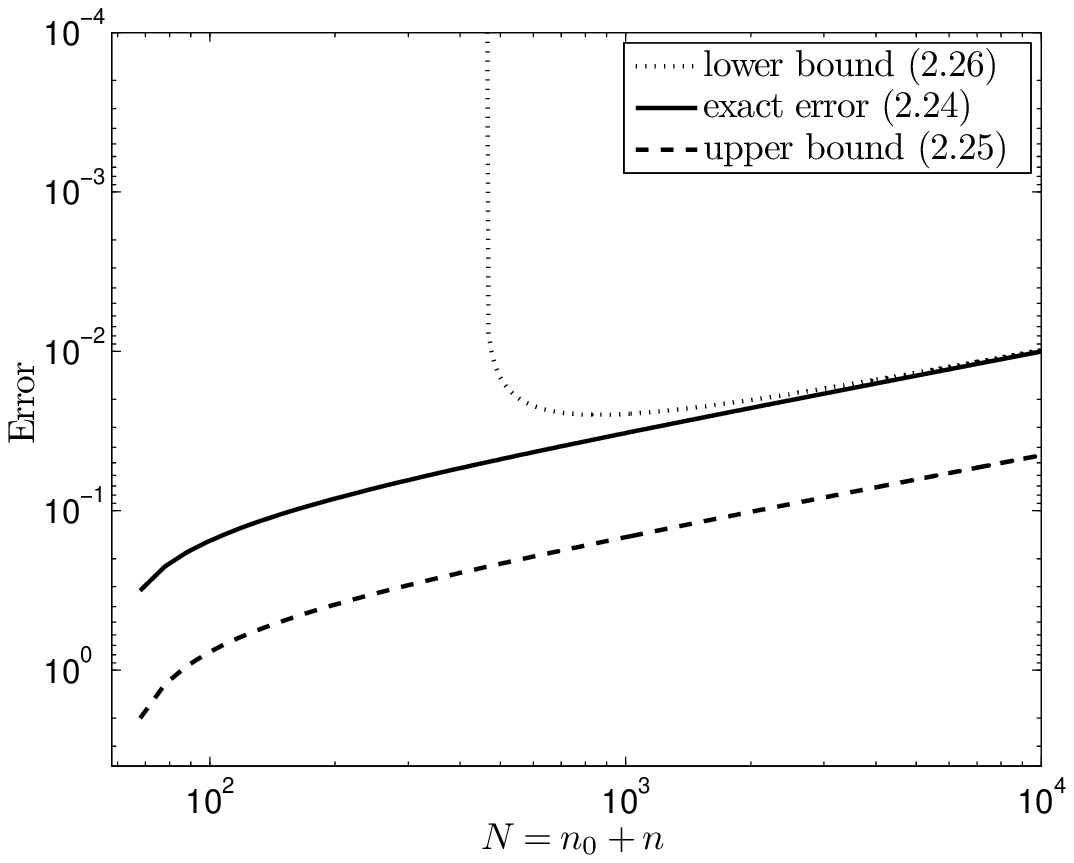}
    \caption{Random walk on the star: Exact error and error bounds, 
    	$\theta=0.1$, $T=10^5$ and
    	$n_0=\left\lceil \frac{\log((2-\theta)T)}{2\log(1-\theta)^{-1}} \right\rceil=58$.}
    \label{n_opt_low_up_start_T100000_theta01}
    \end{center}  
\end{figure}
For $n\geq \frac{4}{\theta^2}$ we get a non-trivial estimate by the lower bound. 
The upper error bound is shifted down since $\beta \neq \beta_1$. 
One could improve this by using \eqref{l2_err_fin} of Theorem~\ref{main_fin} directly. 
In the present setting one looses asymptotically a factor of $\sqrt{2/\theta}$.\\
%
%
%

Let us summarize the important facts of this section.
The error was considered for the eigenfunction $u_1$
corresponding to $\beta_1$. If $n_0$ goes to infinity, then $u_1$
is the function which maximizes the error 
for integrands $f$ with $\norm{f}{2}\leq1$. 
The bound of Theorem~\ref{main_coro} applied in this setting gives tight results if $\beta_1=\beta$. 
Otherwise Theorem~\ref{main_fin} achieves the right asymptotic coefficient if $\beta_1$ and $\beta$ are known. 
For the considered examples one has the eigenvalues and the eigenfunctions explicitly.
In applications it is usually difficult to estimate $\beta_1$ or $\beta$, 
but there are different auxiliary tools, e.g. canonical path technique, 
conductance (see \cite{jerrum} and \cite{dia_stroock_eigen}), log-Sobolev inequalities and path coupling,
see \cite{peres}.  
However, if the eigenvalues $\beta_1$ and $\beta_{\abs{D}-1}$ are available, then the error can be 
approximated by the lower and upper bound.

\section{Notes and remarks}  \label{notes_finite}
Let us comment how the results fit into the published literature. 
An elementary and powerful technique how to bound the error for $S_{n,n_0}$ or $S_n$ is based on Doeblin's theory,
see \cite[pp.~27]{stroock_book}. 
Let $A_k=(a_k(x,y))_{x,y\in D}$ be the $k$th Ces{\`a}ro sum given by
\[
A_k = \frac{1}{k} \sum_{j=0}^{k-1} P^j. 
\]
Assume that
\begin{equation} \label{doeblin_cond}
	 \exists M \in \N\setminus\set{1},\, y_0 \in D \;\text{and}\; \gamma>0 \quad\text{such that}\quad \forall x\in D:\quad
 a_M(x,y_0) \geq \gamma.
\end{equation}
Then for any $n_0$ the error obeys
\[
e_\nu(S_{n,n_0},f)^2 \leq \frac{8(M-1)}{n \gamma} \norm{f}{\infty}^2.
\]
Condition \eqref{doeblin_cond} states that there is a state $y_0$ where 
the expected value of visiting it, in average, until $M$ 
from every other state is uniformly bounded from below by rate $\gamma$. 
If the transition matrix is irreducible then there exists an $M$ 
such that $A_M > 0$ and one has that \eqref{doeblin_cond} is satisfied (see for example \cite[Lemma~7.3, p.~50]{behrends}). 
It is difficult to obtain $\gamma$ and $M$. 
Let us state a toy example where one can compute $\gamma$ and 
$M$ explicitely. 
Let $D=\set{0,1}^d$. 
We consider a Markov chain, which independently samples with respect to $\pi$, 
with $\pi(x)=2^{-d}$ for $x\in D$. This is Monte Carlo with an i.i.d. 
sample. Consequently we get as best possible parameters $\gamma=2^{-d-1}$ and $M=2$. 
The error estimate behaves exponentially bad in terms of $d$. 
In contrast, the estimate of Theorem~\ref{main_fin} is independent of $d$.
In general, even if one can get $\gamma$ and $M$, then these constants are often exponentially bad in terms of some other parameters.  
Usually $\gamma$ is close to zero and $M$ is huge.
However, with this bound even the periodic case is covered and reversibility is not necessary. 
But on the other hand 
the optimal coefficient $\frac{1+\beta_1}{1-\beta_1}$ of the leading term of 
Corollary~\ref{prop_stat} is not reached and the burn-in cannot be used to tune the algorithm.\\


The approach to use the spectral representation of reversible transition matrices is not new.
In \cite{bassetti_dia} the result of Proposition~\ref{expl_stat} is presented.
By the same arguments a slightly worse bound is shown in \cite[Proposition~4.1, p.~40]{aldous}. It applies if $\beta_1 \geq 0$ and gives 
\begin{equation}  \label{aldous_stat}
e_\pi(S_n,f)^2 \leq \frac{2}{n(1-\beta_1)}\norm{f}{2}^2 + \frac{2 \exp\set{-n(1-\beta_1)}  }{n^2(1-\beta_1)^2 }\norm{f}{2}^2.
\end{equation}
Furthermore if the initial distribution $\nu$ is not the stationary one, a different algorithm is considered.
Namely, the burn-in $n_0^*$ is randomly chosen, independent of $(X_n)_{n\in\N}$, 
by the Poisson distribution with parameter $n_0$, and
\[
S^*_{n,n_0} (f) = \frac{1}{n} \sum_{j=1}^n f(X_{j+n_0^*}). 
\]
Then it is proven in \cite[Proposition~4.2, p.~41]{aldous} that
\[
e_\nu(S^*_{n,n_0},f)^2 \leq e_\pi(S_n,f)^2 \left( 1 +  \norm{\frac{1}{\pi}}{\infty} \exp\set{-n_0(1-\beta_1)}  \right).
\]
This bound applies also for periodic Markov chains and after applying \eqref{aldous_stat} 
it gives an estimate with respect to $\norm{\cdot}{2}$. 
The optimal coefficient $\frac{1+\beta_1}{1-\beta_1}$ of the leading term, see Corollary~\ref{asymp_err_coro}, is not reached, 
also if Corollary~\ref{prop_stat} instead of \eqref{aldous_stat} is applied. 
The burn-in $n_0^*$ is randomly chosen rather than deterministically, since then one can translate the discrete time Markov chain into a continuous time Markov chain and avoids discussions of negative eigenvalues. 
This technique is similar to the idea of considering a lazy Markov chain. \\

In \cite{niemiro_poka} an explicit error bound is published which holds
also for non-reversible Markov chains with an absolute $\ell_2$-spectral gap, 
i.e. $\beta=\norm{P}{\ell_2^0 \to \ell_2^0} < 1$.
In the proof of the error bound the multiplicative reversibilization $PP^*$ of $P$ is used, 
where $P^*$ is the adjoint operator of $P$ 
acting on $\ell_2$.
It follows from \cite[Corollary~4.2, p.~320]{niemiro_poka} that
\[
e_\nu(S_{n,n_0},f)^2 \leq \frac{1+\beta}{n(1-\beta)} \norm{f}{2}^2 + \frac{2 \beta}{(1-\beta )^2 n^2} \norm{f}{2}^2 
	+ \frac{  2(1+\beta)   
	\beta^{n_0} \norm{\frac{\nu}{\pi}-1}{2}}{ (1-\beta ) n^2} \norm{f}{\infty} \norm{f}{2}.
\]
One obtains an error bound uniformly with respect to $\norm{f}{2}$ by using 
$\norm{f}{\infty} \leq \norm{\frac{1}{\pi}}{\infty}^{1/2} \norm{f}{2}$.
The spectral gap can be implied by aperiodicity and irreducibility of the Markov chain, 
see \cite[Lemma~12.1, p.~153]{peres}.
But it is remarkable that the chain can be non-reversible. 
If $\beta=\beta_1$ 
then the error bound has the right coefficient of the leading term. 
Then it is essentially the same bound as in Theorem~\ref{main_fin}.\\

Also confidence estimates of $S_{n,n_0}$ are of interest.  
The goal is to achieve 
for given precision $\e \in (0,1)$ and confidence parameter $\a \in (0,1)$ that
\begin{equation}  \label{confidence}
\Pr( \abs{S_{n,n_0}(f)-S(f)} \geq \e ) \leq \a.
\end{equation}
Such
approximations for confidence intervals can be implied by the mean square error.
\begin{lemma} \label{lemm_markov}
Let $(X_n)_{n\in\N}$ be a Markov chain with transition matrix $P$ and initial distribution $\nu$ 
and let $\e \in (0,1)$. Then 
\[
\Pr( \abs{S_{n,n_0}(f)-S(f)} \geq \e ) \leq \frac{e_\nu(S_{n,n_0},f)^2}{\e^2}.
\] 
\end{lemma}
\begin{proof}
The result is an application of the Markov inequality.
\end{proof}
Suppose that $\norm{f}{2}\leq 1$. 
If one applies 
Lemma~\ref{lemm_markov} and the burn-in is chosen as in Theorem~\ref{main_coro} then it follows for
\[
n_0 \geq \frac{\log(\sqrt{\norm{\frac{1}{\pi}}{\infty}} \norm{\frac{\nu}{\pi}-1}{2})}{\log(\beta^{-1})}
\quad \text{and} \quad
n\geq \frac{4\a^{-1}\e^{-2}}{ 1-\beta }
\]
that \eqref{confidence} is true. Note that the burn-in is chosen independently of $\a$.
In \cite[Theorem~12.19, p.~165]{peres} a similar bound is deduced by coupling arguments.
It implies a slightly worse result if the initial state is deterministically chosen. If
\[
n_0 \geq \frac{\log(2 \a^{-1} \norm{ \frac{1}{\pi}}{\infty})}{1-\beta} 
\quad \text{and} \quad
n	\geq \frac{4 \a^{-1} \e^{-2}}{1-\beta}
\]
then \eqref{confidence} is true. The main difference is the dependence of $\a$ in the choice of the burn-in.
One can essentially boost this confidence level 
by using a median of independent runs of the Markov chain Monte Carlo method. 
This is explained in \cite{niemiro_poka}.\\

However, both presented results are far away from well known Chernoff bounds.
These exponential inequalities for finite Markov chains 
are shown in \cite{gillman_chern} for random walks on graphs. 
In \cite{Lez98}, this Chernoff bound was extended 
and refined for Markov chains on finite and general state spaces, 
furthermore for discrete and continuous time.
For irreducible and reversible Markov chains on finite $D$ and $\norm{f}{\infty}\leq 1$ one obtains
from \cite[Theorem~1.1, p.~850]{Lez98} that
\begin{equation} \label{lezaud_exp_in}
%
\Pr( \abs{S_{n,n_0}(f)-S(f)} \geq \e ) \leq 3 \norm{\frac{\nu P^{n_0}}{\pi}}{2} \exp\set{-n (1-\beta_1) \frac{\e^2}{12}}.
\end{equation}
In other words, if 
\[
n_0\geq \frac{\log(\norm{\frac{\nu}{\pi}-1}{2})}{\log(\beta^{-1})} \quad \text{and} \quad 
%
%
%
n\geq \frac{12\e^{-2} \log(6\a^{-1})}{1-\beta_1}
\]
then \eqref{confidence} holds true. This is better than using Lemma~\ref{lemm_markov}.
In \cite{hoeffding_bounds} Hoeffding bounds for reversible Markov chains are presented.\\

Such exponential inequalities also imply an error bound of 
the mean square error by the following well known formula, see for example \cite[Lemma~2.4, p.~26]{kallenberg}.
\begin{lemma} \label{moments_tails}
Let $(X_n)_{n\in\N}$ be a Markov chain with transition matrix $P$ and initial distribution $\nu$.
Then
 \[
  e_\nu(S_{n,n_0},f)^2 = \int_0^\infty \Pr( \abs{S_{n,n_0}(f)-S(f)} \geq \sqrt{\eps} )\, \dint \e.
 \]
\end{lemma}
By Lemma~\ref{moments_tails} and by \eqref{lezaud_exp_in} one obtains the following error bound
\[
 \sup_{\norm{f}{\infty}\leq 1}e_\nu(S_{n,n_0},f)^2 \leq 
 \frac{36(1 + \beta^{n_0}\norm{\frac{\nu}{\pi}-1}{2}) }{n(1- \beta_1)}.
\]
The asymptotic coefficient as described in Corollary~\ref{prop_stat} and Corollary~\ref{asymp_err_coro} is not reached.
However, the error bound applies also for periodic Markov chains.\\

Let us provide a conclusion.
Different explicit error bounds for finite state spaces are known.
The results presented in Section~\ref{sec_error_bound} are not entirely new. 
In the literature one can find similar estimates where some of the assumptions 
like aperiodicity or reversibility are weakened. 
The justification and discussion of the burn-in in Section~\ref{burn_in_fin} 
and the lower bound of Theorem~\ref{main_coro} seem to be new.
In the following we will extend the results to general state spaces.

\chapter{General state spaces} \label{ch_gen}
In the following we study the mean square error of Markov chain Monte Carlo methods on general state spaces.
The state space can be countable or uncountable. 
In Section~\ref{MC_cont} we provide the basic definitions and properties of Markov chains on general state spaces. 
The estimates of the mean square error are shown in Section~\ref{err_bound_gen}. 
We suggest and justify a recipe how to choose the burn-in in Section~\ref{burn_in_sec}.
Afterwards the error bound is applied to illustrating examples and
finally we discuss how the results fit into the published literature. 

\section{Markov chains} \label{MC_cont}
In this section facts and definitions of Markov chains on general state spaces are stated.
The paper \cite{rosenthal} of Rosenthal and Roberts surveys various results about Markov chains on general state spaces.
For further reading we refer to \cite{revuz,Nummelin,tweed}.\\

Let $(D,\D)$ be a measurable space. 
In most examples $D$ is contained in $\R^d$ and $\D$ is given by $\Borel(D)$, 
where $\Borel(D)$ denotes the Borel $\sigma$-algebra over $D$.
In the following we provide the definition of a transition kernel 
and a Markov chain. 

\begin{defin}[\index{Markov kernel}Markov kernel, \index{transition kernel}transition kernel]
The function $K\colon D \times \D \to [0,1]$ is called a
\emph{Markov kernel} or a \emph{transition kernel} if   
\begin{enumerate}[(i)]
 \item for each $x\in D$ the mapping $A\in\D \mapsto K(x,A)$ is a probability measure on $(D,\D)$, 
 \item for each $A\in\D$ the mapping $x\in D \mapsto K(x,A)$ is a $\D$-measurable real-valued function.
\end{enumerate}
\end{defin}

\begin{defin}[\index{Markov chain}Markov chain]
A sequence of random variables $(X_n)_{n\in\N}$ on a probability space $(\Omega,\F,\Pr)$ mapping into $(D,\D)$ 
is called a \emph{Markov chain with transition kernel $K$} 
if for all $n\in\N$ and $A\in \D$ one has
\[
 \Pr(X_{n+1}\in A\mid X_1,\dots,X_n)=\Pr(X_{n+1}\in A\mid X_n)=K(X_n,A) \quad\text{almost surely.}
\] 
The distribution 
\[
  \nu(A)=\Pr(X_1\in A),\quad A\in\D
\]
is called the \emph{initial distribution}.
\end{defin}


Suppose that we have a transition kernel $K$ and a probability measure $\nu$.
For simplicity let us assume that $D\subset \R^d$ and $\D=\Borel(D)$.
For any transition kernel there exists 
a \emph{random mapping representation}, see for example 
Kallenberg \cite[Lemma~2.22, p.~34]{kallenberg}.
A random mapping representation is a measurable function $\Phi\colon D\times [0,1] \to D$, which satisfies
\[
  \Pr(\Phi(x,Z)\in A) = K(x,A), \quad x\in D,\, A\in \D,   
\] 
where the random variable $Z\colon(\Omega,\F,\Pr) \to ([0,1],\Borel([0,1]))$ 
is uniformly distributed.
Then a Markov chain can be constructed as follows.  
Let $(Z_n)_{n\in\N}$, with $Z_n:(\Omega,\F,\Pr) \to ([0,1],\Borel([0,1]))$, 
be a sequence of i.i.d. random variables with uniform distribution,  
and assume that $X_1$ has distribution $\nu$, then one can see that $(X_n)_{n\in\N}$ defined by 
\[
  X_n=\Phi(X_{n-1},Z_n),\quad n\geq2,
\]
is a Markov chain with transition kernel $K$ and initial distribution $\nu$.\\


The transition kernel $K$ of a Markov chain describes the probability of getting from state $x\in D$ to $A\in\D$ in one step, 
i.e. for all $k\in \N$ one has
\[
K(x,A)=\Pr(X_{k+1}\in A \mid X_k=x).
\] 
The $n$ step transition kernel is inductively given by
\[
K^n(x,A)=\int_D K^{n-1}(y,A)\, K(x,\dint y)=\int_{D}K(y,A)\,K^{n-1}(x,\dint y).
\]
The first equality of the previously stated equation is the definition and for 
a proof of the second equality see \cite[Proposition~1.6, p.~11]{revuz} or \cite[Theorem~3.4.2, p.~61]{tweed}.
The function $K^n$ again constitutes a transition kernel.
The $n$ step transition probability from state $x\in D$ to $A\in\D$ is
\[
\Pr(X_{k+n}\in A \mid X_k=x)=K^{n}(x,A).
\] 
This is seen by integrating over the conditional distribution of the previous step:
\begin{align*}
\Pr(X_{k+1} \in A \mid X_k=x) 	& =K(x,A),\\
\Pr(X_{k+2} \in A \mid  X_k=x) 	& =\int_D \Pr(X_{k+2} \in A \mid X_{k+1}=y , X_k=x) \Pr(X_{k+1}\in \dint y\mid X_k=x)\\
				& =\int_D \Pr(X_{k+2} \in A \mid X_{k+1}=y)\, K(x,\dint y) = K^2(x,A),\\
\vdots			\quad \quad \quad \quad \quad 					& \quad\quad\quad			\vdots \\
\Pr(X_{k+n} \in A \mid  X_k=x) 	& =\int_D \Pr(X_{k+n} \in A \mid X_{k+n-1}=y , X_k=x) \\
				& \qquad \qquad \qquad \qquad \qquad \qquad \qquad \quad \times \Pr(X_{k+n-1}\in \dint y\mid X_k=x)\\
				& =\int_D \Pr(X_{k+n} \in A \mid X_{k+n-1}=y)\, K^{n-1}(x,\dint y) = K^{n}(x,A).	
\end{align*}

In the following let us assume that we have a Markov chain $(X_n)_{n\in\N}$ with transition kernel $K$ 
and initial distribution $\nu$. 
The expectation $\expect_{\nu,K}$ is taken with respect to the joint distribution of $(X_n)_{n\in\N}$,
say $W_{\nu,K}$, 
which is defined on $(D^{\N},\sigma(\A))$  where
\begin{align*}
  D^{\N} & = \set{ \omega=(x_1,x_2,\dots) \mid x_i \in D \text{ for all } i\geq1 } \quad \text{and}\\
  \A & = \bigcup_{k\in\N}\set{A_1\times A_2 \times \dots \times A_k \times D \times \dots \mid    A_i \in \D,\; i=1,\dots,k},
\end{align*}
see \cite[Theorem~3.4.1, p.~60]{tweed} or \cite[Theorem~2.8, p.~17]{revuz}. 
For $k\in \N$ one has with $A_1 \times \dots \times A_k \subset D^k$ that
\begin{equation} \label{Pr_gen}
\begin{split}
& W_{\nu,K}(A_1 \times \dots \times A_k \times D \times \cdots)  = \Pr(X_1 \in A_1,\dots,X_{k} \in A_k) \\ 
& \quad = \int_{A_1} \int_{A_2} \dots  \int_{A_{k-1}}   K(x_{k-1},A_k) K(x_{k-2},\dint x_{k-1}) \dots K(x_1,\dint x_2)\, \nu( \dint x_1). 
\end{split}
\end{equation}
\\

Now we present properties of transition kernels. 
These properties have finite state space counterparts, see Section~\ref{MC_finite}.

By $\mathcal{M}(D)$ denote the set of \emph{real-valued
signed measures}\footnote{The set function $\mu\colon \D \to \R$ is a real-valued signed measure 
if $\mu(\emptyset)=0$ and for pairwise disjoint $A_1,A_2,\dots$, with $A_k\in\D$ for $k\in\N$, one has 
$\mu(\cup_{k=1}^\infty A_k)=\sum_{k=1}^\infty \mu(A_k).$}  on $(D,\D)$.
For any $\nu\in\mathcal{M}(D)$ let us define
\[
\nu P^m(A)=\int_D K^m(x,A)\, \nu(\dint x),\quad A \in\D,\ m\in\N.
\]
Note that the mapping $\nu \mapsto \nu P^m$ defines a linear operator on $\mathcal{M}(D)$.
 If $\nu$ is a probability measure then $\nu P^m$ 
 is the distribution of $X_{m+1}$, where $(X_n)_{n\in\N}$ is a Markov chain
with transition kernel $K$ and initial distribution $\nu$.  

\begin{defin}[\index{stationarity}stationarity]
Let $\pi$ be a probability measure on $(D,\D)$. 
Then $\pi$ is called a \emph{stationary distribution} of a transition kernel $K$
if 
\begin{equation*}
	\pi P (A)
=\pi(A) , \quad A\in \D.
\end{equation*}
\end{defin}  
Roughly spoken that means: Choosing the initial state with respect to a stationary distribution $\pi$, 
then, after a single transitions the same distribution as before arises, i.e. 
\[
 \Pr(X_1\in A)=\pi(A)=\pi P(A)=\Pr(X_2\in A), \quad A\in \D.
\]
\begin{defin}[\index{reversibility}reversibility]
Let $\pi$ be a probability measure on $(D,\D)$. A 
transition kernel $K$ is called \emph{reversible with respect to $\pi$} if 
\begin{equation*}
 \int_B K(x,A)\, \pi(\dint x)=\int_A K(x,B)\, \pi(\dint x), \quad A, B\in\D. 	
\end{equation*}
\end{defin} 
If a transition kernel $K$ is reversible with respect to a distribution $\pi$, then
$\pi$ is a stationary distribution of $K$.
If the initial distribution of a Markov chain with transition kernel $K$ is $\pi$,
then reversibility with respect to $\pi$ is equivalent to 
\[
\Pr(X_1\in A, X_2\in B)= \Pr(X_1 \in B, X_2 \in A),\quad A,B\in \D.
\]
A Markov chain is called reversible with respect to $\pi$, if the corresponding transition kernel is reversible
with respect to $\pi$. 
\begin{defin}[\index{lazy version}lazy version]
Let $K$ be a transition kernel and let $\mathbf{1}_A(x)$ be the indicator function of $A\in\D$ for $x\in D$. 
Then we call
\begin{equation*}
 \widetilde{K}(x,A)= \frac{1}{2}(\mathbf{1}_A(x)+K(x,A)),\quad x\in D,\;A\in\D,
\end{equation*}
the \emph{lazy version of $K$}.
\end{defin} 
If $\pi$ is a stationary distribution of $K$, then $\pi$ is also a stationary distribution of $\widetilde{K}$.
If $K$ is reversible with respect to $\pi$, 
then $\widetilde{K}$ is also reversible with respect to $\pi$. 
For a Markov chain with transition kernel $K$ and initial distribution $\nu$ we may define a \emph{lazy Markov chain},
a Markov chain with transition kernel $\widetilde{K}$ and initial distribution $\nu$.\\

Assume that $\pi$ is a stationary distribution of a transition kernel $K$ and let 
$f\colon D\to\R$ be an integrable function with respect to $\pi$. 
Let us define
\[
P^m f (x) =\int_D f(y)\, K^m(x,\dint y),\quad x\in D,\ m\in\N.
\]
We call $P$ the \emph{Markov operator} or the \emph{transition operator}.
If a Markov chain $(X_n)_{n\in\N}$ with transition kernel $K$ and initial distribution $\delta_x$, 
the point mass at $x\in D$, is given, 
then $P^m f(x)$ is the expectation of $f(X_{m+1})$. \\

Let us state some well known properties of the operator $P$ acting on functions and on signed measures.

\begin{lemma} \label{lem_nuP_pun}
Let $\pi$ be a stationary distribution of the transition kernel $K$ and 
let $f\colon D\to \R$ be an integrable function with respect to $\pi$.
Then one obtains for $\nu \in \mathcal{M}(D)$ that
\begin{equation}   \label{nuP_Pun}
\int_D f(x)\, (\nu P^m) (\dint x) = \int_D (P^m f)(x)\, \nu(\dint x), \quad m\in \N,
\end{equation}
whenever one of the integrals exist.
In particular 
\begin{equation} \label{rem_f}
S(f)=\int_D f(x)\, \pi(\dint x)= \int_D (P^m f)(x)\, \pi(\dint x), \quad m\in\N.
\end{equation}
\end{lemma}
\begin{proof}
Equation \eqref{rem_f} is an immediate consequence of \eqref{nuP_Pun} and stationarity.
Hence one has to prove \eqref{nuP_Pun}.
The equality holds for indicator functions and for simple functions. Then by the standard 
procedure of integration theory the equality can be extended
to positive and afterwards to integrable functions. 
\end{proof}
Note that if a Markov chain $(X_n)_{n\in\N}$ with transition kernel $K$ and initial distribution $\nu$
is given, then \eqref{nuP_Pun} can be rewritten as
\[
    \expect_{\nu,K}[f(X_{m+1})] = \expect_{\nu,K}[\expect_{\nu,K}[f(X_{m+1})\mid X_1]].
\]
The following result is well known, see for example \cite[equation (1.2), p.~365]{lova_simo1}.
\begin{lemma}  \label{f_eq}
Let the transition kernel $K$ be reversible with respect to $\pi$ and 
let $F\colon D \times D \to \R$. 
Then
\begin{equation}   \label{F_x_y_n}
\int_D \int_D F(x,y)\,K^m(x,\dint y)\,\pi(\dint x)=\int_D \int_D F(y,x)\,K^m(x,\dint y)\,\pi(\dint x),\quad m\in\N,
\end{equation} 
whenever one of the integrals exist.
\end{lemma}

\begin{proof}
The reversibility of the transition kernel $K$ implies reversibility 
of the $m$ step transition kernel $K^m$. Hence it is sufficient to show the assertion for $m=1$.
By using the reversibility one has
\[
\int_D \int_D \mathbf{1}_{A\times B} (x,y)\, K(x,\dint y)\, \pi(\dint x)
=\int_D \int_D \mathbf{1}_{A\times B} (y,x)\, K(x,\dint y)\, \pi(\dint x), \quad A,B\in\D.
\]
The equality of the integrals can be extended to arbitrary sets $C\in \D \otimes \D$, 
where $\D \otimes \D$ is the product $\sigma$-algebra of $\D$ with itself. 
This is an application of the Dynkin's Theorem. 
Then it is straightforward to consider the cases where $F$ is a simple function, 
a positive function and finally an integrable one.
\end{proof}

For $p\in[1,\infty)$ let us define
\[
L_p=L_p(D,\pi)=\set{ f\colon D\to \R \mid \norm{f}{p}^p = \int_D \abs{f(x)}^p \pi(\dint x) < \infty }.
\]
For $p=\infty$ the essential-supremum norm with respect to $\pi$ is defined by
\[
\norm{f}{\infty}=\esssup_{y\in D} \abs{f(y)}=\inf_{N \in \D,\ \pi(N)=0} \sup_{y\in D\setminus N} \abs{f(y)},
\] 
such that
\[
L_\infty=L_\infty(D,\pi)=\set{ f\colon D\to \R \mid  \norm{f}{\infty} < \infty }.
\] 
Sometimes it is convenient to consider bounded functions on $D$, 
not $\pi$-a.e. bounded ones, thus we define
\[
L_B=L_B(D)=\set{f \colon D\to \R \mid \abs{f}=\sup_{x\in D}\abs{f(x)}<\infty}.
\]

The next result is standard, see for example \cite[Lemma~1, p. 334]{baxter}.
\begin{lemma} \label{op_norm}
Let $p\in[1,\infty]$. For any transition kernel $K$ with a stationary distribution $\pi$ 
it follows that
\[
\norm{Pf}{p} \leq \norm{f}{p} \quad \text{and}\quad \norm{P}{L_p\to L_p}=1  .
\]
\end{lemma}
\begin{proof}
If $p<\infty$, then by Jensen inequality $(\text{J})$ and \eqref{rem_f} one obtains
\begin{align*}
\int_D \abs{Pf(x)}^p \pi(\dint x) & \leq  \int_D \left(\int_D \abs{f(y)} K(x,\dint y) \right)^p \pi(\dint x) \\
 & \underset{(\text{J})}{\leq} \int_D \int_D  \abs{f(y)}^p K(x,\dint y)\, \pi(\dint x) 
\underset{\eqref{rem_f}}{=} \int_D \abs{f(x)}^p \pi(\dint x).
\end{align*}
Since $\pi$ is a stationary distribution of the transition kernel one has for $N\in \D$ that
\[
\pi(N)=0\quad  \Longleftrightarrow \quad K(\cdot,N)=0\quad \pi\mbox{-a.e.}
\]
Null sets with respect to $\pi$ are the same as null sets with respect to $K(x,\cdot)$ 
for almost all $x\in D$. 
Hence
\[
\abs{Pf(x)}\leq \int_D \abs{f(y)} K(x,\dint y) \leq \norm{f}{\infty} \quad \pi\mbox{-a.e.} 
\]
and we have $\norm{Pf}{p}\leq \norm{f}{p}$ for $p\in[1,\infty]$. 
Let $u(x)=1$ for all $x\in D$. Then $Pu=u$ with $\norm{u}{p}=1$ and we obtain $\norm{P}{L_p\to L_p}=1$.
\end{proof}
The closed subspace
\[
L_p^0=\{f\in L_p\mid S(f)=0\} 
\]
of $L_p$ is important.
Note that $L_2$ and $L_2^0$ are Hilbert spaces with inner product 
\[
\scalar{f}{g}=\int_D f(x) g(x)\, \pi(\dint x).
\] 
Then
\[
L_2=L_2^0 \oplus (L_2^0)^\bot, \quad \mbox{where} \quad 
(L_2^0)^{\bot}  = \set{ f\in L_2\mid f\equiv c 
,\;c\in\R}.
\]
On the Hilbert spaces $L_2$ and $L_2^0$ there exists the \emph{adjoint operator} $P^*$ such that
\[
\scalar{Pf}{g}=\scalar{f}{P^*g}.
\]
Furthermore
\[
\norm{P}{L^0_2\to L^0_2} = \norm{P^*}{L^0_2\to L^0_2}
 \quad \text{and} \quad 
\norm{P-S}{L_2\to L_2} = \norm{P^*-S}{L_2\to L_2}.
\]
The following facts about adjoint operators are helpful.
Let $T:L_p\to L_p$, with $p\in [1,\infty)$,  be a linear  bounded operator. 
Then the adjoint operator $T^*:L_q \to L_q$, with $q\in(1,\infty]$, is defined as follows.
Suppose that $p$ and $q$ are chosen such that $p^{-1}+q^{-1}=1$. 
It is well known that $L_q$ is isometrically isomorphic to the dual space $(L_p)'$, 
where the isomorphism is given by 
\[
A \colon L_q \to(L_p)' , \quad A(g) (f) = \scalar{f}{g}, \quad f\in L_p.
\]
Then there exists the dual operator $T^\times \colon (L_p)' \to (L_p)'$ 
and the adjoint operator acting on $L_q$ can be defined as $T^* = A^{-1} T^\times A$.
Figure~\ref{komm_dia} illustrates the construction by a diagram.
\begin{figure}[H] 
\centering
\begin{tikzpicture}
  \matrix (m) [matrix of math nodes, row sep=3em, column sep=3em]
    { L_q & L_q\\ 
      (L_p)' & (L_p)'\\ };
   { [start chain] \chainin (m-1-1);
    \chainin (m-2-1) [join={node[right] {$\scriptstyle A$}}];
    \chainin (m-2-2) [join={node[above] {$\scriptstyle T^\times$}}];
    \chainin (m-1-2) [join={node[right] {$\scriptstyle A^{-1}$}}];
   }
  
  { [start chain] \chainin (m-1-1);
    \chainin (m-1-2) [join={node[above] {$\scriptstyle T^*$}}];
   }
\end{tikzpicture}
\caption{Illustration of the definition of the adjoint operator $T^* \colon L_q \to L_q$ of $T\colon L_p \to L_p$.}
\label{komm_dia}
\end{figure}
Furthermore, for all $f\in L_p$ and for all $g\in L_q$ one has
\begin{align*}
\scalar{f}{T^* g} & = \scalar{f}{A^{-1} T^\times A g} = A (A^{-1} T^\times A g )(f) \\
		  & = (T^\times A)( g )(f) \underset{(\text{dual operator})}{=} A(g)(Tf) =\scalar{Tf}{g}.
\end{align*}
Then 
\begin{align*}
\norm{T}{L_p \to L_p} & = \norm{T^\times }{(L_p)' \to (L_p)'} 
= \sup_{\norm{Ag}{(L_p)'}\leq1} \norm{T^\times Ag}{(L_p)'}\\
& = \sup_{\norm{g}{q}\leq1} \norm{A^{-1} T^\times Ag}{q} 
= \norm{T^*}{L_q \to L_q}.
\end{align*}
If 
$T=P-S$, then it follows that
\[
\norm{P-S}{L_p \to L_p} = \norm{P^*-S}{L_q \to L_q}.
\]
Let $\nu \in \mathcal{M}(D)$. If there exists a density of $\nu$ with respect to $\pi$ then
we denote it by $\frac{d\nu}{d\pi}$ and for $q\in[1,\infty]$ let
\[
\norm{\nu}{q}= 
\begin{cases}
  \norm{\frac{d\nu}{d\pi}}{q}, & \nu \ll \pi,\\
  \infty, & \mbox{otherwise}. 
\end{cases}
\]
Set
\[
\mathcal{M}_q = \mathcal{M}_q(D,\pi)= \set{ \nu \in \mathcal{M}(D)\mid \norm{\nu}{q} < \infty  }.
\] 
The function space $L_q$ is isometrically isomorphic to the space of signed measures $\mathcal{M}_q$, 
in symbols $L_q \cong \mathcal{M}_q$.
The space of singed measures $\mathcal{M}_2$ is a Hilbert space 
and the inner product is the inner product of $L_2$ of the densities, one has 
\[
\scalar{\nu}{\mu}=\int_D \frac{d\nu}{d\pi}(x)\,\frac{d\mu}{d\pi}(x)\, \pi(\dint x)=
\scalar{\frac{d\nu}{d\pi}}{\frac{d\mu}{d\pi}}, \quad \nu,\mu \in \mathcal{M}_2. 
\]
Furthermore set
\[
\mathcal{M}_q^0 = \set{ \nu \in \mathcal{M}_q\mid \nu(D)=0 }.
\]
Then
\[
\mathcal{M}_2=\mathcal{M}_2^0 \oplus (\mathcal{M}_2^0)^\bot, \quad \mbox{where} \quad 
(\mathcal{M}_2^0)^{\bot}  = \set{ \nu\in \mathcal{M}_2\mid \nu=c\cdot \pi 
,\;c\in\R}.
\]
Clearly, $\mathcal{M}_2^0$ is also a Hilbert space. We have
$L_2^0 \cong \mathcal{M}_2^0$ and $(L_2^0)^{\bot} \cong (\mathcal{M}_2^0)^{\bot}$. 
Let us recall that the transition kernel applies to signed measures $\nu\in \mathcal{M}_q$ as
\[
\nu P (A) = \int_D K(x,A)\, \nu(\dint x), \quad A\in\D.
\]
\begin{lemma} \label{P_norm}
Let $K$ be a transition kernel and
let $\pi$ be a stationary distribution of $K$.    
    \begin{enumerate}[(i)]
     \item \label{dense} 	
        Let $q\in(1,\infty]$ and $\nu \in \mathcal{M}_q$. Then 
	\[
 	\frac{d(\nu P)}{d\pi}(x)=P^*(\frac{d\nu}{d\pi})(x) \quad \pi\text{-a.e.}
 	\]
 	and
 	\[
 	\norm{P}{L^0_2\to L^0_2} = \norm{P}{\mathcal{M}^0_2 \to \mathcal{M}^0_2}. 
	\]
 	\item 	\label{adj_rev}
 	Reversibility with respect to $\pi$ is equivalent 
 	to $P$ being self-adjoint acting on $L_2$ and $\mathcal{M}_2$, i.e.
 	\[
	  \scalar{Pf}{g}=\scalar{f}{Pg} 
	  \quad \mbox{and} \quad 
	  \scalar{\nu P}{\mu} = \scalar{\nu}{\mu P}.
 	\]
 \end{enumerate}
\end{lemma}
\begin{proof}
	First, let us prove assertion \eqref{dense}.
	For all $f\in L_p$ with $p$ chosen such that $p^{-1}+q^{-1}=1$ one has
  \begin{align*}
\scalar{f}{\frac{d(\nu P)}{d\pi}} & = \int_D f(x)\, (\nu P)(\dint x) \underset{\eqref{nuP_Pun}}{=} \int_D (Pf)(x)\, \nu(\dint x) 
				    = \scalar{Pf}{\frac{d\nu}{d\pi}}	= \scalar{f}{P^*(\frac{d\nu}{d\pi})}.
 \end{align*}
Hence 
we have $\pi$-a.e.  
\[
  \frac{d(\nu P)}{d\pi}(x)=P^*(\frac{d\nu}{d\pi})(x).
\]
By using the previous equation one obtains
   \begin{align*}
 		\norm{P}{\mathcal{M}^0_2 \to \mathcal{M}^0_2} & = \sup_{\norm{\frac{d\mu}{d\pi}}{2}=1,\,\mu(D)=0} \norm{\frac{d(\mu P)}{d\pi} }{2} 
 		= \sup_{\norm{\frac{d\mu}{d\pi}}{2}=1,\,S(\frac{d\mu}{d\pi})=0}  
 		  \norm{P^*(\frac{d\mu}{d\pi})}{2} \\
 		& = \norm{P^*}{L^0_2\to L^0_2} = \norm{P}{L^0_2\to L^0_2} .
 \end{align*} 
 Let us turn to assertion \eqref{adj_rev}.
 It is clear that self-adjointness implies reversibility.
 The other direction follows by
 \begin{align*}
 \scalar{Pf}{g} 
= \int_D \int_D f(y)g(x) K(x,\dint y)\,  \pi(\dint x) 
& \underset{\eqref{F_x_y_n}}{=} \int_D \int_D f(x) g(y)\, K(x,\dint y)\, \pi(\dint x) 
= \scalar{f}{Pg}.
 \end{align*}
  The result with respect to $\mathcal{M}_2$ is shown by using \eqref{dense} and the self-adjointness of $P$ on $L_2$.
 \end{proof}

In the following we introduce several convergence properties of a Markov chain $(X_n)_{n\in\N}$
with transition kernel $K$ and initial distribution $\nu$. We assume that $\pi$ is a stationary distribution
of $K$. The goal is to quantify the speed of convergence of $\nu P^m$ to $\pi$ for increasing $m\in\N$.
%
For further details let us refer to 
\cite{hybrid}, \cite{rosenthal}  or \cite{chen}.

\begin{defin}[$L_2$-spectral gap\index{$L_2$-spectral gap}] \label{def_spectral_gap}
Let $P$ be the Markov operator with corresponding transition kernel $K$. 
Then there exists an (absolute) \emph{$L_2$-spectral gap}, if
\[
\beta=\norm{P}{L_2^0 \to L_2^0}<1,
\] 
where
the $L_2$-spectral gap is given by $1-\beta$.
\end{defin}
Let us briefly explain what this means for reversible transition kernel.
If the transition kernel $K$ is reversible with respect to $\pi$, then let $\spec(P | L_2)$ be the spectrum 
of the self-adjoint operator $P$ acting on $L_2$ and $\spec(P | L^0_2)$ be the spectrum of $P$ acting on $L_2^0$. 
Since $\norm{P}{L_2 \to L_2} \leq 1$ the spectrum $\spec(P | L_2)$ is contained in $[-1,1]$.
Let us define
\[
  \lambda=\inf \set{ \a \mid \a\in \spec(P|L_2^0)}
  \quad \text{and} \quad
  \Lambda=\sup \set{\a \mid \a\in \spec(P|L_2^0)}.
\]
Since $P$ is self-adjoint, it is well known that
\[
\lambda=\inf_{\norm{g}{2}=1,\, g\in L_2^0} \scalar{Pg}{g}
\quad \text{and} \quad
 \Lambda=\sup_{\norm{g}{2}=1,\, g\in L_2^0} \scalar{Pg}{g}.
\]
Then we have
\[
\spec(P|L_2^0) \subset [\lambda,\Lambda] 
\quad \mbox{and} \quad
\beta=\norm{P}{L_2^0 \to L_2^0}=\max\{ \Lambda,\abs{\lambda}\}.
\]
The existence of an $L_2$-spectral gap implies that $-1<\lambda \leq \Lambda<1$,
consequently there is a gap between $1\in\spec{(P| L_2)}$ 
and $\beta$, the second largest absolute value of $\spec(P|L_2)$.
%
%
 
\begin{defin}[\index{$L_2$-geometric ergodicity}$L_2$-geometric ergodicity] \label{geo_erg}
A transition kernel $K$ with stationary distribution $\pi$ is called \emph{$L_2$-geometrically ergodic}, if 
for all probability measures $\nu \in \mathcal{M}_2$ 
there exists an $\a \in[0,1)$ and $C_\nu<\infty$
such that 
\[
	\norm{\nu P^n-\pi}{2} \leq C_\nu\, \a^n  ,\quad n\in\N. 
\] 
\end{defin}
An $L_2$-spectral gap implies $L_2$-geometric ergodicity.
\begin{prop} \label{gap_geo}
Let $K$ be a transition kernel with stationary distribution $\pi$.
Assume that the Markov operator $P$ has an $L_2$-spectral gap, i.e. $1-\beta>0$. 
Then 
the transition kernel $K$ is $L_2$-geometrically ergodic.
\end{prop}
\proof
If $\nu\in \mathcal{M}_2$ and $\nu(D)=1$, then one obtains $(\nu-\pi)(D)=0$ and the proof is completed by
\[
\norm{\nu P^n-\pi}{2}=\norm{(\nu-\pi)P^n}{2}\leq \norm{P}{\mathcal{M}_2^0\to \mathcal{M}_2^0}^n \norm{\nu-\pi}{2}
= \beta^n \norm{\nu-\pi}{2}.
\sq
\]
If the transition kernel is reversible with respect to $\pi$, then 
$L_2$-geometric ergodicity 
and the existence of an $L_2$-spectral gap are equivalent. 
This result is shown in \cite{hybrid}. 

\begin{prop} \label{geo_geo}
Let the transition kernel $K$ be reversible with respect to $\pi$. 
Then the following statements are equivalent:
\begin{enumerate}[(i)]
\item
The transition kernel is $L_2$-geometrically ergodic.
\item 
The Markov operator $P$ has an $L_2$-spectral gap.
\end{enumerate}
\end{prop}
\begin{proof}
See \cite[Theorem~2.1, p.~17]{hybrid}.
\end{proof}
For further details and even more equivalences of $L_2$-geometric ergodicity,
see \cite{hybrid,roberts_tweed}.
The next definition is similar to $L_p$-exponential convergence in \cite{chen}.
\begin{defin}[\index{$L_p$-exponential convergence}$L_p$-exponential convergence]  \label{L_p_exp_def}
Let $p\in[1,\infty]$, let $\a\in[0,1)$ and $M<\infty$. Then 
the transition kernel $K$ with stationary distribution $\pi$ 
is called 
\emph{$L_p$-exponentially convergent with $(\a,M)$} if 
\[
\norm{P^n-S}{L_p\to L_p} \leq M \a^n, \quad n\in \N.
\]
\end{defin}
The transition kernel is called \emph{$L_p$-exponentially convergent} if there exist an 
$M<\infty$ and an $\a\in[0,1)$ such that it is $L_p$-exponentially convergent with $(\a,M)$.\\

The Markov chain is called $L_2$-geometrically ergodic or $L_p$-exponentially convergent
if the corresponding transition kernel $K$ is $L_2$-geometrically ergodic or $L_p$-expo\-nentially convergent.

Let $p$ and $q$ be chosen such that $p^{-1}+q^{-1}=1$.
The condition of $L_p$-exponential convergence implies convergence of $\nu P^n$
to the
stationary distribution $\pi$ for increasing $n\in\N$ in $\mathcal{M}_q$.
\begin{coro}
Let $p\in[1,\infty)$ and $\nu \in \mathcal{M}_q$
with $p^{-1}+q^{-1}=1$.
Let the transition kernel $K$ with stationary distribution $\pi$ be $L_p$-exponentially convergent with $(\a,M)$. 
Then
\[
 \norm{\nu P^n-\pi}{q} 
\leq  M \norm{\nu-\pi}{q} \a ^n,\quad n\in \N.
\]
\end{coro}
\proof
The assertion is proven by
\begin{align*}
\norm{\nu P^n -\pi}{q}  &= \norm{(\nu-\pi)P^n}{q} = 
\norm{\frac{d((\nu-\pi) P^n)}{d \pi}}{q} 
		= \norm{(P^n)^*\left(\frac{d\nu}{d\pi}-1\right)}{q}  \\
	&	= \norm{((P^n)^*-S)\left(\frac{d\nu}{d\pi}-1\right)}{q}
		\leq \norm{(P^n-S)^*}{L_q\to L_q} \norm{\frac{d\nu}{d\pi}-1}{q} \\
	&	\leq \norm{P^n-S}{L_p \to L_p} \norm{\frac{d\nu}{d\pi}-1}{q}
		\leq M \norm{\nu-\pi}{q} \a^n.
\sq
\end{align*}
	
In the following we consider relations between 
the existence of an $L_2$-spectral gap and $L_p$-exponential convergence. 
First, let us add some helpful inequalities.

\begin{lemma} \label{P_S_gen}
Let $\pi$ be a stationary distribution of the transition kernel $K$. Then
\begin{equation} \label{P_S_2_ident_gen}
\norm{P^n}{L_2^0 \to L_2^0}=\norm{P^n-S}{L_2\to L_2} \leq \beta^n,\quad n\in \N.
\end{equation}
If $\,p\in[1, \infty]$ then
\begin{equation} \label{P_S_p_ident_gen}
\norm{P^n}{L_p^0\to L_p^0} \leq \norm{P^n-S}{L_p\to L_p} \leq 2 \norm{P^n}{L_p^0\to L_p^0},\quad n\in \N.
\end{equation}
\end{lemma}
\begin{proof}
Note that if $P$ is a \emph{normal operator}, i.e. $PP^*=P^* P$, then $\norm{P^n}{L_2^0 \to L_2^0} = \beta^n$, 
otherwise one has $\norm{P^n}{L_2^0 \to L_2^0}  \leq \norm{P}{L_2^0 \to L_2^0}^n = \beta^n$.
By
\begin{align*}
\norm{P^n-S}{L_2\to L_2} & =\sup_{\norm{f}{2}\leq1} \norm{(P^n-S)f}{2} = \sup_{\norm{f}{2}\leq1} \norm{P^n(f-S(f))}{2} \\
			 & \leq 
				\sup_{\norm{g}{2}\leq1,\; S(g)=0} \norm{P^n g}{2} 
			= \norm{P^n}{L_2^0 \to L_2^0} 
\end{align*}	
and													
\begin{align*}
\norm{P^n}{L_p^0 \to L_p^0} & = \sup_{\norm{g}{p}\leq1,\; S(g)=0} \norm{P^n g}{p}
			= \sup_{\norm{g}{p}\leq1,\; S(g)=0} \norm{P^n g-S(g)}{p}\\
			&	\leq \sup_{\norm{f}{p}\leq1} \norm{(P^n-S)f}{p} = \norm{P^n-S}{L_p\to L_p}
\end{align*}
claim \eqref{P_S_2_ident_gen} and the first part of \eqref{P_S_p_ident_gen} are shown. 
Furthermore one obtains
\begin{align*}
\norm{P^n-S}{L_p\to L_p} & = \sup_{\norm{f}{p}\leq1} \norm{P^n f-Sf}{p} 
=2 \sup_{\norm{f}{p}\leq1} \norm{P^n (\frac{1}{2}(f-Sf))}{p} \\
& \leq 2 
	\sup_{\norm{g}{p}\leq1,\;S(g)=0} \norm{P^n g}{p}
= 2 \norm{P^n }{L_p^0\to L^0_p},
\end{align*}
which finishes the proof.
\end{proof}

In a general setting it follows that 
an $L_2$-spectral gap implies $L_p$-exponential convergence
for all $p\in(1,\infty)$. 

\begin{prop} \label{prop_Lp_gen}
Let $p \in(1, \infty)$.
Let $\pi$ be a stationary distribution of the transition kernel $K$ and $n\in\N$. 
The existence of an $L_2$-spectral gap, $1-\beta>0$, implies $L_p$-exponential convergence. 
We obtain
\begin{equation}\label{norm_Lp_gen}
	\norm{P^n-S}{L_p\to L_p} 	\leq
	\begin{cases}
	  2^{2/p}\,\beta^{{2n\frac{p-1}{p}}}, 	& p\in(1,2),\\
	   2^{2\frac{p-1}{p}}\,\beta^{2n/p},		&	p\in[2,\infty).
	\end{cases}
	\end{equation} 
\end{prop} 

\proof
Let $p\in(1,2)$. 
Lemma~\ref{P_S_gen} gives 
\[
\norm{P^n-S}{L_2 \to L_2} \leq \beta^n \quad \text{and} \quad \norm{P^n-S}{L_1 \to L_1} \leq 2.
\]
We apply Proposition~\ref{riesz_thorin} (Interpolation Theorem of Riesz-Thorin), where $T=P^n-S$ and $q_1=2$, $q_2=1$ such that $\theta=\frac{2-p}{p}$.  
The case where $p\in(2,\infty)$ follows by the same interpolation argument, since 
by Lemma~\ref{P_S_gen} one has
\[
\norm{P^n-S}{L_2 \to L_2} \leq \beta^n \quad \text{and} \quad \norm{P^n-S}{L_\infty \to L_\infty} \leq 2.
\sq
\]
From Proposition~\ref{prop_Lp_gen} and actually already 
from \eqref{P_S_2_ident_gen} it follows that an $L_2$-spectral gap implies $L_2$-exponential convergence.
With the additional assumption of normality of $P$ one can prove the reverse direction.
\begin{prop} \label{spec_L_2conv}
Let $\pi$ be a stationary distribution of the transition kernel $K$. 
Let the Markov operator $P$ be normal, i.e. $PP^*=P^*P$.
Then the following statements are equivalent: 
\begin{enumerate}[(i)]
	\item \label{gap_prop}
				There exists an $L_2$-spectral gap, i.e. $1-\beta>0$.
	\item \label{L_2_imply_spec} 
	There exist an $\a\in[0,1)$ and $M<\infty$ such that the transition kernel $K$ is
				$L_2$-exponentially convergent with $(\a,M)$.
\end{enumerate}
In particular \eqref{L_2_imply_spec} implies
\begin{equation*}
	\beta = \norm{P-S}{L_2\to L_2} 	\leq \a, 
\end{equation*} 
so that 
\[
\beta=\min\set{\a \mid \exists \; M < \infty \; \text{with} \; \norm{P^n-S}{L_2 \to L_2} \leq M \a ^n, \,n\in\N}.
\]
\end{prop}
\begin{proof}
By \eqref{P_S_2_ident_gen} of Lemma~\ref{P_S_gen} 
one has that \eqref{gap_prop}
implies \eqref{L_2_imply_spec} with $(\a,M)=(\beta,1)$.
Now we show that $\eqref{L_2_imply_spec}$ implies \eqref{gap_prop}.
One has
\[
\norm{P}{L^0_2 \to L^0_2}^2 = \norm{PP^*}{L^0_2 \to L^0_2},
\]
where $PP^*$ is self-adjoint and $(P^*)^n=(P^n)^*$ for all $n\in\N$.
Then
\begin{align*} 
 \norm{P^n-S}{L_2 \to L_2}^2 & = \norm{P^n}{L_2^0 \to L_2^0}^2 
=\norm{P^n (P^n)^*}{L^0_2 \to L^0_2}\\ 
& = \norm{P^n (P^*)^n}{L^0_2 \to L^0_2} 
\underset{\text{(normality)}}{=} \norm{(PP^*)^n}{L^0_2 \to L^0_2}
\end{align*}
such that
\begin{equation}  \label{equi_2}
\norm{P^n-S}{L_2 \to L_2} \leq M \a^n  
\quad \Longleftrightarrow
\quad
\norm{(PP^*)^n}{L^0_2 \to L^0_2} \leq M^2 \a^{2n}.  
\end{equation}
By the spectral radius formula and the self-adjointness (s-a) of $PP^*$ one obtains
\begin{align*}
& \qquad \norm{P}{L_2^0 \to L_2^0}^2=\norm{PP^*}{L^0_2 \to L^0_2}  \underset{\text{(s-a)}}{=} r[PP^*] \\
& 	=\lim_{n\to \infty} (\norm{(PP^*)^n}{L^0_2 \to L^0_2})^{1/n}
\underset{\eqref{equi_2}}{\leq} \a^2 \lim_{n\to \infty} (M^2)^{1/n} \leq \a^2. 
\end{align*}
Hence the proof is completed.
\end{proof}

By an interpolation argument we get that 
$L_\infty$-exponential convergence or $L_1$-exponential convergence imply an $L_2$-spectral gap 
if the Markov operator is normal.

\begin{prop}  \label{L_1conv_spec}
        Let $\pi$ be a stationary distribution of the transition kernel $K$. 
	Let $K$ be $L_1$-exponentially convergent or $L_\infty$-exponentially 
	convergent with $(\a,M)$.
	Suppose that the Markov operator $P$ is normal, i.e. $PP^*=P^*P$.
	Then there exists an $L_2$-spectral gap, in particular one obtains
	\begin{equation}  \label{L_1conv_spec_eq}
		\beta= \norm{P-S}{L_2\to L_2} \leq \sqrt{\a}.
	\end{equation}
\end{prop}
\begin{proof}
	We show that $L_1$-exponential convergence with $(\a,M)$ implies $\beta\leq \sqrt{\a}$. 
	For $L_\infty$-exponentially convergent Markov chains the claim follows by the same arguments, 
	where the roles of $L_\infty$ and $L_1$ are interchanged.\\
	By the assumptions of the proposition and Lemma~\ref{P_S_gen} one has
	\[
		\norm{P^n-S}{L_1\to L_1} \leq \a^n M,\quad \text{and} \quad 
		\norm{P^n-S}{L_\infty \to L_\infty } \leq 2.
	\]
	By Proposition~\ref{riesz_thorin} (Interpolation Theorem of Riesz-Thorin), where $T=P^n-S$ and $q_1=1$, $q_2=\infty$, $\theta=\frac{1}{2}$
	one obtains $L_2$-exponential convergence with $(\sqrt{\a},2^{3/2}M^{1/2})$. 
	Then Proposition~\ref{spec_L_2conv} implies $\beta\leq \sqrt{\a}$ and the proof is completed.
\end{proof}

Another way to measure the convergence of $\nu P^n$ to $\pi$ for increasing $n\in\N$
is provided by using the total variation distance, defined as follows.

\begin{defin}[\index{total variation distance}total variation distance]
The \emph{total variation distance} between two probability measures $\nu,\mu \in \mathcal{M}(D)$ is defined by
\[
\norm{\nu-\mu}{\text{\rm tv}}=\sup_{A\in\D}\abs{\nu(A)-\mu(A)}.
\]
\end{defin}

The total variation distance can be considered as an $L_1$-norm.
\begin{lemma}  \label{tv_present}
Let $\nu,\mu \in \mathcal{M}(D)$ be probability measures. Then
\begin{equation} \label{bounded_fct}
\norm{\nu-\mu}{\text{\rm tv}}=\frac{1}{2} \sup_{\abs{f}\leq 1} \abs{\int_D f(x) (\nu(\dint x)-\mu(\dint x))},
\end{equation}
where $\abs{f}= \sup_{x\in D} \abs{f(x)}$.
If $\nu,\mu \in \mathcal{M}_1$, then
$
\norm{\nu-\mu}{\text{\rm tv}}=\frac{1}{2}\norm{\nu-\mu}{1}.
$
\end{lemma}
\begin{proof}
See \cite[Proposition~3, p.~28]{rosenthal}.
\end{proof}


Now we can define uniform ergodicity of a transition kernel $K$.

\begin{defin}[\index{uniform ergodicity}uniform ergodicity,
	      \index{$\pi$-a.e. uniformly ergodicity}$\pi$-a.e. uniform ergodicity]
Let $M<\infty$ and $\a\in [0,1)$. Then the transition kernel $K$ with stationary distribution $\pi$ 
is called \emph{uniformly ergodic with $(\a,M)$} if one has
for all $x\in D$ that
\begin{equation}  \label{pi_uni_erg}
	\norm{K^n(x,\cdot)-\pi}{\text{\rm tv}} \leq M \a^n  ,\quad n\in\N. 
\end{equation}
If the inequality of $\eqref{pi_uni_erg}$ holds $\pi$-a.e, rather than for all $x\in D$, then the transition kernel $K$ 
is called \emph{$\pi$-a.e uniformly ergodic with $(\a,M)$}. 
A Markov chain with transition kernel $K$ is called \emph{uniformly ergodic} or \emph{$\pi$-a.e uniformly ergodic} 
if there exist an $M<\infty$ 
and an $\a\in[0,1)$ such that $K$ is 
uniformly ergodic or $\pi$-a.e uniformly ergodic with $(\a,M)$. 

\end{defin}
Obviously, if the transition kernel is uniformly ergodic then it is also $\pi$-a.e. uniformly ergodic.
Note that in other references, e.g. \cite{chen}, uniform ergodicity is called strong ergodicity.\\

Uniform ergodicity is closely related to $L_\infty$-exponential convergence.
An important relation is presented in the following proposition.
Recall that $L_B=L_B(D)$ denotes the class of bounded functions on $D$. 
\begin{prop} \label{uniform_L_inf}
Let $\a\in[0,1)$ and $M<\infty$. Let $\pi$ be a stationary distribution of the transition kernel $K$. 
Then the following statements are equivalent:
\begin{enumerate}[(i')]
	\item \label{richtig_unif_erg}
	The transition kernel $K$ is uniformly ergodic with $(\a,M)$.
	\item \label{B_infty_erg}
	The transition operator $P$ satisfies 
	\[
	\norm{P^n-S}{L_B \to L_B}\leq 2M\,\a^n,\quad n\in\N. 
	\]
\end{enumerate}
Furthermore (\ref{richtig_unif_erg}') and (\ref{B_infty_erg}') 
imply the following equivalent statements:
\begin{enumerate}[(i)]
	\item \label{unif_erg}
	The transition kernel $K$ is $\pi$-a.e. uniformly ergodic with $(\a,M)$.
	\item \label{L_infty_erg}
	The transition kernel $K$ is $L_\infty$-exponentially convergent with $(\a,2M)$.
\end{enumerate}
\end{prop}
\begin{proof}
By Lemma~\ref{tv_present} the equivalence of (\ref{richtig_unif_erg}') and (\ref{B_infty_erg}') holds true.
The equivalence of \eqref{unif_erg} and \eqref{L_infty_erg} remains to prove.
First, let us show that $\pi$-a.e.
\[
\sup_{\norm{f}{\infty}\leq 1} \abs{P^n f(x)-S(f)} = \sup_{\abs{f}\leq 1} \abs{ P^n f(x)-S(f) }.
\]
Note that
\[
\pi(N)=0\quad \Longleftrightarrow \quad K^n(\cdot,N)=0\quad \pi \mbox{-a.e.}
\]
for all $N\in\D$ and $n\in\N$, since $\pi$ is the stationary distribution.  
Suppose that $f\in L_\infty$.
Obviously, if $N\in\D$ and $\pi(N)=0$ then $\pi$-a.e.
\[
\abs{P^n f(x)-S(f)}=\abs{P^n (\mathbf{1}_{N^c}f)(x)-S(\mathbf{1}_{N^c}f)}.
\]
Let $\norm{f}{\infty}\leq 1$, i.e. $\pi(\set{x\in D: f(x)>1})=0$. 
Define 
\[
g(x)=	\begin{cases}
			f(x), & f(x)\leq 1, \\
			1,		&	f(x)>1,
			\end{cases}
\] 
such that $f(x)=g(x)$ holds $\pi$-a.e. and $\abs{g}\leq1$. Thus, $\pi$-a.e.
\[
\abs{P^n f(x)-S(f)} = \abs{P^n g(x)-S(g)} \leq \sup_{\abs{g}\leq1} \abs{P^n g(x)-S(g)}, 
\]
so that $\pi$-a.e.
\[
\sup_{\norm{f}{\infty}\leq1}\abs{P^n f(x)-S(f)} \leq \sup_{\abs{g}\leq1} \abs{P^n g(x)-S(g)}. 
\]
The inequality in the other direction is clearly also correct, i.e. $\pi$-a.e.
\begin{equation*} 
\sup_{\norm{f}{\infty}\leq1}\abs{P^n f(x)-S(f)} = \sup_{\abs{g}\leq1} \abs{P^n g(x)-S(g)}.
\end{equation*}
By applying the essential-supremum on both sides of the previous equation and \eqref{bounded_fct} one obtains
\[
\norm{P^n-S}{L_\infty \to L_\infty} = 2  
\esssup_{x\in D} \norm{K^n(x,\cdot)-\pi}{\text{\rm tv}}.
\]
Hence the proof is completed.
\end{proof}

It is known that there are transition kernels 
where the Markov operators have an $L_2$-spectral gap
and the transition kernels are not uniformly ergodic, see \cite{mengersen}.
Furthermore, uniform ergodicity implies an $L_2$-spectral gap, see \cite{hybrid}. 
In this sense uniform ergodicity is a stronger property than the existence of an $L_2$-spectral gap.  

\begin{prop}  \label{uni_impl_geo}
Let $\a\in[0,1)$ and $M<\infty$.
Let the transition kernel $K$ be reversible with respect to $\pi$.
Then the following statements are equivalent:
\begin{enumerate}[(i)]
\item \label{L_1_erg}
The transition kernel $K$ is $L_1$-exponentially convergent with $(\a,2M)$.
\item \label{L_inf_erg}
The transition kernel $K$ is $L_\infty$-exponentially convergent with $(\a,2M)$.
\item \label{uni_erg}
The transition kernel $K$ is $\pi$-a.e. uniformly ergodic with $(\a,M)$.
\end{enumerate}

Each of the conditions 
imply that the Markov operator has an $L_2$-spectral gap.
We have
\[
\beta=\norm{P}{L_2^0\to L_2^0} \leq \a.
\]
\end{prop}

\proof
First we prove the equivalence of \eqref{L_1_erg} and \eqref{L_inf_erg}.
By reversibility one can see for $f\in L_1$ and $h\in L_\infty$ that
\[
\scalar{(P^n-S)f}{h}\underset{\eqref{F_x_y_n}}{=}\scalar{f}{(P^n-S)h}.
\]
The adjoint operator of $P^n-S$ acting on $L_1$ is $P^n-S$ acting on $L_\infty$.
Then, 
one has
\[
\norm{P^n-S}{L_1 \to L_1} = \norm{P^n-S}{L_\infty\to L_\infty}
\]
and the equivalence is obvious.

By Proposition~\ref{uniform_L_inf} one has that \eqref{L_inf_erg} is equivalent to \eqref{uni_erg}.

The last implication follows by an interpolation argument.
Proposition~\ref{riesz_thorin} (Interpolation Theorem of Riesz-Thorin) with $q_1=\infty$, $q_2=1$ and $\theta=1/2$ is applied. 
Then,
\begin{equation}  \label{L_2_est_L_2}
\norm{P^n}{L_2^0 \to L_2^0} \underset{\eqref{P_S_2_ident_gen}}{=}\norm{P^n-S}{L_2\to L_2} \leq 4 M \a^n,\quad n\in\N.
\end{equation}
Because of the self-adjointness (s-a) of $P$ one can apply
the spectral radius formula and one obtains
\begin{align*}
\beta 	& = \norm{P}{L_2^0 \to L_2^0}
				  \underset{\text{(s-a)}}{=} 
				  r[P]=\lim_{n\to \infty} (\norm{P^n}{L^0_2 \to L^0_2})^{1/n} 
				\underset{\eqref{L_2_est_L_2}}{\leq} \a \cdot \lim_{n\to \infty } (4 M)^{1/n} = \a. \sq 
\end{align*}

In Figure~\ref{rev_erg} we present a survey of the discussed relations between 
the terms of convergence and ergodicity.
\begin{figure}[htb] 
 	\centering
    \input{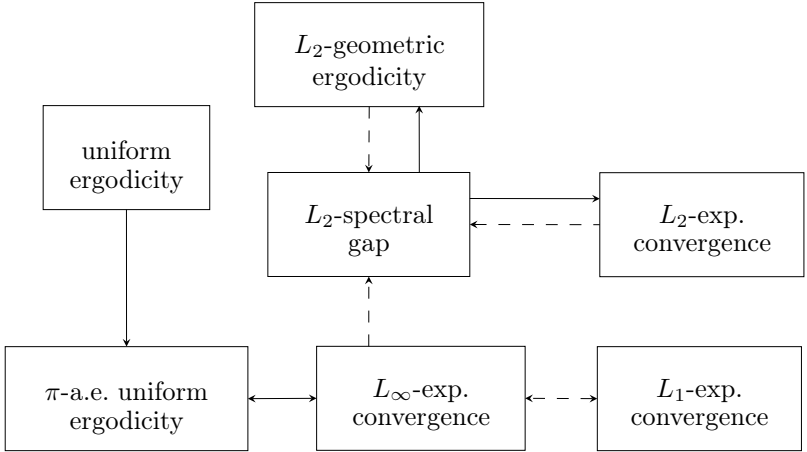}
    \caption{Ergodicity terms and their relations are illustrated. 
    A solid line represents the implication without any assumption of reversibility.
    A dashed line represents the implication under the assumption of reversibility. 
    }
    \label{rev_erg}
\end{figure} 


\section{Error bounds} \label{err_bound_gen}
In this section we prove error bounds on general state spaces. 
We assume that we have a Markov chain $(X_n)_{n\in\N}$ with transition kernel $K$ and initial distribution $\nu$, 
where $\pi$ is a stationary distribution, and
compute 
\[
  S_{n,n_0}(f)=\frac{1}{n}\sum_{j=1}^n f(X_{j+n_0})
\]
as approximation for $S(f)=\int_{D} f(x)\, \pi(\dint x)$. 
The error is measured in the mean square sense, i.e.
\[
   e_\nu(S_{n,n_0},f)=\left( \expect_{\nu,K} \abs{S_{n,n_0}(f)-S(f)}^2 \right)^{1/2}.
\]
Now let us present a helpful result.
\begin{lemma}
Let $(X_n)_{n\in\N}$ be a Markov chain with transition kernel $K$ and initial distribution $\nu$. 
Then for $i,j\in \N$ with $j\leq i$ it follows that
\begin{equation} \label{expect_ff}
\expect_{\nu,K}[f(X_i)f(X_j)] = \int_D P^j ( fP^{i-j}f )(x)\, \nu(\dint x).
\end{equation}
Moreover, if $\pi$ is a stationary distribution and $\nu=\pi$ then
\begin{equation} \label{expect_pi}
  \expect_{\pi,K}[f(X_i)f(X_j)]=\scalar{f}{P^{i-j}f}.
\end{equation}
\end{lemma}
\begin{proof}
The calculation
\begin{align*}
\expect_{\nu,K}[f(X_i)f(X_j)] & 
= \underbrace{\int_D \dots \int_D}_{i\text{-times}} 
			f(x_i)f(x_j)\, K(x_{i-1},\dint x_i) \dots K(x_1,\dint x_2)\, \nu(\dint x_1)\\
& = \underbrace{\int_D \dots \int_D}_{j\text{-times}} 
			f(x_j) P^{i-j}f (x_j)\, K(x_{j-1},\dint x_j)\dots K(x_1,\dint x_2)\, \nu(\dint x_1)\\
& = \int_D P^j(fP^{i-j}f)(x)\, \nu(\dint x)
\end{align*}
proves \eqref{expect_ff} and by \eqref{rem_f} one can see \eqref{expect_pi}.
\end{proof}
First we assume that the initial distribution of the Markov chain is a stationary one.
Hence it is not necessary to do any burn-in, i.e. $n_0=0$.
The resulting method is denoted by $S_n$ instead of $S_{n,0}$.
Afterwards we turn to the general method $S_{n,n_0}$ where the initial distribution might differ from
a stationary one.

In the next statement we assume that the transition kernel is reversible with respect to $\pi$.
Then we can apply the Spectral
Theorem for linear, bounded and self-adjoint operators, see Theorem~\ref{spec_thm}.

\begin{prop} \label{expl_stat_c}
Let $f\in L_2$ and $g=f-S(f)$. Let $(X_n)_{n\in\N}$ be a Markov chain with transition kernel $K$ 
and initial distribution $\pi$, let $K$ be reversible with respect to $\pi$ and let
\[
\lambda = \inf \set{ \a\mid \a \in \spec(P|L_2^0) },
 \quad\quad
\Lambda = \sup \set{ \a\mid \a \in \spec(P|L_2^0) }. 
\]
Suppose that $\Lambda<1$. Then 
\begin{equation} \label{err_present_c}
e_\pi(S_n,f)^2 = 
  \frac{1}{n^2} \int_\lambda^\Lambda W(n,\a)\, \dint\scalar{E_{\set{\a}} g}{g} = \frac{1}{n^2} \scalar{W(n,P)g}{g},
\end{equation}
where $E$ denotes the spectral measure\footnote{The definition of a spectral measure and the Spectral Theorem for linear, bounded self-adjoint operators 
are stated in Section~\ref{sec_spec}.} which corresponds to $P\colon L_2^0 \to L_2^0$ 
and recall that  
\begin{equation*}
   	W(n,\a)=\frac{n(1-\a^2)-2\a(1-\a^{n})}{(1-\a)^2}, \quad \alpha\in[-1,1)	.
\end{equation*}
\end{prop}
\proof
 Since $f\in L_2$ we have $g\in L_2^0$. The error obeys
 \begin{align*}
     e_\pi(S_n,f)^2 &=\expect_{\pi,K}\abs{\frac{1}{n}\sum_{j=1}^n g(X_j )}^2
     		     =\frac{1}{n^2}\, \expect_{\pi,K} \abs{ \sum_{j=1}^n g(X_j)}^2 \\
     		   & =\frac{1}{n^2} \sum_{j=1}^n \expect_{\pi,K}[g(X_j)^2]
     				+ \frac{2}{n^2} \sum_{j=1}^{n-1} \sum_{i=j+1}^n \expect_{\pi,K}[g(X_j)g(X_i)].
 \end{align*}
 For $i,j\in\N$ with $j\leq i$ we obtain
 \begin{align*}
  	 \expect_{\pi,K}[g(X_i)g(X_j)]	
\underset{\eqref{expect_pi}}{=} 
	\scalar{g}{P^{i-j}g}
 	= \int_\lambda^\Lambda \a^{i-j}\,\dint\scalar{E_{\set{\a}} g}{g},
\end{align*} 		
 where the last equality is an application of Theorem~\ref{spec_thm}. 
 Altogether this gives
 \begin{align*}
 e_\pi(S_n,f)^2 & =
                 \frac{1}{n^2} \int_\lambda^\Lambda 
 		  \left[ n + 2 \sum_{j=1}^{n-1} \sum_{i=j+1}^n \a^{i-j} \right] \,\dint\scalar{E_{\set{\a}} g}{g}\\
 		&	= \frac{1}{n^2} \int_\lambda^\Lambda 
 				\left[ n + 2 \frac{(n-1)\a-n\a^2+\a^{n+1}}{(1-\a)^2}\right]\,\dint\scalar{E_{\set{\a}} g}{g}\\
 		&=
 		\frac{1}{n^2} \int_\lambda^\Lambda W(n,\a)\,\dint\scalar{E_{\set{\a}} g}{g}	= \frac{1}{n^2} \scalar{W(n,P)g}{g}.
\sq
\end{align*} 

By the Spectral Theorem we have a representation of the error depending on the Markov operator $P$.
In this setting one can show a relation between the operator norm of $W(n,P)\colon L_2^0 \to L_2^0$ 
and the maximal error of $S_n$ for integrands $f$ which satisfy $\norm{f}{2}\leq1$.
This is stated in the next corollary.

\begin{coro}  \label{worst_stat_gen}
Let $(X_n)_{n\in\N}$ be a Markov chain with transition kernel $K$ and initial distribution $\pi$, let $K$ be 
reversible with respect to $\pi$ and suppose that $\Lambda<1$. Then 
\[
\sup_{\norm{f}{2}\leq1}e_\pi(S_n,f)^2 
= \frac{1}{n^2} \norm{W(n,P)}{L_2^0\to L_2^0} 
=\frac{1+\Lambda}{n(1-\Lambda)}-\frac{2\Lambda(1-\Lambda^n)}{n^2(1-\Lambda)^2} \leq \frac{2}{n(1-\Lambda)}.
\]
\end{coro}
\proof
The last inequality of the assertion follows by Lemma~\ref{lemm_mono_inc}.
The mapping $\alpha \mapsto W(n,\a)$ of Proposition~\ref{expl_stat_c} is increasing, 
see also Lemma~\ref{lemm_mono_inc}.
For $g=f-S(f)$ we have
\begin{align*}
e_\pi(S_n,f)^2	& =		
		\frac{1}{n^2}\int_\lambda^\Lambda W(n,\a)\,\dint\scalar{E_{\set{\a}} g}{g}
		\leq	\frac{1}{n^2} W(n,\Lambda) \int_\lambda^\Lambda \,\dint \scalar{E_{\set{\a}} g}{g}\\
		& = \frac{1}{n^2} W(n,\Lambda) \scalar{g}{g}
		= \left(\frac{1+\Lambda}{n(1-\Lambda)}-\frac{2\Lambda(1-\Lambda^n)}{n^2(1-\Lambda)^2}	\right) \norm{g}{2}^2.	
\end{align*}
The assertion is proven by 
\begin{align*}
W(n,\Lambda) & 	= \max_{\a \in \spec(P|L_2^0)} \abs{W(n,\a)} 
							= \norm{W(n,P)}{L_2^0\to L_2^0}  =\sup_{\norm{g}{2}\leq1,\;g\in L_2^0} \scalar{W(n,P)g}{g}\\
					 &	= \sup_{\norm{g}{2}\leq1,\;g\in L_2^0} n^2 \cdot e_\pi(S_n,g)^2 \leq n^2 \sup_{\norm{f}{2}\leq1}  e_\pi(S_n,f)^2.
\sq
\end{align*}
If the transition kernel $K$ is reversible with respect to $\pi$ and the Markov operator has an $L_2$-spectral gap, 
then
\[
\beta=\norm{P}{L_2^0\to L_2^0}=\max\{\Lambda,\abs{\lambda}\}<1.
\] 
Note that Proposition~\ref{expl_stat_c} holds already if $\Lambda<1$.
Hence an $L_2$-spectral gap 
is not necessary.
If the transition kernel $K$ is not reversible
but one has an $L_2$-spectral gap, then the following error bound can be shown.
\begin{prop}  \label{ohne_rev_stat}
Let $(X_n)_{n\in\N}$ be a Markov chain with transition kernel $K$ and initial distribution $\pi$.
Let $\pi$ be a stationary distribution of $K$. 
Let $f\in L_2$ and assume that there exists an $L_2$-spectral gap $1-\beta>0$.
Then
\begin{equation} \label{err_present_non_rev}
e_\pi(S_n,f)^2 \leq \frac{2}{n(1-\beta)} \norm{f}{2}^2.
\end{equation}
\end{prop}
\proof
 Let $g=f-S(f)$.
 The error obeys
 \begin{align*}
     & e_\pi(S_n,f)^2
     	=\frac{1}{n^2} \sum_{j=1}^n \expect_{\pi,K}[g(X_j)^2]
     	+ \frac{2}{n^2} \sum_{j=1}^{n-1} \sum_{i=j+1}^n \expect_{\pi,K}[g(X_j)g(X_i)].
 \end{align*}
 For $i,j\in \N$ with $j\leq i$ we have by the Cauchy-Schwarz inequality (CS) that  
 \begin{align*}
  	& \expect_{\pi,K}[g(X_i)g(X_j)]	
 			  	= \scalar{g}{P^{i-j}g}
 			\underset{\text{(CS)}}{\leq} \norm{P^{i-j}}{L_2^0 \to L_2^0} \norm{g}{2}^2.
\end{align*} 		
 Then, with $W(n,\beta)$ from Proposition~\ref{expl_stat_c} one has
 \[
  e_\pi(S_n,f)^2 \leq \frac{W(n,\beta)}{n^2} \norm{g}{2}^2 
  \underset{\eqref{mono_inc}}{\leq} \frac{2}{n(1-\beta)} \norm{f}{2}^2.
 \sq
 \]

The estimates of the error under the assumption that the initial distribution is a
stationary one seem to be restrictive.
If we could sample $\pi$ directly we would approximate $S(f)$ by Monte Carlo with an i.i.d. sample.
However, even if it is possible 
it might happen that the direct sampling procedure is computationally expensive, such that 
it is reasonable to generate only the initial state by sampling from $\pi$ and afterwards 
run a Markov chain with stationary distribution $\pi$.  

The error of a Markov chain Monte Carlo method with stationary initial distribution 
is related to the error with not necessarily stationary initial distribution.
\begin{prop} \label{connect_lem}
Let $r\in[1,2]$, let $f\in L_{2r}$ and let $\nu\in \mathcal{M}_{r/(r-1)}$ be a probability measure. 
Let $(X_n)_{n\in\N}$ be a Markov chain with transition kernel $K$ and initial distribution $\nu$ 
and let $\pi$ be a stationary distribution of $K$. 
Then
\begin{align}  \label{gen_con}
e_\nu(S_{n,n_0},f)^2 =  e_\pi(S_n,f)^2
+ \frac{1}{n^2}\sum_{j=1}^{n} L_{j+n_0}(g^2)
+ \frac{2}{n^2} \sum_{j=1}^{n-1} \sum_{k=j+1}^n L_{j+n_0}(gP^{k-j}g),
\end{align}
where $g=f-S(f)$ and
\[
L_i(h) = 	\scalar{(P^i-S)h}{(\frac{d\nu}{d\pi}-1)}, \quad h\in L_r,\,i\in\N.
\]
\end{prop}

\begin{proof}
The proof is adapted from \cite[Lemma~6, p.~17]{expl_error}. 
One has
\begin{align*}
& \expect_{\nu,K} \abs{S(f)-S_{n,n_0}(f)}^2
  =\frac{1}{n^2} \sum_{j=1}^n \sum_{i=1}^n \expect_{\nu,K} [g(X_{n_0+j})g(X_{n_0+i})]\\
& = \frac{1}{n^2} \sum_{j=1}^n \int_{D} P^{n_0+j}(g^2)(x)\, \nu(\dint x)
+ \frac{2}{n^2} \sum_{j=1}^{n-1} \sum_{k=j+1}^n \int_{D} P^{n_0+j}(g P^{k-j}g)(x)\,\nu(\dint x).
\end{align*}
For $h\in L_r$ and $\nu\in \mathcal{M}_{r/(r-1)}$ we have for all $i\in\N$ that $\frac{d\nu}{d\pi}\cdot P^i h$ 
is integrable with respect to $\pi$. Then
the following transformation holds true
\begin{align*} 
& \int_{D} (P^i h)(x)\, \nu(\dint x) 
	= \scalar{ P^i h }{ \frac{d\nu}{d\pi}}
	= \scalar{ P^i h }{1} + \scalar{ P^i h }{ (\frac{d\nu}{d\pi}-1)}\\
& = \scalar{ P^i h }{1} + \scalar{ P^i h }{ (\frac{d\nu}{d\pi}-1)}
	-	\underbrace{\scalar{ h }{ S(\frac{d\nu}{d\pi}-1)}}_{=0}\\
& = \scalar{ P^i h }{1} + \scalar{ (P^i-S) h }{ (\frac{d\nu}{d\pi}-1)}\\
& =	\int_{ D} (P^i h)(x)\pi(\dint x) 
		+ \scalar{ (P^i-S) h }{ (\frac{d\nu}{d\pi}-1)}.
\end{align*}
Formula \eqref{gen_con} is shown by using the previous calculation
for $h=g^2$ and $h=g P^{k-j} g$.
\end{proof}

Equation \eqref{gen_con} 
is still an exact error formula. 
The next lemma provides an estimate of the functional $L_k(\cdot)$ for $k\in\N$.

\begin{lemma} \label{L_est}
Let $r\in[1,2]$, $\nu\in \mathcal{M}_{r/(r-1)}$ and $h\in L_r$. 
Recall that $\beta=\norm{P}{L_2^0 \to L_2^0}$. 
\begin{enumerate}[(i)] 
\item \label{spec_L_k}
If $r\in (1,2]$, then
\begin{align}
  \label{q_r_gen}
	\abs{L_k(h)} & 
		\leq 2^{2/r}\beta^{2k\frac{r-1}{r}} \norm{\frac{d\nu}{d\pi}-1}{\frac{r}{r-1}}   \norm{h}{r},\quad k\in\N.
\end{align}
\item \label{zwei_L1}
If $r=1$ and the transition kernel is $L_1$-exponentially convergent with $(\a,M)$, then
\begin{equation}  \label{L_uniform}
	\abs{L_k(h)}\leq M \a^k \norm{\frac{d\nu}{d\pi}-1}{\infty} \norm{h}{1},\quad k\in\N.
\end{equation}

\end{enumerate}
\end{lemma}
\begin{proof}
After applying H\"older's inequality (HI) with conjugate parameter $r$ and $s=\frac{r}{r-1}$ 
to $L_k(h)=\scalar{(P^k-S)h}{(\frac{d\nu}{d\pi}-1)}$ one has
\[
  \abs{L_k(h)} \underset{\text{(HI)}}{\leq} 
  		\norm{(P^k-S)h}{r}\norm{\frac{d\nu}{d\pi}-1}{s}
  		\leq \norm{P^k-S}{L_r\to L_r} \norm{\frac{d\nu}{d\pi}-1}{s} \norm{h}{r}.
\]
By equation \eqref{norm_Lp_gen} the claim of \eqref{spec_L_k} is proven and by the $L_1$-exponential convergence 
the inequality of $\eqref{zwei_L1}$ holds.
\end{proof}
Note that if $r=2$ then one has 
	$ \abs{L_k(h)}\leq \beta^k \norm{\frac{d\nu}{d\pi}-1}{2} \norm{h}{2} $, see \eqref{P_S_2_ident_gen}.
This is by a factor of two better than \eqref{q_r_gen}, but not essentially different.\\

In Lemma~\ref{L_est} we have seen that 
under suitable assumptions one can ensure
an exponential decay of $L_k(\cdot)$ for increasing $k\in\N$.
This fact is used to show for reversible Markov chains which are
$L_1$-exponentially ergodic with $(\a,M)$ that there exists a constant $C_{\nu,\a,M}$, 
which is independent of $n$ and $n_0$, such that
\[
\abs{e_\nu(S_{n,n_0},f)^2 - e_\pi(S_n,f)^2} \leq C_{\nu,\a,M} \norm{f}{2}^2 \frac{\a^{n_0}}{n^2}.
\]
An immediate consequence of the inequality is an explicit error bound.
The following lemma and remark imply such an inequality and provide $C_{\nu,\a,M}$ explicitly.

\begin{lemma} \label{err_thm_uni}
Let $(X_{n})_{n\in\N}$ be a 
Markov chain with transition kernel $K$ and initial distribution $\nu$, where $\nu\in \mathcal{M}_\infty$. 
Let $K$ be reversible with respect to $\pi$ 
and $L_1$-exponentially convergent with $(\a,M)$.
Let $f\in L_2$ 
and
\begin{equation*}
U(\a,n) = \sum_{j=1}^n   \a^{j} +  2 \sum_{j=1}^{n-1} \sum_{k=j+1}^n \a^{k}.
\end{equation*}
Then
\begin{equation} \label{L1_lemm}
\abs{e_\nu(S_{n,n_0},f)^2-e_\pi(S_n,f)^2} \leq \frac{U(\a,n)}{n^2}\;  
		M \norm{\frac{d\nu}{d\pi}-1}{\infty} \a^{n_0} \norm{f}{2}^2.
\end{equation}
\end{lemma}

\begin{proof}
	Let $g=f-S(f)$.
	The equation \eqref{gen_con} implies
	\begin{align*}
	\abs{e_\nu(S_{n,n_0},f)^2-e_\pi(S_n,f)^2} \leq 
 \frac{1}{n^2}\sum_{j=1}^{n} \abs{ L_{j+n_0}(g^2)}
+ \frac{2}{n^2} \sum_{j=1}^{n-1} \sum_{k=j+1}^n \abs{ L_{j+n_0}(gP^{k-j}g)}.
\end{align*}	
By \eqref{L_uniform} of Lemma~\ref{L_est} one obtains
 \begin{align*}
 		\abs{L_{j+n_0}(g^2)} & 
 		\leq  M  \a^{j+n_0} \norm{\frac{d\nu}{d\pi}-1}{\infty}\;   \norm{g}{2}^2,\\
 		\abs{L_{j+n_0}(gP^{k-j}g)}& 
 		\leq  M  \a^{j+n_0} \norm{\frac{d\nu}{d\pi}-1}{\infty}\;   \norm{g P^{k-j} g}{1}.
 \end{align*}
By the reversibility and $L_1$-exponential convergence of $K$ we get 
from Proposition~\ref{uni_impl_geo} that $\beta=\norm{P}{L_2^0 \to L_2^0}\leq \a$.
Then by applying the Cauchy-Schwarz inequality (CS) 
one has
\[
\norm{g P^{k-j}g }{1} \underset{\text{(CS)}}{\leq} \norm{g}{2} \norm{P^{k-j}g}{2}
		\leq \norm{g}{2}^2 \norm{P^{k-j}}{L_2^0\to L_2^0} 
		\leq  \a^{k-j} \norm{g}{2}^2.
\]
Let $\e_0= \a^{n_0} M \norm{\frac{d\nu}{d\pi}-1}{\infty}$. Then
  \begin{align*}
 		\sum_{j=1}^n& \abs{L_{j+n_0}(g^2)} +
 		2\sum_{j=1}^{n-1} \sum_{k=j+1}^n \abs{L_{j+n_0}(gP^{k-j}g)}\\
	&\leq  \e_0 \norm{g}{2}^2  \sum_{j=1}^n   \a^{j}
 				+  2 \e_0 \norm{g}{2}^2 \sum_{j=1}^{n-1} \sum_{k=j+1}^n \a^{k}\\
 	    &=  \e_0 \norm{g}{2}^2 \left( \sum_{j=1}^n   \a^{j} +  2 \sum_{j=1}^{n-1}  \sum_{k=j+1}^n \a^{k} \right)\\
	    &=  \e_0 \cdot U(\a,n) \cdot  \norm{g}{2}^2
 		 \leq  \e_0 \cdot U(\a,n) \cdot  \norm{f}{2}^2.
 \end{align*}
Thus the proof is completed.
\end{proof}

\begin{remark}
The function $U(\a,n)$ is already studied in Lemma~\ref{lemma_U_fin}.
Let us repeat the result.
For all $n\in \N$ we have
\[
U(\a,n)\leq \frac{2}{(1-\a)^2}.
\] 
Then, from Lemma~\ref{err_thm_uni} it follows that
\[
e_\nu(S_{n,n_0},f)^2 \leq e_\pi(S_n,f)^2
		+  \frac{  2M \norm{\frac{d\nu}{d\pi}-1}{\infty} \a^{n_0} }{n^2  (1-\a)^2}\norm{f}{2}^2	.
\]
If the initial distribution $\nu$ is $\pi$ then one has the error formula of  
Proposition~\ref{expl_stat_c}. 
\end{remark}

\begin{remark}  \label{normal_op}
Note that in Lemma~\ref{err_thm_uni} reversibility of $K$ was essentially 
used to apply Proposition~\ref{uni_impl_geo}.
If the Markov operator is normal, i.e. $PP^*=P^*P$, then one has by Proposition~\ref{L_1conv_spec}
that $\beta=\norm{P}{L_2^0 \to L_2^0} \leq \sqrt{\a}$.
By this observation we get a very similar estimate as in Lemma~\ref{err_thm_uni} 
for normal Markov operators which are not necessarily reversible.
The only difference to \eqref{L1_lemm} is that $\a$ has to be substituted by $\sqrt{\a}$.
Then 
\[
U(\sqrt{\a},n) 
\leq \frac{2}{(1-\sqrt{\a})^2}
\leq \frac{8}{(1-\a)^2}.
\]
The last inequality is implied by $1-\a^{r}\geq r(1-\a)$ for $r\in[0,1]$ 
which is a conclusion of the \emph{Bernoulli inequality with real exponent}\footnote{
The Bernoulli inequality with real exponent $r\in[0,1]$ states
for any real number $x>-1$ that $ (1+x)^r \leq 1+rx $.
}
. 
\end{remark}

The next theorem summarizes the main result for a Markov chain with a 
reversible and $L_1$-exponentially 
convergent transition kernel.

\begin{theorem} \label{main_thm_unif}

Let $(X_n)_{n\in\N}$ be a Markov chain with transition kernel $K$ and initial distribution $\nu$.
Let $K$ be reversible with respect to $\pi$ and $L_1$-exponentially convergent with $(\a,M)$.
Let $f\in L_2$ and assume that the probability measure $\nu\in \mathcal{M}_\infty$. 
%
	Then
\begin{align} \label{uni_b}
e_\nu(S_{n,n_0},f)^2 	& \leq \frac{2}{n(1-\Lambda)}\norm{f}{2}^2 
  			+\frac{  2M \norm{\frac{d\nu}{d\pi}-1}{\infty} \a^{n_0} }{n^2  (1-\a)^2}\norm{f}{2}^2	
	 \end{align}
and for $g=f-S(f)$ we have	 
\begin{align}\label{asymp_gen_unif} 
  \lim_{n\to \infty} n\cdot e_\nu(S_{n,n_0},f)^2 &=\lim_{n\to \infty} n\cdot e_\pi(S_{n},f)^2 	
  =\scalar{(I+P)(I-P)^{-1}g}{g}.
\end{align}
\end{theorem}

\begin{proof}
By Lemma~\ref{err_thm_uni} and Lemma~\ref{lemma_U_fin} the first equality of \eqref{asymp_gen_unif} holds true. 
By the reversibility of the transition kernel Proposition~\ref{expl_stat_c} applies, so that
\[
\lim_{n\to \infty} n\cdot e_\pi(S_{n},f)^2 = \lim_{n\to \infty} \frac{1}{n} \scalar{W(n,P)g}{g} = \scalar{(I+P)(I-P)^{-1}g}{g}.
\]
The rest follows via Lemma~\ref{err_thm_uni}, Corollary~\ref{worst_stat_gen} and Lemma~\ref{lemma_U_fin}.
\end{proof}

\begin{remark}
Under the assumptions of Theorem~\ref{main_thm_unif} one has by Proposition~\ref{uni_impl_geo} that 
$\pi$-a.e. uniform ergodicity with $(\a,\widetilde{M})$ 
is equivalent to $L_1$-exponential convergence with $(\a,2 \widetilde{M})$. 
Hence one can restate Theorem~\ref{main_thm_unif} 
for uniformly ergodic Markov chains 
and obtains the same result with $M=2\widetilde{M}$. 
This is the general state space counterpart to Theorem~\ref{main_fin}, 
where $\widetilde{M}$ is of the magnitude of $\norm{\frac{1}{\pi}}{\infty}$ 
and $\beta=\a$.\\
Furthermore note that
if the Markov operator is normal and not necessarily reversible, 
then one can get a similar
error bound by using 
Remark~\ref{normal_op}.
\end{remark}

\begin{remark}
  The error bound of \eqref{uni_b} might be interpreted as follows:
  The burn-in $n_0$ is reasonable to eliminate the influence of the initial distribution, 
  while $n$ has to decrease $e_\pi(S_n,f)$.
  For large $n$ the error behaves exactly as the error 
	where one started by the stationary distribution. 
	Hence the bias of the initial distribution disappears after sufficiently many steps.
  If the initial distribution falls together with the stationary one, 
  then the bias of the initial part vanishes completely.
\end{remark}  

Another consequence of Lemma~\ref{err_thm_uni} and Lemma~\ref{lemma_U_fin} is 
the following result concerning the asymptotic error for $\norm{f}{2}\leq1$.

\begin{coro}  \label{asymp_err_coro_gen}
Under the same assumptions as in Theorem~\ref{main_thm_unif} it follows that
  \[
    \lim_{n\to \infty} n\cdot \sup_{\norm{f}{2}\leq 1} e_{\nu}(S_{n,n_0},f)^2 
    = \frac{1+\Lambda}{1-\Lambda}
  \]
  and
  \[
    \lim_{n_0 \to \infty}  \sup_{\norm{f}{2}\leq 1} e_{\nu}(S_{n,n_0},f)^2 
    = 	\frac{1+\Lambda}{n(1-\Lambda)}-\frac{2\Lambda(1-\Lambda^n)}{n^2(1-\Lambda)^2}.
  \]
\end{coro}
\begin{proof}
 Let us define
  \[
    c_{n,n_0}=\frac{2\a^{n_0} M \norm{\frac{d\nu}{d\pi}-1}{\infty}}{n^2(1-\a)^2}.
  \]
  One has $\lim_{n\to \infty} n\cdot c_{n,n_0}=0$ and $\lim_{n_0\to \infty} c_{n,n_0}=0$.
  For $\norm{f}{2}\leq1$ we obtain by Lemma~\ref{err_thm_uni} and Lemma~\ref{lemma_U_fin} that
  \[
    \abs{e_\nu(S_{n,n_0},f)^2-e_\pi(S_{n},f)^2} \leq c_{n,n_0}.
  \] 
  Hence
\begin{equation}  \label{eq_low_gen}
  \sup_{\norm{f}{2}\leq 1}e_\pi(S_{n},f)^2 - c_{n,n_0} 
  \leq \sup_{\norm{f}{2}\leq 1}  e_\nu(S_{n,n_0},f)^2 
\leq \sup_{\norm{f}{2}\leq 1} e_\pi(S_{n},f)^2 + c_{n,n_0}.
\end{equation}
Recall that $\Lambda= \sup\set{ \a\mid \a \in \spec(P|L_2^0) }$. Then by Corollary~\ref{worst_stat_gen}  
we have
\[
\sup_{\norm{f}{2}\leq1}e_\pi(S_n,f)^2
= 	\frac{1+\Lambda}{n(1-\Lambda)}-\frac{2\Lambda(1-\Lambda^n)}{n^2(1-\Lambda)^2} .
\]
By taking the limits in \eqref{eq_low_gen} the assertions are proven.
\end{proof}

In many examples it is known that the transition kernel is $L_1$-exponentially convergent or $\pi$-a.e. uniformly ergodic,
but it is difficult to obtain reasonable values of $(\a,M)$ explicitly.
  Then at least the asymptotic result can be used. 
  This is similar to results of \cite{sokal,bremaud,mathe1}.

\begin{remark}  \label{lower_bound_gen}
Observe that we have a lower and an upper bound of the error of $S_{n,n_0}$.
Exactly as in Remark~\ref{rem_low_bound} one obtains by \eqref{eq_low_gen} that
\[
\frac{1+\Lambda}{n(1-\Lambda)}-\frac{2}{n^2(1-\Lambda)^2}-c_{n,n_0}
\leq \sup_{\norm{f}{2}\leq1}	e_\nu(S_{n,n_0},f)^2
\leq \frac{2}{n(1-\Lambda)} +c_{n,n_0}.
\]
\end{remark}

We showed an error bound of $S_{n,n_0}$ with respect to $\norm{\cdot}{2}$ 
for Markov chains which are reversible and
$L_1$-exponentially convergent.
The condition of the $L_1$-exponential convergence is rather restrictive.
This motivates the study of Markov chains which satisfy a weaker convergence property,
namely we assume that there is an $L_2$-spectral gap, i.e. $1-\beta>0$.
This is enough to obtain error bounds for integrands $f\in L_p$ with $p\in(2,\infty]$. 
The following lemmas lead to the fact that 
there exists a constant $C_{\nu,\beta,p}$, independent of $n_0$ and $n$, such that
\[
\abs{e_\nu(S_{n,n_0},f)^2 - e_\pi(S_n,f)^2} \leq C_{\nu,\beta,p} \norm{f}{p}^2 \frac{\beta^{n_0}}{n^2}.
\]
Note that it is not assumed that the Markov chain is reversible with respect to $\pi$.

\begin{lemma} \label{err_thm_gen}
Let $(X_n)_{n\in\N}$ be a Markov chain with transition kernel $K$ and initial distribution $\nu$.
Let $\pi$ be a stationary distribution of $K$.
Let $f\in L_p$, let $\nu\in \mathcal{M}_{\max\set{2,\frac{p}{p-2}}}$ with $p\in(2,\infty]$ and
\begin{equation*}
V(\beta,n,p) = 4
\begin{cases}
 2^{4/p} \sum_{j=1}^n\beta^{2j\frac{p-2}{p}}
 				+ 2^{\frac{3p+2}{p}}\sum_{j=1}^{n-1}\beta^{2j\frac{p-3}{p}}\sum_{k=j+1}^n \beta^{2k/p}, 
 						& p\in(2,4),\\
 2\sum_{j=1}^n  \beta^j+ 2^{\frac{3p+2}{p}}  \sum_{j=1}^{n-1} \beta^{2j/p} \sum_{k=j+1}^n \beta^{k\frac{p-2}{p}}, 
 						& p\in[4,\infty].
\end{cases}
\end{equation*}
\begin{enumerate}[(i)]
	\item \label{err_L_p} 
		If $p\in(2,4)$, then
		 	\begin{align*}
	 			\abs{e_\nu(S_{n,n_0},f)^2-e_\pi(S_n,f)^2} \leq 
	 			\frac{V(\beta,n,p)}{n^2}\, \beta^{2n_0\frac{p-2}{p}} \norm{\frac{d\nu}{d\pi}-1}{\frac{p}{p-2}}\norm{f}{p}^2.
	 		\end{align*}
	\item \label{err_L_gro_4}
	 	If $p\in[4,\infty]$, then 
	 		\begin{align*}
	 			\abs{e_\nu(S_{n,n_0},f)^2-e_\pi(S_n,f)^2} \leq 
	 			\frac{V(\beta,n,p)}{n^2}\,\beta^{n_0} \norm{\frac{d\nu}{d\pi}-1}{2} \norm{f}{p }^2.
	 		\end{align*}
\end{enumerate}
\end{lemma}
\begin{proof}
	First, let $g=f-S(f)$ and observe that for $p\geq1$ 
	one obtains	
	\begin{equation} \label{rel_g_f}
			\norm{g}{p}\leq \norm{f}{p}+\abs{S(f)} \leq \norm{f}{p}+\norm{f}{1} \leq 2\norm{f}{p}.
	\end{equation}
	The equation \eqref{gen_con} implies
	\begin{align} \label{exact_error_a}  
	\abs{e_\nu(S_{n,n_0},f)^2-e_\pi(S_n,f)^2} \leq
 \frac{1}{n^2}\sum_{j=1}^{n} \abs{ L_{j+n_0}(g^2)}
+ \frac{2}{n^2} \sum_{j=1}^{n-1} \sum_{k=j+1}^n \abs{ L_{j+n_0}(gP^{k-j}g)}.
\end{align}	
Let $p\in(2,4)$. Then it follows by \eqref{q_r_gen} with $r=\frac{p}{2}$ and $r/(r-1)=\frac{p}{p-2}$ that
 \begin{align*}
 		\abs{L_{j+n_0}(g^2)} & 
 		\leq  2 ^{4/p} \beta^{2j\frac{p-2}{p}}\;  \beta^{2n_0\frac{p-2}{p}}
		    \norm{\frac{d\nu}{d\pi}-1}{\frac{p}{p-2}}\;   \norm{g}{p}^2,\\
 		\abs{L_{j+n_0}(gP^{k-j}g)}& 
 		\leq  2 ^{4/p} \beta^{2j\frac{p-2}{p}}\;\beta^{2n_0\frac{p-2}{p}} 
		     \norm{\frac{d\nu}{d\pi}-1}{\frac{p}{p-2}}\;   \norm{g P^{k-j} g}{p/2}.
 \end{align*}
By applying the Cauchy-Schwarz inequality (CS) and \eqref{norm_Lp_gen} one obtains
\[
\norm{g P^{k-j}g }{p/2} \underset{\text{(CS)}}{\leq} \norm{g}{p} \norm{P^{k-j}g}{p}
	\leq \norm{g}{p}^2 \norm{P^{k-j}}{L_p^0\to L_p^0} 
	\underset{\eqref{norm_Lp_gen}}{\leq}  2^{2\frac{p-1}{p}}\beta^{2\frac{k-j}{p}} \norm{g}{p}^2.
\]
Let $\e_0(p)= \beta^{2n_0\frac{p-2}{p}} \norm{\frac{d\nu}{d\pi}-1}{\frac{p}{p-2}}$.  Then
  \begin{align*}
 		\sum_{j=1}^n& \abs{L_{j+n_0}(g^2)} +
 		2\sum_{j=1}^{n-1} \sum_{k=j+1}^n \abs{L_{j+n_0}(gP^{k-j}g)}\\
 		&\leq 2^{4/p} \e_0(p) \norm{g}{p}^2  \sum_{j=1}^n   \beta^{2j\frac{p-2}{p}}
 		+  2^{\frac{3p+2}{p}} \e_0(p) \norm{g}{p}^2 \sum_{j=1}^{n-1} \beta^{2j\frac{p-3}{p}} \sum_{k=j+1}^n \beta^{2k/p}\\
 		 &=  \e_0(p) \norm{g}{p}^2 \left( 2^{4/p} \sum_{j=1}^n   \beta^{2j\frac{p-2}{p}} +  2^{\frac{3p+2}{p}} \sum_{j=1}^{n-1} \beta^{2j\frac{p-3}{p}} \sum_{k=j+1}^n \beta^{2k/p} \right)\\
 		 & \underset{\eqref{rel_g_f}}{\leq } V(\beta,n,p) \cdot \e_0(p) \norm{f}{p}^2.
 \end{align*}
Thus, claim \eqref{err_L_p} is shown.\\
Let us turn to \eqref{err_L_gro_4}, i.e. $p\in[4,\infty]$. Equation \eqref{q_r_gen} with $r=2$ is used to get 
 \begin{align*}
 		\abs{L_{j+n_0}(g^2)} & 
 		\leq  2\beta^{j+n_0} \norm{\frac{d\nu}{d\pi}-1}{2}\; \norm{g}{4}^2,\\
 		\abs{L_{j+n_0}(gP^{k-j}g)}& 
 		\leq  2\beta^{j+n_0} \norm{\frac{d\nu}{d\pi}-1}{2}\; \norm{gP^{k-j}g}{2}.
 \end{align*}
By H\"older's inequality (HI) with conjugate parameters $\frac{p}{2}$ and $\frac{p}{p-2}$ one obtains 
\begin{align*}
\norm{gP^{k-j}g}{2} & \underset{\text{(HI)}}{\leq} \norm{g}{p} \norm{P^{k-j}g}{\frac{2p}{p-2}} 
		\leq \norm{P^{k-j}}{L_{2p/(p-2)}^0\to L_{2p/(p-2)}^0} \norm{g}{p} \norm{g}{\frac{2p}{p-2}} \\
	&	\leq \norm{P^{k-j}}{L_{2p/(p-2)}^0\to L_{2p/(p-2)}^0} \norm{g}{p}^2
 \underset{\eqref{norm_Lp_gen}}{\leq} 2^{\frac{p+2}{p}} \beta^{(k-j)\frac{p-2}{p}} \norm{g}{p}^2.
\end{align*}
Note that in the third inequality of the last estimation it was essential that $p\in[4,\infty]$ 
for using
$
\norm{g}{\frac{2p}{p-2}} \leq \norm{g}{p}.
$
Thus, for $\e_0=\beta^{n_0} \norm{\frac{d\nu}{d\pi}-1}{2}$ one has
\begin{align*}
 		\sum_{j=1}^n& \abs{L_{j+n_0}(g^2)} +
 		2\sum_{j=1}^{n-1} \sum_{k=j+1}^n \abs{L_{j+n_0}(gP^{k-j}g)}\\
 		&\leq \e_0 \norm{g}{4}^2  2\sum_{j=1}^n\beta^j 
 					+ \e_0 \norm{g}{p}^2 2^{2+\frac{p+2}{p}}  \sum_{j=1}^{n-1} \beta^{2j/p} \sum_{k=j+1}^n \beta^{k\frac{p-2}{p}}\\
 		&\leq  \e_0 \norm{g}{p}^2 \left( 2\sum_{j=1}^n\beta^j+ 2^{\frac{3p+2}{p}}  \sum_{j=1}^{n-1} \beta^{2j/p} \sum_{k=j+1}^n \beta^{k\frac{p-2}{p}} \right)
 		\underset{\eqref{rel_g_f}}{\leq} V(\beta,n,p) \cdot \e_0 \norm{f}{p}^2. 
\end{align*}
Finally by substituting this in equation \eqref{exact_error_a} everything is shown.
\end{proof}
Let us consider $V(\beta,n,p)$.
If $p\in(2,\infty]$ and $1-\beta>0$, then we show that the mapping $n\mapsto V(\beta,n,p)$ is bounded.
\begin{lemma}  \label{lemma_V_gen}
Let $p\in(2,\infty]$ and $1-\beta>0$.
For all $n\in \N$ we obtain
\begin{equation} \label{bound_V}
V(\beta,n,p)\leq \frac{64p}{(p-2)(1-\beta)^2}.
\end{equation}
%
%


%
%

\end{lemma}
\begin{proof}
The inequalities indicated by $(\star)$ follow from $1-\beta^r\geq r(1-\beta)$ for $r\in[0,1]$.
 First, let $p\in(2,4)$. By the geometric series one can estimate
\begin{align*}
\frac{V(\beta,n,p)}{4} 
& =2^{4/p} \sum_{j=1}^n\beta^{2j\frac{p-2}{p}}
 	+ 2^{\frac{3p+2}{p}}\sum_{j=1}^{n-1}\beta^{2j\frac{p-3}{p}}\sum_{k=j+1}^n \beta^{2k/p}\\
& =2^{4/p} \sum_{j=1}^n\beta^{2j\frac{p-2}{p}}
 	+ 2^{\frac{3p+2}{p}}\beta^{2/p}\sum_{j=1}^{n-1}\beta^{2j\frac{p-2}{p}}\sum_{k=0}^{n-j-1} \beta^{2k/p}\\
& \leq 2^{4/p} \sum_{j=1}^n\beta^{2j\frac{p-2}{p}}
 	+ \frac{2^{\frac{3p+2}{p}}\beta^{2/p}}{1-\beta^{2/p}}\sum_{j=1}^{n-1}\beta^{2j\frac{p-2}{p}}\\
\displaybreak
& \leq \left(\frac{2^{4/p}+\beta^{2/p} 2^{4/p}(2^{3-2/p}-1)}{1-\beta^{2/p}}\right) \sum_{j=1}^n\beta^{2j\frac{p-2}{p}}
  \leq \frac{2^{3+2/p}}{1-\beta^{2/p}} \sum_{j=1}^n\beta^{2j\frac{p-2}{p}}\\
& \hspace{-1.8ex}\underset{p\in(2,4)}{\leq} \frac{16}{(1-\beta^{2/p})(1-\beta^{2(p-2)/p})} 
	\underset{(\star)}{\leq} \frac{4p^2}{(p-2)(1-\beta)^2} \underset{p\in(2,4)}{\leq} \frac{16p}{(p-2)(1-\beta)^2}.
\end{align*} 
For $p\in[4,\infty]$, again by the geometric series, we can estimate
\begin{align*}
\frac{V(\beta,n,p)}{4} 
& =2 \sum_{j=1}^n\beta^j+ 2^{\frac{3p+2}{p}}  \sum_{j=1}^{n-1} \beta^{2j/p} \sum_{k=j+1}^n \beta^{k\frac{p-2}{p}} \\
& =2 \sum_{j=1}^n\beta^j+ 2^{\frac{3p+2}{p}} \beta^{\frac{p-2}{p}} \sum_{j=1}^{n-1} \beta^{j} 
						\sum_{k=0}^{n-j-1} \beta^{k\frac{p-2}{p}}\\
& \leq  \left(	2+\frac{2^{\frac{3p+2}{p}}\beta^{\frac{p-2}{p}}}{1-\beta^{\frac{p-2}{p}}}
 	\right)\sum_{j=1}^n\beta^j
  \leq  \left(	\frac{2+\beta^{\frac{p-2}{p}}(2^{\frac{3p+2}{p}}-2)}{1-\beta^{\frac{p-2}{p}}}
 	\right)\sum_{j=1}^n\beta^j	
 	\\
& \hspace{-1.8ex}\underset{p\in[4,\infty]}{\leq} \frac{8\sqrt{2}}{1-\beta^{\frac{p-2}{p}}} \sum_{j=1}^n\beta^j
	\leq \frac{8\sqrt{2}}{(1-\beta)(1-\beta^{\frac{p-2}{p}})}	
	\underset{(\star)}{\leq} \frac{8\sqrt{2}p}{(p-2)(1-\beta)^2}.	
\end{align*} 
This completes the proof. 
\end{proof}

The main error bound of $S_{n,n_0}$ for Markov chains with an $L_2$-spectral gap is presented in the next theorem.

\begin{theorem} \label{main_thm}
Let $(X_{n})_{n\in\N}$ be a Markov chain with transition kernel $K$ and initial distribution $\nu$.
Let $\pi$ be a stationary distribution of $K$. 
For $p\in(2,\infty]$ let $f\in L_p$ and $\nu\in \mathcal{M}_{\max\set{2,\frac{p}{p-2}}}$. 
Suppose that the Markov operator has an $L_2$-spectral gap, i.e. $1-\beta>0$. 
Then we have

\begin{align*}
e_\nu(S_{n,n_0},f)^2 	& \leq e_\pi(S_n,f)^2 
	+\frac{ 64 p\norm{f}{p}^2 }{n^2 (p-2) (1-\beta)^2}	
  	\begin{cases}
  	\beta^{2n_0\frac{(p-2)}{p}} \norm{\frac{d\nu}{d\pi}-1}{\frac{p}{p-2}}, & p\in(2,4),\\
  	\beta^{n_0} \norm{\frac{d\nu}{d\pi}-1}{2}, & p\in[4,\infty],
	\end{cases}
\end{align*}
where
\[
e_\pi(S_n,f)^2 \leq 
	       \begin{cases}
		\frac{2}{n(1-\Lambda)}\norm{f}{p}, & \text{if $K$ is reversible with respect to }\pi,
				  \\
		\frac{2}{n(1-\beta)}\norm{f}{p}, & \text{otherwise}.
		\end{cases}  
\]
Furthermore
\begin{align}\label{asymp_gen} 
 \lim_{n\to \infty} n\cdot e_\nu(S_{n,n_0},f)^2 &=\lim_{n\to \infty} n\cdot e_\pi(S_{n},f)^2 
\end{align}
 and if $K$ is reversible with respect to $\pi$ then \eqref{asymp_gen} is equal to
\[
\scalar{(I+P)(I-P)^{-1}g}{g}, \quad \text{where} \quad g=f-S(f).
\]
\end{theorem}
\begin{proof}
By Lemma~\ref{err_thm_gen} and Lemma~\ref{lemma_V_gen} the equality of \eqref{asymp_gen} is true. 
If the transition kernel is reversible, then by Proposition~\ref{expl_stat_c} the asymptotic result holds since
\[
\lim_{n\to \infty} n\cdot e_\pi(S_{n},f)^2 = \lim_{n\to \infty} \frac{1}{n} \scalar{W(n,P)g}{g} = \scalar{(I+P)(I-P)^{-1}g}{g}.
\]

By Lemma~\ref{err_thm_gen} and Lemma~\ref{lemma_V_gen} one obtains the estimate of $e_\nu(S_{n,n_0},f)^2$.
The estimate of
$e_\pi(S_n,f)^2$ follows by Proposition~\ref{ohne_rev_stat} and 
for a reversible transition kernel by Corollary~\ref{worst_stat_gen}.
\end{proof} 

\begin{remark}
A large burn-in $n_0$ guarantees that the influence of the initial distribution disappears and a large $n$ makes 
$e_\pi(S_n,f)$ small.
The condition of the $L_1$-exponential convergence could be substituted 
by the existence of an $L_2$-spectral gap by paying the price of considering error bounds 
in terms of $L_p$-norms of the integrand for $p\in(2,\infty]$. 
If $p$ converges to $2$, then the bound goes to infinity. 
However, for $p>2$ one has an explicit error bound. 
If the initial and stationary distribution is the same, 
then the influence of the initial part vanishes for all $p\in(2,\infty]$. 
\end{remark} 
\begin{remark}
Let
\[
c_{n,n_0}(p)= \frac{64p}{n^2(p-2)(1-\beta)^2} 
	      \begin{cases}
		\beta^{2n_0\frac{(p-2)}{p}} \norm{\frac{d\nu}{d\pi}-1}{\frac{p}{p-2}}, &	p\in(2,4),\\
		\beta^{n_0} \norm{\frac{d\nu}{d\pi}-1}{2},			&	p\in[4,\infty].
	      \end{cases}
\]
For $\norm{f}{p}\leq 1$ we have by Lemma~\ref{err_thm_gen} and Lemma~\ref{lemma_V_gen} that 
\begin{equation*}  
\abs{e_\nu(S_{n,n_0},f)^2-e_\pi(S_n,f)^2} \leq c_{n,n_0}(p). 
\end{equation*}
Observe that this implies a lower error bound for $S_{n,n_0}$.
We do not use it because of the lack of a general lower bound of $\sup_{\norm{f}{p}\leq1} e_\pi(S_{n},f)^2$ for $p\in(2,\infty]$. 
\end{remark}

\begin{remark}
Let $K$ be a transition kernel which is reversible with respect to $\pi$. We use
the notation $\beta_K=\beta$ and $\Lambda_K=\Lambda$ to indicate the transition kernel.
The lazy version of $K$ is given by $\widetilde{K}$. Then one has
\[
  \beta_{\widetilde{K}} = \Lambda_{\widetilde{K}} 
= \frac{1}{2}(1+\Lambda_K).
\]
If one has an estimate of $\Lambda_K$, 
then one also has an estimate of $\beta_{\widetilde{K}}$ and one can apply Theorem~\ref{main_thm}.
There are some techniques, e.g. canonical paths (see \cite{can_path_gen}) and the conductance concept 
(see \cite{lawler_sokal,lova_simo1} and \cite{jerrum,dia_stroock_eigen}) 
which are helpful to estimate $\Lambda_K$.
However, in general it is a challenging task.

\end{remark}

\section{Burn-in} \label{burn_in_sec}
Assume that computational resources for $N=n+n_0$ steps of the Markov chain are available.
The burn-in $n_0$ and the sample size $n$ should be chosen such that the 
error bound is as small as possible. One encounters the same trade-off as for finite state spaces.
In the next statement the error bound for an explicit burn-in is stated.

\begin{theorem}{\ \\[-4ex]}
\label{main_coro_gen}
\begin{enumerate}[(i)]
\item \label{erstens_burn_in}
Suppose that we have a Markov chain which is reversible with respect to $\pi$ 
and $L_1$-exponentially convergent with $(\a,M)$. Let
\[
n_0= \max\set{\left\lceil \frac{\log(M \norm{\frac{d\nu}{d\pi}-1}{\infty})}{\log(\a^{-1})}\right\rceil,0}.
\]
Then 
\begin{align*}
\sup_{\norm{f}{2}\leq1}e_\nu(S_{n,n_0},f)^2
		& \leq \frac{2}{n(1-\beta)} 
				+				 
				 \frac{2}{n^2(1-\a)^2} \\
		& 	\label{alpha_est}
		    \leq \frac{2}{n(1-\a)} 
				+				 
				 \frac{2}{n^2(1-\a)^2}.
\end{align*}
\item	\label{zweitens_burn_in}
  Suppose that we have a Markov chain with Markov operator $P$ which has an $L_2$-spectral gap $1-\beta>0$.
  For $p\in(2,\infty]$ let $n_0(p)$ be the smallest natural number (including zero) 
  which is greater than or equal to
  \[
    \frac{1}{\log(\beta^{-1})} 
    \begin{cases}
	\frac{p}{2(p-2)}\log\left(\frac{32p}{p-2} \norm{\frac{d\nu}{d\pi}-1}{\frac{p}{p-2}}\right), & p\in(2,4),\\
	 \log\left(64 \norm{\frac{d\nu}{d\pi}-1}{2}\right) , & p\in[4,\infty]. 
    \end{cases}
  \]
  Then
  \[
    \sup_{\norm{f}{p}\leq1}e_\nu(S_{n,n_0(p)},f)^2
	\leq \frac{2}{n(1-\beta)} +\frac{2}{n^2(1-\beta)^2}.
  \]
\end{enumerate}
\end{theorem}

\begin{proof}
Assertion \eqref{erstens_burn_in} follows from Theorem~\ref{main_thm_unif} and Proposition~\ref{uni_impl_geo}. 
Claim \eqref{zweitens_burn_in} 
is an application of 
Theorem~\ref{main_thm}.
\end{proof}
Note that $\log(\beta^{-1})=(1-\beta)+\sum_{j=2}^\infty \frac{(1-\beta)^j}{j !}$ and $\log(\beta^{-1})\geq1-\beta$.
This can be used to estimate the suggestion of the burn-in.
Now we justify the choice of the burn-in.
  
For simplicity we assume that $\a=\beta$.
Let us define
\[
C(p)=
\begin{cases}
	M  \norm{\frac{d\nu}{d\pi}-1}{\infty}, 				& p=2,\\
	\frac{32p}{p-2}  \norm{\frac{d\nu}{d\pi}-1}{\frac{p}{p-2}},	& p\in(2,4),\\
	64 \norm{\frac{d\nu}{d\pi}-1}{2},				& p\in[4,\infty].
\end{cases}
\]
We consider numerical experiments under the following conditions.
Suppose that
\begin{itemize}
	\item\; the computational resources are either $N=10^5$ or $N=10^6$.
	\item\; $\beta=0.9$ or $\beta=0.99$ or $\beta=0.999$.
  	\item\; $C=C(p)=10^{30}$, independent of $p$. 
\end{itemize}

Then the suggestion of the burn-in of Theorem~\ref{main_coro_gen} for $p=2$ and $p\in[4,\infty]$ has the form
\[
n_0^{\set{2}\cup [4,\infty)} = \left \lceil\frac{\log(C)}{\log(\beta^{-1})} \right \rceil,
\]
whereas for $p\in(2,4)$ it still depends on $p$, such that
\[
n_0^{(2,4)} = \left \lceil \frac{p}{2(p-2)} \frac{\log(C)}{\log(\beta^{-1})} \right \rceil.
\]
The error for $\norm{f}{p}\leq1$ where $p\in\set{2}\cup[4,\infty)$ is bounded by
\[
\text{est}_{\set{2}\cup[4,\infty)}(n,n_0)=\sqrt{\frac{2}{n(1-\beta)}+\frac{2C \beta^{n_0}}{n^2(1-\beta)^2}}
\]
whereas for $p\in(2,4)$ we have the upper estimate 
\[
\text{est}_{(2,4)}(n,n_0)=\sqrt{\frac{2}{n(1-\beta)}+\frac{2C \beta^{2n_0\frac{p-2}{p}}}{n^2(1-\beta)^2}}.
\]
With the restriction $N=n+n_0$ one can numerically compute a burn-in, 
which approximates the minimal upper error bound.
This is a $1$-dimensional minimization problem with different parameters.
Let us denote the numerically computed values of the burn-in by
$n^{\set{2}\cup[4,\infty)}_{\text{opt}}$ for $p\in\set{2}\cup[4,\infty)$ 
and $n^{(2,4)}_{\text{opt}}$ for $p\in(2,4)$ respectively.\\

\begin{table}[htb]
\begin{center}
\small
\begin{tabular}{|c|c|c|c|c|c|}

\hline 

& & & & &\\[-2ex]

$N$ & $\beta$	&  $n^{\set{2}\cup [4,\infty)}_{\text{opt}}$	& $n_0^{\set{2}\cup [4,\infty)}=\left\lceil \frac{\log(C)}{\log(\beta^{-1})} \right\rceil$ 
                &  $n^{(2,4)}_{\text{opt}}$ & $n_0^{(2,4)}=\left\lceil \frac{p}{2(p-2)}\frac{\log(C)}{\log(\beta^{-1})}\right\rceil$  \\
		& 	& {\scriptsize (by Maple)} 	& \scriptsize{(suggested above)} & \scriptsize{(by Maple)}	& \scriptsize{(suggested above, $p=2.1$)}    		\\[1ex]  		
\hline & & & & & \\[-1.5ex] 

$10^5$ & $0.9$   	&   $656$					& $656$		 	&	$6655$		&	$6885$	\\
$10^6$ & $0.9$		&	  $656$					&	$656$			&	$6655$		&	$6885$	\\
$10^5$ & $0.99$   &   $6873$				&	$6874$		&	$69642$		&	$72169$	\\
$10^6$ & $0.99$		&	  $6874$				&	$6874$ 		&	$69715$		&	$72169$	\\
$10^5$ & $0.999$  &   $68977$				&	$69043$  	&	$79011$		&	$724952$\\
$10^6$ & $0.999$	&	  $69041$				&	$69043$ 	&	$699520$	&	$724952$\\[1ex]
\hline
\end{tabular} 
\end{center}
\caption{For $C=10^{30}$ and $p=2.1$. 
The numerically computed value $n^{\text{Int}}_{\text{opt}}$ which approximately
	minimizes the mapping $n_0 \mapsto \text{est}_\text{Int}(N-n_0,n_0)$,
	  either $\text{Int}=\set{2}\cup [4,\infty)$ or $\text{Int}=(2,4)$.}
  
\label{tab_burn_in}
\end{table}

Table~\ref{tab_burn_in} gives a collection of  $n^{\set{2}\cup[4,\infty)}_{\text{opt}}$ and
 $n^{(2,4)}_{\text{opt}}$ where $p=2.1$. 
The suggested $n_0$ of Theorem~\ref{main_coro_gen} 
is close to the numerically computed values of the burn-in, which approximately minimize the error bound.  
For $N=10^{5}$ and $\beta=0.999$ the difference
between $n^{(2,4)}_{\text{opt}}$ and $n_0^{(2,4)}$ is large. 
In this situation Theorem~\ref{main_thm} gives for no choice 
of $n$ and $n_0$ with $N=10^5$ an error smaller than $1$. 
The available resources $N=n+n_0$ are too small, 
such that the suggested burn-in cannot be reached.
If the computational resources are large enough, then 
the computed values $n^{\set{2}\cup[4,\infty)}_{\text{opt}}$ 
and $n^{(2,4)}_{\text{opt}}$ are of the same magnitude 
as the suggested $n_0^{\set{2}\cup [4,\infty)}$ and $n_0^{(2,4)}$.\\

If an error of at most $\e\in (0,1)$ is desired, then 
the suggested choice $n_0^{\set{2}\cup [4,\infty)}$ or $n_0^{(2,4)}$,
depending on $p$, of the burn-in is independent
of the precision $\e$. We choose $n_0$ as suggested in Theorem~\ref{main_coro_gen} and
\[
n 
 \geq   \frac{ 1+\sqrt{1+4\e^2}}{(1-\beta)\e^2}
\qquad\text{to achieve} \qquad 
e_\nu(S_{n,n_0},f) \leq \e.
\]  
\\
Let the Markov chain be reversible with respect to $\pi$ and let $\Lambda=\beta$. 
For different fixed values $n_0$ a plot of
\[
\text{est}_{\set{2}\cup[4,\infty)}(N-n_0,n_0)
\quad \text{and} \quad 
\sup_{\norm{f}{2}\leq1} e_\pi(S_N,f)=\sqrt{\frac{1+\Lambda}{N(1-\Lambda)} -\frac{2\Lambda(1-\Lambda^N)}{N^2(1-\Lambda)^2}}
\] 
is presented in Figure~\ref{diff_asymp}.
Roughly spoken one can see that if the burn-in is chosen too small 
a vertical shifting takes place and if the burn-in is chosen 
too large a horizontal shifting takes place. 
\begin{figure}[htb] 
  \begin{center}
    \includegraphics[height=9cm]{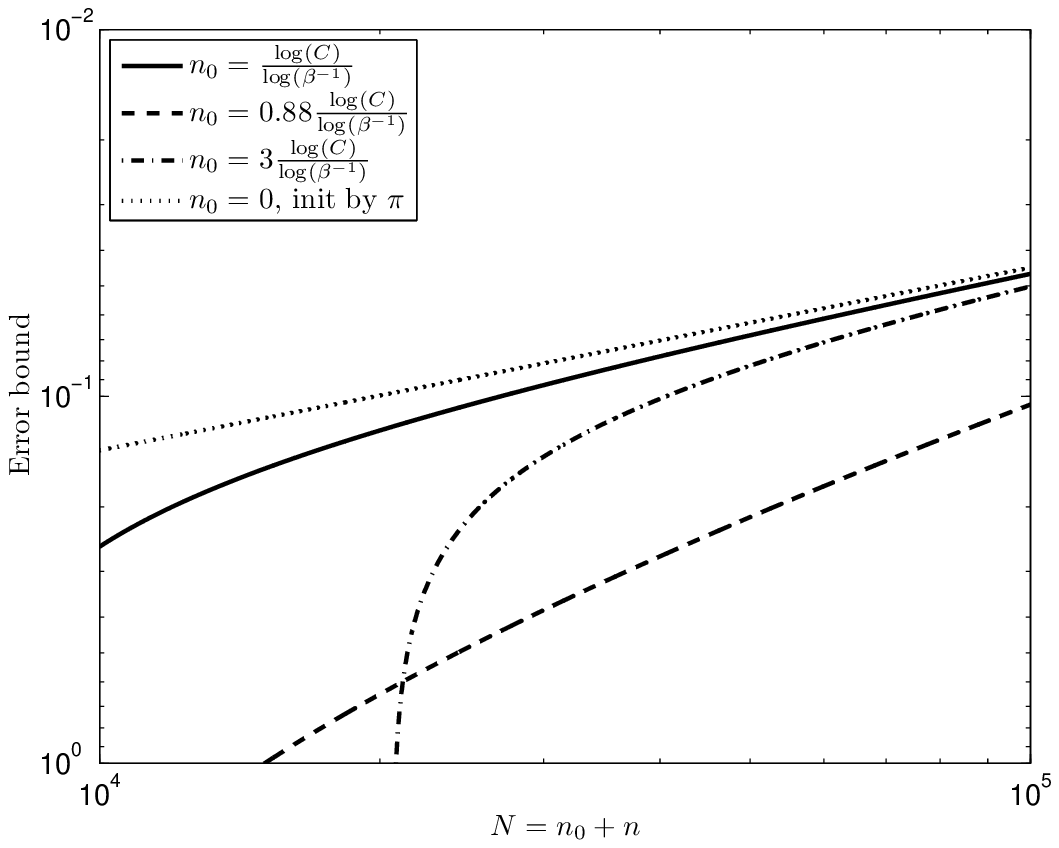}
    \caption{For $\beta=\Lambda=0.99$ and $C=10^{30}$ the mapping $N \mapsto \text{est}_{\set{2}\cup[4,\infty)}(N-n_0,n_0)$ is plotted 
    for different values of $n_0$. The dotted curve is a plot of the mapping $N\mapsto \sup_{\norm{f}{2}\leq1} e_\pi(S_N,f)$.}
    \label{diff_asymp}
    \end{center}  
\end{figure} 
Summarized one can say, if $\beta$, $C$ and $p$ are given, then choose the burn-in as suggested above. 
If there is an estimate of $\log(C)/\log(\beta^{-1}) $, then 
one should ensure that it is not smaller than the real quotient. 
As seen in Figure~\ref{diff_asymp} if it is slightly smaller there is already a strong influence. 
By choosing the burn-in too large the influence is less heavy.\\ 
 
If there is nothing known about $\beta$ or $C$ another
strategy is to choose $n=n_0=N/2$ for even $N$. 
This has the advantage that no information about $\beta$ or $C$ is needed. 
In Figure~\ref{exact} we plotted
\[
\text{est}_{\set{2}\cup[4,\infty)}(N/2,N/2),\;
\text{est}_{\set{2}\cup[4,\infty)}(N-n_0^{\set{2}\cup [4,\infty)},n_0^{\set{2}\cup [4,\infty)})
\; \text{ and} \;
\sup_{\norm{f}{2}\leq1} e_\pi(S_N,f) 
\]
where $N\in[10^4,10^5]$.
\begin{figure}[htb] 
  \begin{center}
    \includegraphics[height=9cm]{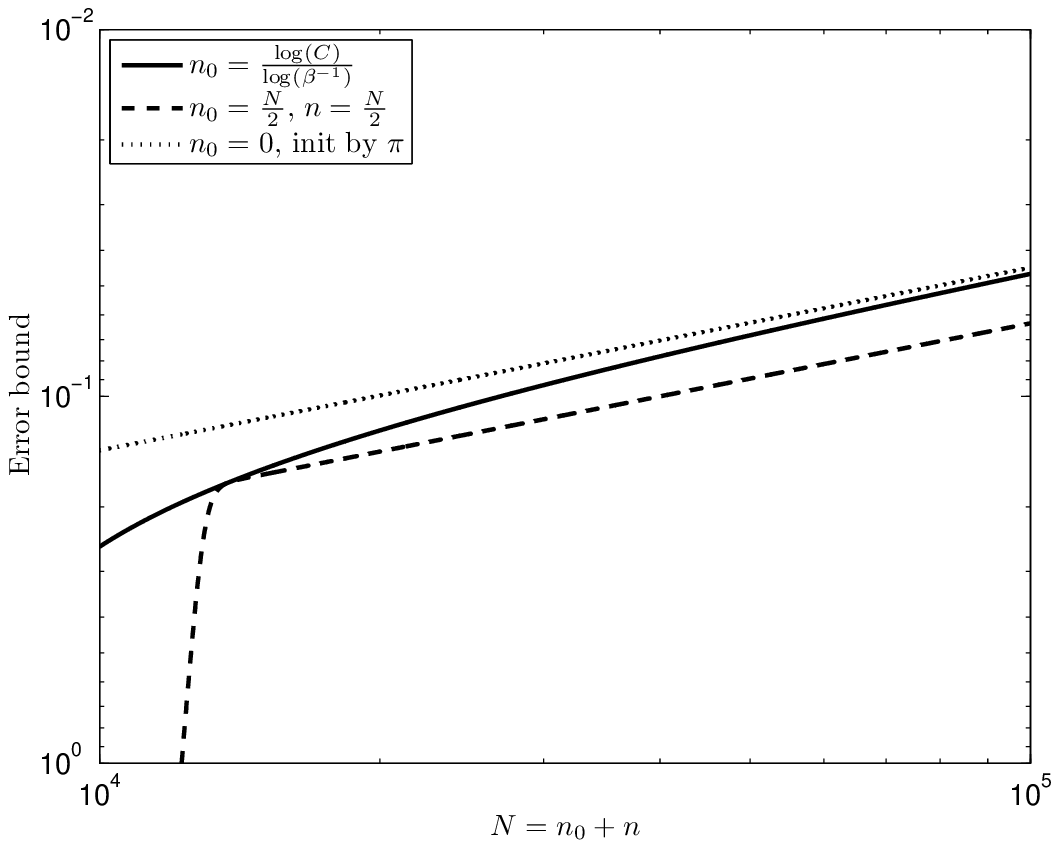}
    \caption{For $\beta=\Lambda=0.99$ and $C=10^{30}$ the mapping $N \mapsto \text{est}_{\set{2}\cup[4,\infty)}(N-n_0,n_0)$ is plotted 
    for different values of $n_0$. The dotted curve is the plot of the mapping $N\mapsto \sup_{\norm{f}{2}\leq1} e_\pi(S_N,f)$.}
    \label{exact}
    \end{center}  
\end{figure} 
Asymptotically
the price of a factor of $\sqrt{2}$ is paid, 
i.e. asymptotically the error is $\sqrt{2}$ times worse 
than $\sup_{\norm{f}{2}\leq1} e_\pi(S_N,f)$, see Figure~\ref{exact}. 
This strategy works well and reaches the same rate of convergence as
 in Theorem~\ref{main_coro_gen}.

\section{Examples}  \label{ex_gen}
For the examples in Section~\ref{toy_fin} 
one can provide all eigenfunctions and eigenvalues.
Usually
it is a challenging task to obtain the necessary information of the spectral structure of the Markov operator,
in particular on general state spaces.
This section contains 
examples to illustrate the error bounds.
The literature provides some tools which can be applied 
to estimate the quantities of interest, e.g. $\Lambda$, $\beta$.
These tools are briefly introduced. For further details we refer to the literature.
Note that the initial distributions of the Markov chains of the following examples 
are chosen to demonstrate the error bounds and not chosen to minimize the burn-in. 
\subsection*{Bounded state spaces}
Suppose that the state space $D$ is a measurable subset of $\R^d$.
The $\sigma$-algebra $\D$ is given by $\Borel(D)$.
We say a transition kernel $K$ has a \emph{transition density} with respect to a positive measure $\mu$
if there is a non-negative function $k\colon D \times D \to [0,\infty]$, 
such that
\[
K(x,A)=\int_A k(x,y)\, \mu(\dint y), \quad  x\in D,\,A\in \Borel(D).
\]
We write $k^n$ for the transition density of $K^n$.\\

Let $D$ be a bounded set and let the function $\rho\colon D \to [0,\infty]$ be integrable with respect to the Lebesgue measure, 
with $\int_D \rho(x)\, \dint x >0$. Then 
\[
\pi_\rho(A)=\frac{\int_A \rho (x)\, \dint x}{\int_D \rho(x) \, \dint x}, \quad A\in\Borel(D),
\]
is a probability measure on $(D,\Borel(D))$.
We say $\rho$ is an \emph{unnormalized density} 
with respect to the Lebesgue measure if $\int_D \rho(x)\, \dint x \neq 1$.
Let $K$ be a transition kernel with transition density $k$ 
with respect to the Lebesgue measure and assume that $\pi_\rho$ is a stationary distribution of $K$.
Furthermore, let $s\in[0,1]$ and let us define
\[
K_s(x,A)= (1-s) K(x,A) + s \mathbf{1}_A(x), \quad  x\in D,\,A\in \Borel(D).
\]
The transition kernel $K_s$ is called the \emph{$s$-modified transition kernel of $K$}.
If $s=\frac{1}{2}$ then the lazy version of $K$ is given and if $s=0$ then one has $K$.
For all $s\in[0,1]$ we have that $\pi_\rho$ is a stationary distribution of $K_s$. The goal is to approximate
\[
S(f)=\int_D f(x)\, \pi_\rho(\dint x).
\]
One obtains for $n\in\N$ that
\begin{equation}  \label{binom_trans}
K_s^n(x,A)= \sum_{i=1}^{n} s^{n-i} (1-s)^i \binom{n}{i} K^{i}(x,A) + s^n \mathbf{1}_A (x), \quad x\in D,\,A\in \Borel(D).
\end{equation}
The case $s=0$ is reasonable if we define $0^0=1$.
The following lemma determines a condition 
which implies $L_1$-exponential convergence of the $s$-modified transition kernel.
For simplicity let us assume that $\int_D \rho(x)\,\dint x =1 $.
\begin{lemma}  \label{s_mod_trans}
	If there exist an $\a \in[0,1)$ and $M<\infty$ such that
\[
2s^n+  		\int_D \esssup_{y\in D} 
		  		\abs{\sum_{i=1}^{n} s^{n-i} (1-s)^{i}\binom{n}{i} 
		  				\frac{k^{i}(x,y)}{\rho(y)}-(1-s^n)} \rho(x)\,\dint x 
		  \leq \a^n M,\quad n\in\N,
\]
then 
the transition kernel $K_s$ is $L_1$-exponentially convergent with $(\a,M)$. 
\end{lemma}
\begin{proof}
The Markov operator of $K_s$ is denoted by
$P_s$. 
Then
\begin{align*}
 & \norm{(P_s^n  -S)f}{1}  
 = \int_D \abs{ \int_D f(y) (\sum_{i=1}^{n} s^{n-i} (1-s)^{i} \binom{n}{i} k^{i}(x,y)\,\dint y)
 				+s^n f(x) - S(f)} \rho(x)\, \dint x \\
 			&\quad\qquad \leq \int_D \int_D \abs{f(y)} 
 				\abs{\sum_{i=1}^{n} s^{n-i} (1-s)^{i}\binom{n}{i} 
 				\frac{k^{i}(x,y)}{\rho(y)}-(1-s^n)} \rho(y)\, \dint y\, \rho(x)\,\dint x\\ 
 			& \quad\qquad\qquad + s^n \int_D \abs{f(x)-S(f)} \rho(x)\, \dint x\\
		  & \quad\qquad\leq \norm{f}{1} \int_D \esssup_{y\in D} 
		  		\abs{\sum_{i=1}^{n} s^{n-i} (1-s)^{i}\binom{n}{i} \frac{k^{i}(x,y)}{\rho(y)}-(1-s^n)} \rho(x)\,\dint x\\
		  & \quad\qquad\qquad + s^n \norm{f-S(f)}{1}\\		
		  & \quad\qquad\leq \norm{f}{1} (2s^n+
		  		\int_D \esssup_{y\in D} 
		  		\abs{\sum_{i=1}^{n} s^{n-i} (1-s)^{i}\binom{n}{i} 
		  				\frac{k^{i}(x,y)}{\rho(y)}-(1-s^n)} \rho(x)\,\dint x )		  
\end{align*}
proves the assertion.
\end{proof}

For $n=1$ and $s=0$ one has a criterion for $L_1$-exponential 
convergence with $(\a,1)$ for the transition kernel $K$. 

%

\begin{coro} \label{L_1_einf}
If there exists an $\a \in[0,1)$ such that
\[
		\int_D \esssup_{y\in D}\abs{\frac{k(x,y)}{\rho(y)}-1} \rho(x)\,\dint x \leq \a,
\]
then the transition kernel $K$ is $L_1$-exponentially convergent with $(\a,1)$.
\end{coro}

\subsubsection*{Example 1}
Let us present an easy example borrowed from \cite[p.~402]{rosenthal_conv}. 
Let $D=[0,1]$ and $\D=\Borel([0,1])$.
The transition kernel is defined by
\[
K(x,A)= \int_A \frac{1+x+y}{x+\frac{3}{2}}\, \dint y, \quad x\in [0,1],\, A\in \Borel([0,1]).
\]
The stationary distribution is given by
\[
\pi(A)=\frac{1}{2}\int_A (x+\frac{3}{2})\, \dint x,\quad A\in \Borel([0,1]).
\]
The transition kernel $K$ is reversible with respect to $\pi$.
These properties can be checked straightforward. 
We have
\begin{align*}
	& \quad	\int_0^1 \esssup_{y\in[0,1]} \abs{\frac{k(x,y)}{\rho(y)}-1}  \, \rho(x)\,\dint x
		= \int_0^1 \esssup_{y\in[0,1]} \frac{\abs{x+y-2xy-\frac{1}{2}}}{2(y+\frac{3}{2})(x+\frac{3}{2})}\,\rho(x)\,\dint x \\
	&	=	\int_0^1 \esssup_{y\in[0,1]} \frac{\abs{x+y-2xy-\frac{1}{2}}}{4(y+\frac{3}{2})}\,\dint x
%
%
%
	 = \frac{1}{6}\int_0^1 \abs{x-\frac{1}{2}}\,\dint x = \frac{1}{24}.
\end{align*}
Hence Corollary~\ref{L_1_einf} gives that the transition kernel is $L_1$-exponentially convergent with $(1/24,1)$.
Because of the reversibility one can apply Proposition~\ref{uni_impl_geo}
and has that the transition kernel is $\pi$-a.e. uniformly ergodic with $(1/24,1/2)$. 
Furthermore there exists an $L_2$-spectral gap, one has $\beta \leq \a= 1/24$.\\
Let $\d\in(0,2/3)$ and let the initial distribution $\nu$ be given by
\[
\nu(A) = \frac{1}{\d} \int_A \mathbf{1}_{[0,\d]}(x)\, \dint x ,\quad A\in \Borel([0,1]).
\]
Hence the initial state is chosen uniformly distributed in $[0,\d]$.
Then
\[
\norm{\frac{d\nu}{d\pi_\rho}-1}{\infty} 
= \esssup_{x\in[0,1]} \abs{ \frac{4\cdot \mathbf{1}_{[0,\d]}(x)}{\d(2x+3)}-1 } = \frac{4}{3\d}-1.
\]
Theorem~\ref{main_coro_gen}\,\eqref{erstens_burn_in} suggests the choice
\[
n_0= \left\lceil \frac{\log( \frac{4}{3\d}-1 )}{\log(24)}\right\rceil
\] 
such that
\[
\sup_{\norm{f}{2}\leq1} e_\nu(S_{n,n_0},f)^2 \leq \frac{48}{23n}+\frac{1152}{529n^2} < \frac{5}{n}.
\]


\subsubsection*{Example 2}
It is taken from \cite[p.~172]{rosenthal_nearly_period}.
Let $D=[-1,1]$ and $\D=\Borel([-1,1])$. The transition kernel is defined by
\[
K(x,A)= \int_A  \mathbf{1}_{[-1,0]}(x) \mathbf{1}_{(0,1]}(y)
	    +  \mathbf{1}_{(0,1]}(x) \mathbf{1}_{[-1,0]}(y)\, \dint y, \quad x\in [-1,1],\,A\in \Borel([-1,1]).
\]
For $x\in[-1,0]$ the next state is uniformly distributed in $(0,1]$
and for $x\in(0,1]$ the next state is uniformly distributed in $[-1,0]$. 
The transition kernel is reversible with respect to the uniform distribution on $D$, thus 
$\pi_\rho$ is given by $\rho(x)=1/2$ for $x\in D$.
For $n\in \N$ we have 
\begin{align*}
 K^n(x,A)=\begin{cases}
 	K(x,A), 	& n\; \text{odd},\\
	K^2(x,A),	& n\; \text{even},
 	\end{cases}
\end{align*}
where
\[
	K^2(x,A)	=	\int_A \mathbf{1}_{[-1,0]}(x) \mathbf{1}_{[-1,0]}(y)
	    +  \mathbf{1}_{(0,1]}(x) \mathbf{1}_{(0,1]}(y)\, \dint y,\quad x\in [-1,1],A\in \Borel([-1,1]).
\]
The spectrum of $P$ is completely known, one has $\spec(P|L_2)=\set{1,0,-1}$ with
\begin{align*}
\mbox{Eig}(P,1) &=\set{f\in L_2 \mid f\equiv c,\,c\in\R}=(L_2^0)^\bot,\\
\mbox{Eig}(P,0)&
		 =\{f\in L_2\mid \int_{-1}^0 f(x)\dint x=\int_0^1 f(x)\dint x=0\},\\
%
\mbox{Eig}(P,-1)&=\{f\in L_2 \mid f(x)=c\,(\mathbf{1}_{[-1,0]}(x)-\mathbf{1}_{(0,1]}(x)),\; c\in \R\},
\end{align*}
where $\mbox{Eig}(P,\lambda)$ denotes the eigenspace of the eigenvalue $\lambda$.
Clearly $\spec(P|L_2^0)=\set{0,-1}$. 
To apply the error bounds one has to pass over to $\widetilde{K}$, the lazy version of $K$.
Let $\widetilde{P}$ be the transition operator which corresponds to $\widetilde{K}$. We denote
$\beta=\beta_{\widetilde{K}}$ and $\Lambda=\Lambda_{\widetilde{K}}$ to indicate the transition kernel $\widetilde{K}$.
We have $\spec(\widetilde{P}|L_2)=\set{1,\frac{1}{2},0}$ 
and $\spec(\widetilde{P}|L_2^0)=\set{\frac{1}{2},0}$. 
The operator $\widetilde{P}$ has an $L_2$-spectral gap, one obtains 
\[
\beta_{\widetilde{K}}=\Lambda_{\widetilde{K}}=\norm{\widetilde{P}}{L_2^0 \to L_2^0} = \frac{1}{2}.
\]
Note that $\widetilde{K}=K_{\frac{1}{2}}$.
By 
the special structure of $K^n$ one obtains 
for $x,y\in D$
that
\begin{align*}
\frac{1}{2^n} \sum_{i=1}^n \binom{n}{i} \frac{k^i(x,y)}{\rho(y)}
	& = \frac{1}{2^{n-1}} 
	  \begin{cases}
		\sum_{i=0}^{\frac{n-1}{2}}	\binom{n}{2i+1} k(x,y) + \sum_{i=1}^{\frac{n-1}{2}} \binom{n}{2i} k^2(x,y), & n\; \text{odd},\\
		\sum_{i=0}^{\frac{n}{2}-1} \binom{n}{2i+1} k(x,y)+\sum_{i=1}^{\frac{n}{2}-1} \binom{n}{2i} k^2(x,y),      & n\; \text{even},
          \end{cases}
	  \\
%
%
	& =	(k(x,y)+k^2(x,y)) -\frac{k^2(x,y)}{2^{n-1}} 
	  = 1  -\frac{k^2(x,y)}{2^{n-1}} .
\end{align*}
It follows that
\begin{align*}
& \quad \int_{-1}^1 \esssup_{y\in[-1,1]} \abs{\frac{1}{2^n} \sum_{i=1}^n \binom{n}{i} \frac{k^i(x,y)}{\rho(y)}-1+\frac{1}{2^n}} \rho(x)\dint x\\
& =	\int_{-1}^1 \esssup_{y\in[-1,1]} \abs{\frac{1}{2^n} -\frac{k^2(x,y)}{2^{n-1}} } \frac{1}{2}\, \dint x
 = \frac{1}{2^n}.
\end{align*}
By Lemma~\ref{s_mod_trans} we get 
with $s=1/2$ that the kernel $\widetilde{K}$ is $L_1$-exponentially convergent with $(1/2,3)$, i.e. 
\[
  \norm{\widetilde{P}^n-S}{L_1\to L_1} \leq \frac{3}{2^n}, \quad n\in\N.
\]
The parameter $\a=1/2$ of the $L_1$-exponential convergence is optimal, since $\beta_{\widetilde{K}}=1/2$ 
and in general for reversible, $L_1$-exponentially convergent transition kernel with $(\a,M)$ one has $\beta\leq \a$.

Let $\d\in(0,1)$. Assume that the initial distribution is given by
\[
\nu(A) = \frac{1}{\d} \int_A \mathbf{1}_{[0,\d]}(x)\, \dint x, \quad A\in \Borel([-1,1]),
\]
i.e. the initial state is chosen with respect to the uniform distribution in $[0,\d]$.
Then
\[
\norm{\frac{d\nu}{d\pi_\rho}-1}{\infty}
= \esssup_{x\in[-1,1]} \abs{\frac{2\cdot\mathbf{1}_{[0,\d]}(x)}{\d}-1}
=	\frac{2}{\d}-1.
\]

Theorem~\ref{main_coro_gen}\,\eqref{erstens_burn_in} suggests the choice
\begin{equation*} \label{burn_in_ex}
n_0= \left \lceil \frac{\log( 3(\frac{2}{\d}-1) )}{\log(2)} \right \rceil
\end{equation*}
such that for $S_{n,n_0}$, which uses a Markov chain with transition kernel $\widetilde{K}$ and initial distribution $\nu$,
one has
\begin{equation}	\label{upp_bound_b_space}
\sup_{\norm{f}{2}\leq1} e_\nu(S_{n,n_0},f) \leq \sqrt{\frac{4}{n}+\frac{8}{n^2}}.
\end{equation}
By Remark~\ref{lower_bound_gen}, by the $L_1$-exponential convergence of $\widetilde{K}$ with $(1/2,3)$
and $\Lambda_{\widetilde{K}}=\beta_{\widetilde{K}}$ one obtains the lower bound 
\begin{equation} \label{low_bound_b_space}
	\sqrt{\frac{3}{n}-\frac{16}{n^2}} \leq \sup_{\norm{f}{2}\leq1} e_\nu(S_{n,n_0},f).
\end{equation}

By 
Corollary~\ref{asymp_err_coro_gen}
one has for all $u \in \mbox{Eig}(\widetilde{P},\frac{1}{2})= \mbox{Eig}(P,0)$ 
with $\norm{u}{2}=1$ that
\[
	e_\pi(S_n,u)^2
	= \sup_{\norm{f}{2}\leq 1} e_{\pi}(S_{n},f)^2
	= \lim_{n_0 \to \infty}	\sup_{\norm{f}{2}\leq 1} e_{\nu}(S_{n,n_0},f)^2.
\]
This motivates the comparism of the lower error bound, the upper error bound and 
the exact error for a specific $u\in\mbox{Eig}(\widetilde{P},\frac{1}{2})$.
Namely, let
\[
	u(x)=	\begin{cases}
		-1, & x\in[-1,-\frac{1}{2}] \cup [0,\frac{1}{2}),\\
		\;\,\,\, 1, & x\in(-\frac{1}{2},0] \cup (\frac{1}{2},1].
		\end{cases}
\]
By $u^2=1$ we get  
\[
L_j(u^2) =0\quad \text{and}\quad L_j(u P^k u) = \frac{1}{2^k} L_j(u^2) = 0, \quad \text{for}\; j,k\in \N.
\] 
Hence by Proposition~\ref{connect_lem} one has
\begin{equation} \label{exact_err_b_space}
	e_{\nu}(S_{n,n_0},u)
	= e_\pi(S_n,u)
	= \sqrt{\frac{3}{n} - \frac{4(1-2^{-n})}{n^2}}.
\end{equation}

\begin{figure}[htb] 
  \begin{center}
    \includegraphics[height=9cm]{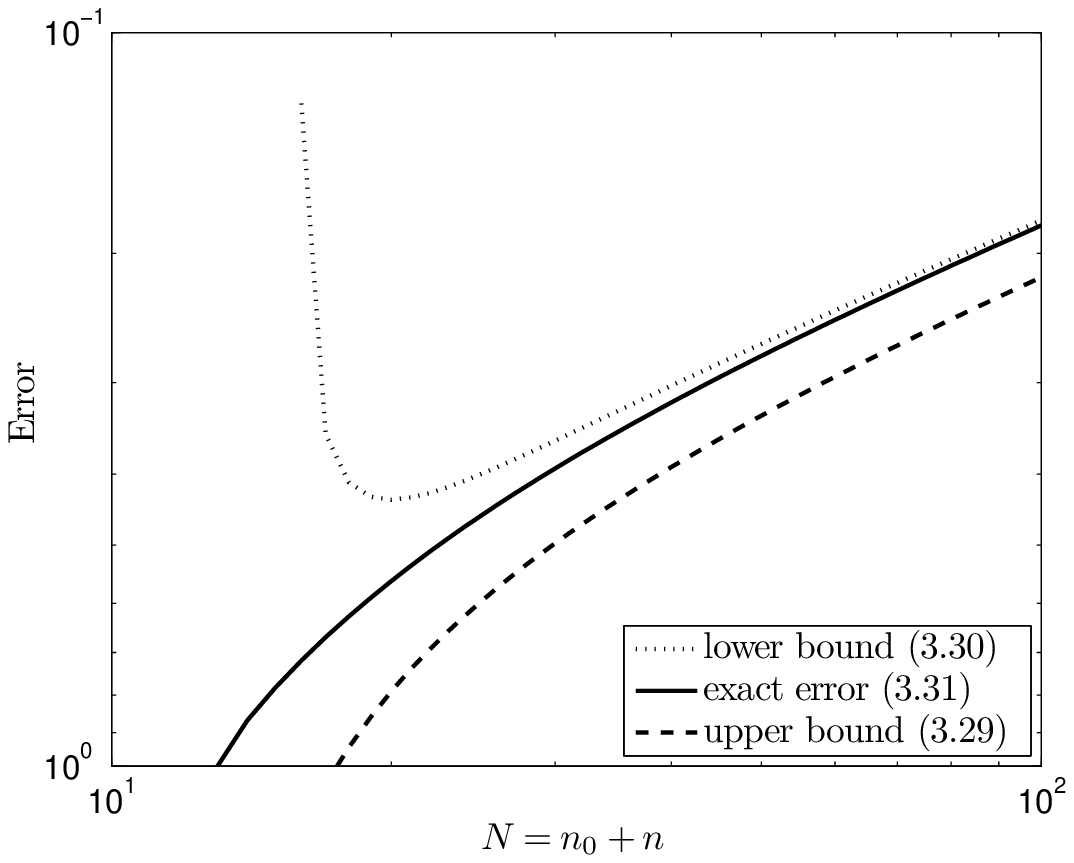}
    \caption{Example 2: Exact error and error bounds, $\d=10^{-3}$ and 
    $n_0= \left \lceil \frac{\log( 3(\frac{2}{\d}-1) )}{\log(2)} \right \rceil $ = 13.}
    \label{n_opt_low_up_d10-3}
    \end{center}  
\end{figure}

In Figure~\ref{n_opt_low_up_d10-3} for $\d=10^{-3}$ the exact error \eqref{exact_err_b_space}, 
the upper error bound	\eqref{upp_bound_b_space} and the lower bound \eqref{low_bound_b_space} are plotted.
The lower bound leads to a non-trivial
estimate if $N\geq n_0+ 6 = 19$. The curve of the upper error estimate is shifted down, 
because the coefficient of the leading term is worse
than the coefficient of the leading term of the exact error $e_\nu(S_{n,n_0},u)$. \\

Lemma~\ref{s_mod_trans} 
provides a tool which can be used to show $L_1$-exponential convergence for several examples.
Unfortunately it is rather difficult to apply for more sophisticated applications.
Next let us present the Metropolis-Hastings algorithm.

\subsection*{Metropolis-Hastings algorithm} \label{Metro_alg}

The Metropolis-Hastings algorithm, suggested in \cite{metro} and extended in \cite{hastings}, is widely used 
in applications. 
The following introduction is based on Mengersen and Tweedie \cite{mengersen}. 
Suppose that the state space $D$ is contained in $\R^d$ and equipped with $\Borel(D)$. 
Let $\pi_\rho$ be a probability measure on $(D,\Borel(D))$ 
given by a possibly unnormalized density $\rho$ with respect to the Lebesgue measure, one has
\[
\pi_\rho(A) = \frac{\int_A \rho(x)\,\dint  x}{\int_D \rho(x) \,\dint x},\quad A\in \Borel(D).
\]  
Let $q\colon D \times D \to [0,\infty]$ be a function which satisfies that 
$q(x,\cdot)$ is integrable with respect to the Lebesgue measure for all $x\in D$ 
and assume that
\[
Q(x,A)=\int_A q(x,y)\, \dint y+\mathbf{1}_A(x)\left( 1-\int_D q(x,y)\, \dint y \right), \quad x\in D,\; A\in \Borel(D),
\] 
is a transition kernel.
It might happen that for some $x\in D$ one has $Q(x,\set{x})>0$. 
If $Q(x,\set{x})=0$ for all $x\in D$ then $q$ is a transition density of $Q$. 
The question is how to modify $Q$ to get a transition kernel with stationary distribution $\pi_\rho$.
For $x,y\in D$ let 
\[
\theta(x,y) = \begin{cases}
		\min\set{\frac{\rho(y)q(y,x)}{\rho(x)q(x,y)} ,1 }, 	& \rho(x)q(x,y)>0,\\
		1,	&	\rho(x)q(x,y)=0,
	      \end{cases}
\]
be the \emph{acceptance probability}.
Then the \emph{Metropolis-Hastings transition kernel $K_\rho$} is defined by
\begin{align*}
K_\rho(x,A) 
& = \int_A \theta(x,y)\, Q(x,\dint y) + \mathbf{1}_A(x)
	\left( \int_D (1-\theta(x,y))\, Q(x,\dint y) \right) \\
& =\int_A \theta(x,y) q(x,y)\, \dint y + \mathbf{1}_A(x)
	\left( \int_D (1-\theta(x,y)) q(x,y)\, \dint y +Q(x,\set{x}) \right),
\end{align*}
where $x\in D$ and $A\in \Borel(D)$. 
In this setting $Q$ is called the \emph{proposal transition kernel} of $K_\rho$.
If $q(x,y)=q(y,x)$ for all $x,y\in D$, then we call $K_\rho$ the \emph{Metropolis transition kernel}.
By the construction one can see that
the transition kernel $K_\rho$ is reversible with respect to $\pi_\rho$, thus one has the desired stationary distribution.

\begin{lemma}
 The Metropolis-Hastings transition kernel $K_\rho$ is reversible with respect to $\pi_\rho$.
\end{lemma}
\begin{proof}
 It is enough to show that
  \[
    \int_A K_\rho(x,B) \,\pi_\rho(\dint x) = \int_B K_\rho(x,A)\, \pi_{\rho}(\dint x)
  \]
 for disjoint $A,B\in\Borel(D)$. Then the assertion follows by the symmetry 
 $\theta(x,y)q(x,y)\rho(x)=\theta(y,x)q(y,x)\rho(y)$ and Fubini's Theorem.
\end{proof}
The Metropolis-Hastings algorithm, which simulates a
transition of the Metropolis-\break Hastings transition kernel, works as follows: 
Let $x\in D$ be the current state.
Choose a proposal state $y$ with respect to $Q(x,\cdot)$. 
Toss a coin, whose probability that ``head'' occurs is $\theta(x,y)$.  
If it is ``head'' then accept the proposal state, i.e. return $y$.
Otherwise reject the proposal, i.e. return $x$. 
Schematically, a single step of the Metropolis-Hastings algorithm is presented in the Procedure 
\ref{Metro_step}$(x,Q,\rho)$.\\ 
\IncMargin{1em}
\begin{procedure}[htb]

\SetKwFunction{rand}{rand}
\SetKwInOut{Input}{input}
\SetKwInOut{Output}{output}
\BlankLine
\Input{ 
	current state $x$, proposal kernel $Q$, unnormalized density $\rho$.
      }
\Output{ next state $y$. }
\BlankLine
\BlankLine
Choose $y$ with respect to $Q(x,\cdot)$\;
\BlankLine
Compute 
\[
\theta(x,y)= 	\begin{cases}
	\min\set{\frac{\rho(y)q(y,x)}{\rho(x)q(x,y)} ,1 }, & \rho(x)q(x,y)>0,\\
		1,	&	\rho(x)q(x,y)=0;
					\end{cases}
\]
\If{$\rand() \geq \theta(x,y)$}
{
$y:=x$\;
}
Return $y$.
\caption{Metropolis-Step($x$,$Q$,$\rho$)}
\label{Metro_step}
\end{procedure}
\DecMargin{1em}

If $\tilde{q}(y)=q(x,y)$ for all $x,y \in D$ then the proposal transition kernel samples independently of $x$.
In this situation one can apply the following result. 
\begin{theorem} \label{met_uni}
  Let  $\tilde{q}\colon D \to [0,\infty]$ be a function with $\int_D \tilde{q}(x)\, \dint x = 1$.
  Let the proposal transition kernel 
  of the Metropolis-Hastings transition kernel $K_\rho$ be 
    $
      Q(x,A)=\int_A \tilde{q}(y) \dint y
    $
  for $x\in D$ and $A \in \Borel(D)$.
  If there exists a $\gamma>0$ such that   
  \[
      \frac{\tilde{q}(y) }{\rho(y)}\geq \gamma,\quad y\in D,
  \]
  then $K_\rho$ is uniformly ergodic.
  We obtain
  \[
      \norm{K_\rho^n(x,\cdot)-\pi}{\text{\rm tv}} \leq (1-\gamma)^n, \qquad x\in D,\;n\in\N.
  \]
\end{theorem}
\begin{proof}
 See \cite[Theorem~2.1, p.~105]{mengersen}.
\end{proof}
\begin{remark}
 	The proof is based on the well known equivalence that
 	a transition kernel $K$ is uniformly ergodic iff 
	the whole state space $D$ is a small set.
 	A set $R\in \Borel(D)$ is called small if there exists a $\gamma>0$, an $m\in\N$
	and a probability measure $\psi$ such that 
	\begin{equation*}
	K^m(x,A) \geq \gamma \psi(A), \quad x\in R,\;A\in \Borel(D).
\end{equation*}
\end{remark} 
  The result of Theorem~\ref{met_uni} 
  will be demonstrated for a toy example, stated in \cite[p.~107]{mengersen}.\\

  Let $D=\R$ and $\D=\Borel(\R)$. 
  Note that the state space is unbounded. 
  The desired distribution is given by
  the density 
  \[
      \rho(y)=\frac{1}{\sqrt{2\pi}}\exp(-\frac{y^2}{2}),\quad y\in \R,
  \]
  i.e. $\pi_\rho$ is an $N(0,1)$ distribution.
  By $N(\mu,\xi^2)$ we denote the normal distribution 
  with mean $\mu$ and variance $\xi^2$.
  Furthermore, assume that the proposal transition kernel
  samples independently from $N(0,\xi^2)$ so that
  \[
    \tilde{q}(y)=\frac{1}{\sqrt{2\pi}\xi} \exp(-\frac{y^2}{2\xi^2}),\quad y\in \R.
  \]
  Let $\xi^2>1$. Then
  \[
     \frac{\tilde{q}(x)}{\rho(x)}\geq \xi^{-1},
  \]
  which implies that
  \[
      \norm{K_\rho^n(x,\cdot)-\pi_\rho}{\text{\rm tv}} \leq (1-\xi^{-1})^n, \qquad x\in D,\;n\in\N.
  \]
  By the reversibility with respect to $\pi_\rho$ 
  of the Metropolis-Hastings transition kernel an immediate consequence is that 
  uniform ergodicity implies $L_1$-exponential convergence,
  since $\pi$-a.e. uniform ergodicity is equivalent to  $L_1$-exponential convergence.
  Hence we have a transition kernel which is $L_1$-exponentially convergent with $(1-\xi^{-1},1)$.
  This implies that the Markov operator $P$ which corresponds
  to the transition kernel $K_\rho$ has an $L_2$-spectral gap, we have
  $1-\beta \geq \xi^{-1}$.
  
  Let $\d\in(0,1)$ and $x_0\in[0,\infty)$.
  The initial state is chosen uniformly distributed in $[x_0-\d,x_0+\d]$.
  Then 
  \[
	\frac{d\nu}{d\pi}(x)=\sqrt{\frac{\pi}{2}} \cdot	\frac{\mathbf{1}_{[x_0-\d,x_0+\d]}(x)}{\d} \,\exp\left(\frac{x^2}{2}\right),\quad x\in D.
  \]
  We obtain
	\[
		\norm{\frac{d\nu}{d\pi}-1}{\infty} = \sqrt{\frac{\pi}{2}} \cdot 	
		\frac{\exp\left(\frac{(x_0+\d)^2}{2}\right)}{\d}-1 \leq
		\sqrt{\frac{\pi}{2}} \cdot 	
		\frac{\exp\left(\frac{(x_0+\d)^2}{2}\right)}{\d}.
	\]
    The method $S_{n,n_0}$ uses a Markov chain with transition kernel $K_\rho$ and initial distribution $\nu$.
	The burn-in is almost chosen as suggested in Theorem~\ref{main_coro_gen}\,\eqref{erstens_burn_in}.
	We use $\log(1-\xi^{-1}) \geq \xi^{-1}$ to estimate the burn-in, such that we set
        \[
	 n_0=\left \lceil \xi \left(\log(\d^{-1})+\frac{(x_0+\d)^2}{2} + 0.23\right )\right \rceil.
	\]
	Then
	\[
		\sup_{\norm{f}{2} \leq 1} e_\nu (S_{n,n_0},f)^2 \leq \frac{2\xi}{n}+\frac{2\xi^2}{n^2}.
	\]	
\subsection*{Contracting Normals}
The next example is described in \cite{baxendale}, 
see also \cite{ro_ro_normals,ro_tw_normals}. Let $D=\R$, $\D=\Borel(\R)$ and $\theta \in (-1,1)$. 
Note that the state space is unbounded. 
The transition kernel is given by
\[
K(x,A)= \frac{1}{\sqrt{2\pi(1-\theta^2)}}\int_A \exp\left(-\frac{(\theta x-y)^2}{2(1-\theta^2)}\right) \dint y,
\quad x\in \R,\,A\in \Borel(\R),
\]
so that $K(x,\cdot)$ is an $N(\theta x,1-\theta^2)$ distribution.
%
By some elementary calculation one can see that a stationary distribution is
\[
\pi(A) =\frac{1}{\sqrt{2\pi}} \int_A \exp\left( -\frac{y^2}{2} \right) \dint y,\quad A\in \Borel(\R),
\]
i.e. $\pi$ is an $N(0,1)$ distribution. The transition kernel $K$ is reversible with respect to $\pi$. 
Suppose that $\theta \in (0,1)$. Then the Markov operator is \emph{positive semi-definite}, i.e.
$\scalar{Pf}{f} \geq 0$, for all $f\in L_2$.
The next result is an application of \cite[Theorem~1.3, p.~702]{baxendale} 
where the Markov operator is self-adjoint and positive semi-definite. 
The same example is considered in \cite[p.~728]{baxendale} and \cite[p.~33]{niemiro_lat}.
\begin{lemma}  \label{tilde}
Let $\theta \in (0,1)$, $c\in(1,\infty)$ and set
\begin{align*}
\lambda & = \theta^2 + \frac{2(1-\theta^2)}{1+c^2},\\
K & =  2+\theta^2(c^2-1),\\
B & = 
	2 \left[\Phi\left(\frac{(1+\theta)c}{\sqrt{1-\theta^2}}\right) 
				- \Phi\left(\frac{\theta c}{\sqrt{1-\theta^2}}\right) \right]
	, \quad \text{where} \quad 
	\Phi(z)=\frac{1}{\sqrt{2\pi}}\int_{-\infty}^z \exp(-\frac{y^2}{2}) \dint y,\\
\a & = 1+ \frac{\log\left( \frac{K-B}{1-B}\right)}{\log(\lambda^{-1})},\\
\hat{\beta} & = \max\set{ \lambda, (1-B)^{1/\a} }<1.
\end{align*}
Then
\[
\beta = \norm{P}{L_2^0 \to L_2^0} \leq \hat{\beta}.
\]
\end{lemma}
\begin{proof}
See \cite[Theorem~1.3, p.~702 and p.~728]{baxendale}.
\end{proof}
Let us illustrate the last lemma.
For any fixed $\theta$ one can numerically minimize the upper estimate $\hat{\beta}$ of $\beta$, depending on $c$.
For example let $\theta = 0.5$. Then, one gets $\hat{\beta} = 0.8946$ for $c = 1.6041$.\\

It exists an $L_2$-spectral gap, thus we can apply Theorem~\ref{main_coro_gen} for $p\in (2,\infty]$.
Let $\d\in(0,1)$ and $x_0\in[0,\infty)$.
The initial state is chosen uniformly distributed on $[x_0-\d,x_0+\d]$.
The density of the initial distribution with respect to $\pi$ is given by
\[
\frac{d\nu}{d\pi}(x)=\sqrt{\frac{\pi}{2}} \cdot	\frac{\mathbf{1}_{[x_0-\d,x_0+\d]}(x)}{\d} \,\exp\left(\frac{x^2}{2}\right).
\]
Then for all $q\in[1,\infty]$ it follows that
\[
\norm{\frac{d\nu}{d\pi}-1}{q} \leq 		
\norm{\frac{d\nu}{d\pi}-1}{\infty} = \sqrt{\frac{\pi}{2}} \cdot 	
		\frac{\exp\left(\frac{(x_0+\d)^2}{2}\right)}{\d}-1 \leq
		\sqrt{\frac{\pi}{2}} \cdot 	
		\frac{\exp\left(\frac{(x_0+\d)^2}{2}\right)}{\d}.
\]
The burn-in is chosen as suggested in Theorem~\ref{main_coro_gen}, 
where we use the previously stated estimate of $\norm{\frac{d\nu}{d\pi}-1}{q}$.
Suppose that the burn-in $n_0(p)$ is the smallest natural number (including zero)
which is greater than or equal to
\[
    \frac{1}{\log(\hat{\beta}^{-1})} 
    \begin{cases}
	\frac{p}{2(p-2)} \left[\log(\frac{16p}{p-2})+\log(\sqrt{2\pi}\,\d^{-1})+\frac{(x_0+\d)^2}{2}\right], & p\in(2,4),\\
	\log(\d^{-1})+\frac{(x_0+\d)^2}{2}+4.39, & p\in[4,\infty]. 
    \end{cases}
  \]
Then
\[
\sup_{\norm{f}{p}\leq1}e_\nu(S_{n,n_0},f)^2
		\leq \frac{2}{n(1-\hat{\beta})} 
				+				 
				 \frac{2}{n^2(1-\hat{\beta})^2}.
\]
In Table~\ref{tab_contracing_normals} one can see how much resources $N$ 
are sufficient to obtain an error less than $\e=0.01$.

\begin{table}[htb]
\begin{center}
\begin{tabular}{|c|c|c|c|c|c|}

\hline 

& & & & &\\[-2ex]

$\theta$ & $c$	&	$\hat{\beta}$	&  $n_0$ & $n$	&  $N$ \\
	 &      & 	& \scriptsize{(for $p=2.1$)} & \scriptsize{(for precision $\e=0.01$)}& \\[1ex]  		
\hline & & & & & \\[-1.5ex] 
  
$0.91$ & $1.12845$  &   $0.999664$	& $2.82241\cdot10^5$  & $5.94614\cdot10^7$  &	$5.97437\cdot10^7$	\\
$0.92$ & $1.11691$  &   $0.999816$	& $5.16275\cdot10^5$  &	$1.08759\cdot10^8$  &	$1.09275\cdot10^8$	\\
$0.93$ & $1.10499$  &   $0.999912$	& $1.08257\cdot10^6$  &	$2.28043\cdot10^8$  &	$2.29126\cdot10^8$	\\
$0.94$ & $1.09260$  &   $0.999966$	& $2.76738\cdot10^6$  &	$5.82923\cdot10^8$  &	$5.85690\cdot10^8$	\\
$0.95$ & $1.07964$  &   $0.999990$	& $9.60536\cdot10^6$  &	$2.02337\cdot10^9$  &	$2.03297\cdot10^9$\\
$0.96$ & $1.06599$  &   $0.999998$	& $5.58578\cdot10^7$  &	$1.17624\cdot10^{10}$  &$1.18183\cdot10^{10}$
\\[1ex]
\hline
\end{tabular} 
\end{center}
\caption{Contracting Normals: The initial distribution $\nu$ is chosen with $x_0=0$ and $\d=0.1$.
	  The burn-in of Theorem~\ref{main_coro_gen} is computed for $p=2.1$ 
	  and $n$ is computed such that one obtains an error less than $\e=0.01$. 
	  The estimate $\hat{\beta}$ of $\beta$ is computed by a 
	  minimizing procedure of Maple for $c\geq1.01$.
	}
\label{tab_contracing_normals}
\end{table}
%
\section{Notes and remarks}
In the last decades explicit error bounds and confidence estimates of Markov chain Monte Carlo
methods on general state spaces attracted more and more attention. 
In the following let us present how the results 
fit into the published literature.\\[-1ex]


In the seminal work of Lov{\'a}sz and Simonovits \cite{lova_simo1} an estimate of $e_\pi(S_n,f)^2$ is shown. 
The paper deals with the computation of the volume of a convex body by a randomized algorithm based on Markov chains.
Let us explain the result of \cite[Theorem~1.9, p.~375]{lova_simo1} in detail.
Let $(X_n)_{n\in\N}$ be a Markov chain with transition kernel $K$ and initial distribution $\nu$ and
let $K$ be reversible with respect to a probability measure $\pi$. 
Then let us define the \emph{conductance} as
\[
\phi(K,\pi)=\inf_{0<\pi(A)\leq \frac{1}{2}}  \frac{\int_A K(x,A^c) \,\pi(\dint x)}{\pi(A)}.
\]
Assume that the Markov operator is positive semi-definite, i.e.
$\scalar{Pf}{f}\geq 0$ for all $f\in L_2$. 
Then 
\begin{equation}  \label{lovasz_err}
e_\pi(S_n,f)^2 \leq \frac{4}{\phi(K,\pi)^2 \cdot n } \norm{f}{2}^2.
\end{equation}
The result 
is slightly worse than the result of Proposition~\ref{expl_stat_c}. 
In Proposition~\ref{expl_stat_c} one has an exact error formula for $e_\pi(S_n,f)^2$. 
Mainly the spectral structure of the Markov operator is used.
In Corollary~\ref{worst_stat_gen} this exact error formula is further estimated and one obtains
\begin{equation} \label{coro_oben}
e_\pi(S_n,f)^2 \leq \frac{2}{(1-\Lambda)n} \norm{f}{2}^2 ,
\quad \text{where} \quad 
\Lambda = \sup \set{\a\mid \a \in \spec(P|L_2^0)}.
\end{equation}
The \emph{Cheeger inequality}\footnote{
The Cheeger inequality is stated in Section~\ref{cond_concept}.
}, given by 
$
1-\Lambda \geq \frac{\phi(K,\pi)^2}{2}, 
$
provides a relation between $\Lambda$ and $\phi(K,\pi)$,
so that \eqref{coro_oben} implies \eqref{lovasz_err}.
Note that in Proposition~\ref{expl_stat_c} and Corollary~\ref{worst_stat_gen} it is not assumed that 
the Markov operator is positive semi-definite, such that the assumptions are slightly less restrictive.
But if one has a transition kernel $K$ 
which determines a not necessarily positive semi-definite transition operator, 
then one can pass over to the lazy version of $K$ and obtains positive semi-definiteness.
However, the estimate of \eqref{lovasz_err} covers the important facts and it seems 
that the refinement of Proposition~\ref{expl_stat_c} is well known.\\


The paper of Math{\'e} \cite{mathe1} contains results concerning the asymptotic integration error 
for uniformly ergodic Markov chains which are 
reversible with respect to $\pi$. 
For example it is shown that for any initial distribution $\nu \in \mathcal{M}_\infty$ one has   
\[
	 \lim_{n\to\infty} n \cdot  \sup_{\norm{f}{2}\leq 1} e_{\nu}(S_{n,n_0},f)^2  
	 = \frac{1+\Lambda}{1-\Lambda}
\]
and for $f\in L_2$ it is proven that
\[
	\lim_{n\to\infty} n\cdot e_{\nu}(S_{n,n_0},f)^2 = \scalar{(I-P)^{-1}(I+P)g}{g},\;\,\text{where}\;g=f-S(f).
\]
The same result is part of Corollary~\ref{asymp_err_coro_gen} and for individual $f$ part of Theorem~\ref{main_thm_unif}. 
In \cite{mathe2} the asymptotic integration error is 
studied for not necessarily reversible and not necessarily uniformly ergodic Markov chains.
It is assumed that the transition kernel is $V$-uniformly ergodic, see \eqref{V_uniform}. 
For further details let us refer to \cite{mathe2}.\\


In \cite[Theorem~8, p.~19]{expl_error} an explicit upper error bound of $e_\nu(S_{n,n_0},f)^2$ 
for general state spaces is provided. 
The result is based on \cite[Theorem~1.9, p.~375]{lova_simo1} and the assumptions are the same.
Namely, the transition kernel $K$ is reversible with respect to $\pi$ 
and the transition operator $P$ is positive semi-definite.
After a burn-in
\begin{equation} \label{cond_dan}
 n_0 \geq \frac{\log(\norm{\frac{d\nu}{d\pi}}{\infty})}{\phi(K,\pi)^2} 
 \quad\text{the error obeys}\quad
 e_\nu(S_{n,n_0},f)^2 \leq \frac{100}{\phi(K,\pi)^2\cdot n} \norm{f}{\infty}^2. 
\end{equation}
The proof of the result is based on Proposition~\ref{connect_lem} which provides the crucial relation between $e_\nu(S_{n,n_0},f)^2$ and
$e_\pi(S_{n,n_0},f)^2$.
By Theorem~\ref{main_thm} and Theorem~\ref{main_coro_gen} one obtains a refined error estimate 
and a refined recipe for the choice of $n_0$.
Note that positive semi-definiteness and reversibility is not needed in Theorem~\ref{main_thm}. 
It is enough that there exists an $L_2$-spectral gap, i.e. $1-\beta>0$.\\


Independently of \cite[Theorem~8, p.~19]{expl_error} in the work of 
Belloni and Chernozhukov
 \cite[Theorem~3, p.~2031]{complex_mcmc_based} a similar 
error bound for $S_{n,n_0}$ is proven.
It is also based on \cite[Theorem~1.9, p.~375]{lova_simo1} such that again the transition kernel 
is assumed to be reversible with respect to $\pi$ and the Markov operator must be positive semi-definite. 
Then it is shown that
\[
e_\nu(S_{n,n_0},f)^2 \leq e_\pi(S_n,f)^2 + 8 \norm{f}{\infty}^2 \norm{\nu P^{n_0}-\pi}{\text{\rm tv}}.
\]
Let the initial distribution $\nu$ be $R$-warm, i.e.
$
\sup_{A\in \D,\, \pi(A)>0} \frac{\nu(A)}{\pi(A)} \leq R. 
$
Then one obtains by \cite[Corollary~1.5, p.~372]{lova_simo1} that
\[
 \norm{\nu P^{n_0}-\pi}{\text{\rm tv}} \leq \sqrt{R} \left(1-\frac{\phi(K,\pi)^2}{2}\right)^{n_0}.
\]
Hence by \cite[Theorem~1.9, p.~375]{lova_simo1} one has
\begin{equation}	\label{belloni_cond}
e_\nu(S_{n,n_0},f)^2 \leq  \frac{4}{\phi(K,\pi)^2 \cdot n } \norm{f}{2}^2 
	+  8  \sqrt{R} \left(1-\frac{\phi(K,\pi)^2}{2}\right)^{n_0}\norm{f}{\infty}^2.
\end{equation}
The result of an explicit error bound for $S_{n,n_0}$, 
when the initial distribution is not the stationary one, is the same as in \cite[Theorem~8, p.~19]{expl_error}.
Note that the burn-in depends on the desired precision. 
We can choose $R = \norm{\frac{d\nu}{d\pi}}{\infty}$ and if one uses $\norm{f}{2}\leq \norm{f}{\infty}$, then 
the upper bound of \eqref{belloni_cond} can be further estimated 
and one obtains an estimate with respect to $\norm{\cdot}{\infty}$.\\


Another result due to {\L}atuszy{\'n}ski and Niemiro is presented in \cite{niemiro_lat}. 
The integration error for $V$-uniformly ergodic Markov chains is estimated, 
where $V\colon D\to[1,\infty)$ is a drift function.
The weighted class of functions
\[
L_V=L_V(D)=\set{ f\colon D\to \R \mid \abs{f}_V= \sup_{x\in D} \frac{\abs{f(x)}}{V(x)} < \infty }
\]
is studied. Let $\a\in [0,1)$ and $M<\infty$. 
A transition kernel $K$ is called \emph{$V$-uniformly ergodic with $(\a,M)$} if 
\begin{equation}		\label{V_uniform}
	\norm{P^n-S}{L_V \to L_V} \leq M \a^n,\quad n\in \N.
\end{equation}
One can substitute the drift function $V$ by $V^{1/r}$ for all $r\geq 1$.
Then there exist an $\a(r) \in [0,1)$ and an $M(r)<\infty$ such that
\begin{equation*}		\label{V_uniform_r}
		\norm{P^n-S}{L_{V^{1/r}} \to L_{V^{1/r}}} \leq M(r) \a(r)^n,\quad n\in\N.
\end{equation*}
This is justified by an interpolation argument in \cite{mathe2} and by different assumptions stated in
\cite
{baxendale}.
Now let us state 
a less general version 
of the main result of \cite[Theorem~3.1, p.~28]{niemiro_lat}.
For $r=2$ and $g=f-S(f)$ one has
\begin{equation}  \label{latus_bound}
		e_{\nu}(S_{n,n_0},f)^2 \leq \frac{\abs{g^2}_V}{n}\left( 1+ \frac{2M(2)\a(2)}{1-\a(2)} \right)
				\left( \norm{V}{1}  + \frac{M^2 \a^{n_0} \norm{\nu-\pi}{V} }{n(1-\a)}\right),
\end{equation}
where $\norm{\nu-\pi}{V}=\sup_{\abs{g}_V\leq1} \abs{\int_D g(x) (\nu(\dint x)-\pi(\dint x))}$.
This seems to be the first explicit error bound of $S_{n,n_0}$ for integrands $f$ which belong to $L_V$.
If the transition kernel is reversible, then $V$-uniform ergodicity with $(\a,M)$ 
is equivalent to the existence of an $L_2$-spectral gap, see \cite{hybrid, roberts_tweed}. 
Furthermore if $V\in L_p$ for some $p>2$ then $L_V\subset L_p$ and the error bound of 
Theorem~\ref{main_thm} can also be applied.
However, in general Theorem~\ref{main_thm} cannot be used in this setting.\\

The paper of Joulin and Ollivier \cite{joulin_ollivier} based on \cite{ollivier} follows a new idea. 
Let $(D,\dist)$ be a metric, complete, separable state space, with metric $\dist$, and
let $K$ be a transition kernel with stationary distribution $\pi$ on $(D,\Borel(D))$. 
Let $\mathcal{P}_{\dist}(D)$ be the set of probability measures $\mu$ on $(D,\Borel(D))$ 
for which there exists an $x_0\in D$ such that
$\int_D \dist(x_0,y)\, \mu(\dint y)<\infty$.
Then let us define the Wasserstein distance between $\mu_1,\mu_2 \in \mathcal{P}_{\dist}(D)$ by
\[
	W_1(\mu_1,\mu_2)= \inf_{\xi \in \Pi(\mu_1,\mu_2)} \int_D\int_D \dist(x,y) \,\xi(\dint x,\dint y),
\]
where $\Pi(\mu_1,\mu_2)$ is the set of probability measures $\xi$ on $(D^2,\Borel(D^2))$
with marginals
$\mu_1$ and $\mu_2$. If there exists a $\kappa>0$ such that
\begin{equation} 
\label{ricci}
		W_1(K(x,\cdot),K(y,\cdot)) \leq (1-\kappa) \dist(x,y),\quad x,y\in D,
\end{equation}
then we say that the transition kernel $K$ has
\emph{positive Ricci curvature $\kappa$}. 
Let the function $f\colon D \to \R$ be integrable with respect to $\pi$ and let 
\[
\norm{f}{\Lip}=\sup_{x,y\in D,\, x\neq y} \frac{\abs{f(x)-f(y)}}{\dist(x,y)}.
\]
The \emph{coarse diffusion constant} $\sigma(x)$ for $x\in D$ of the transition kernel is defined by
\[
  \sigma(x)^2=\frac{1}{2} \int_D\int_D \dist(y,z)^2 \,K(x,\dint y)\, K(x,\dint z),
\]
and the \emph{local dimension} $n_x$ for $x\in D$ is defined by
\[
n_x = \inf_{\norm{f}{\Lip}=1} \frac{2\sigma(x)^2}{\int_D \int_D \abs{f(y)-f(z)}^2 K(x,\dint z)K(x,\dint y)}. 
\]
If the transition kernel has positive Ricci curvature, then
by \cite[Proposition~1, p.~2423, and Theorem~2, p.~2424]{joulin_ollivier} 
one obtains that
\begin{align*}
	e_{\d_x}(S_{n,n_0},f)^2 \leq & 
	\left( \frac{1}{\kappa^2 n}+\frac{1}{\kappa^3 n^2} \right)\norm{f}{\Lip}^2 \sup_{x\in D} \frac{\sigma(x)^2}{n_x} \\
	& \quad+	\frac{(1-\kappa)^{2(n_0+1)}}{\kappa^4 n^2}\norm{f}{\Lip}^2
		 \left( \int_D \dist(x,y)\, K(x,\dint y) \right)^2.
\end{align*}
The estimate is reasonable for any deterministic initial state $x\in D$, the initial distribution is $\d_x$. 
For further estimates and details let us refer to \cite{joulin_ollivier}. 
Let $p\in(2,\infty]$, let $\norm{f}{\Lip}<\infty$ and 
assume that there exists an $x_0\in D$ such that $\norm{\dist(\cdot,x_0)}{p}<\infty$ 
then one obtains $f\in L_p$, in particular
\[
    \norm{f}{p} \leq 2^{\frac{p-1}{p}}( \norm{f}{\Lip}
    \norm{\dist(\cdot,x_0)}{p}
    +\abs{f(x_0)}).
\]
If the transition kernel is reversible with respect to $\pi$ 
and $\norm{\sigma}{2}<\infty$,  
then one can show that 
a positive Ricci curvature $\kappa>0$ of $K$ 
implies an $L_2$-spectral gap of $P$, 
it follows that $1-\beta \geq \kappa$, see \cite[Proposition~30, p.~831]{ollivier}. 
In this setting 
Theorem~\ref{main_thm} can be applied 
when the initial distribution $\nu$ belongs to $\mathcal{M}_{\max\set{2,\frac{p}{p-2}}}$. \\


A regenerative Markov chain Monte Carlo 
algorithm for the approximation of $S(f)$ is studied in \cite{nie_lat_reg}.
Roughly spoken, if one has certain information of a small set, then one can explicitly estimate the mean square error
of this regenerative estimator
for uniformly and $V$-uniformly ergodic Markov chains, see \cite{nie_lat_reg} for details.\\


The literature provides also confidence estimates for $S_{n,n_0}$.
One can apply Lemma~\ref{lemm_markov} if an upper bound of $e_\nu(S_{n,n_0},f)^2$ is available.
These estimates can be boosted by a median trick 
explained in \cite{niemiro_poka} and applied in \cite{niemiro_lat}. 
However, exponential inequalities such as Hoeffding or Chernoff bounds 
for Markov chain Monte Carlo are better, 
see \cite{Krue,lezaud_fin, hoeff_unif,joulin_ollivier,Mia}.
Asymptotic confidence estimates are discussed in \cite{FleJon}.
\\

%

 
Let us provide a conclusion.
There are different explicit error bounds of the mean square error 
for $S_{n,n_0}$ on general state spaces.
In some situations these estimates could be improved.
It seems that the error bound with respect to $\norm{\cdot}{2}$ is not known so far.
Let us recall that we assumed that the used Markov chain is $L_1$-exponentially convergent 
and reversible with respect to $\pi$. 
If we only assume that the Markov chain has an $L_2$-spectral gap, 
then we showed an estimate of the error 
uniformly with respect to $\norm{\cdot}{p}$ for $p\in(2,\infty]$. 
Upper error bounds with respect to $\norm{\cdot}{\infty}$ 
are known but with respect to $\norm{\cdot}{p}$ seem to be new.
In this setting it
is not assumed that the Markov chain is reversible with respect to $\pi$, we require hat 
$\pi$ is the stationary distribution.
The suggestion of the burn-in $n_0$ of Theorem~\ref{main_coro_gen} performs well and also 
appears to be new.
All error bounds hold for bounded and unbounded state spaces 
whenever estimates of the crucial parameters, for example $\Lambda$,  $\beta$ or $(\a,M)$, are available.

\chapter{Applications} \label{appl}
In numerous applications one wants to compute for $D\subset \R^d$ an integral of the form
\begin{equation}  \label{int_f_rho}
  \int_D f(x) \cdot c \rho(x)\, \dint x, 
\end{equation}
with density $c \rho$, where the number $c$ is unknown. 
Of course $c$ can be defined by
\[
  \frac{1}{c} = \int_D \rho(x)\, \dint x.
\]
However, it is desirable to have algorithms
that are able to compute \eqref{int_f_rho} without any pre-computation of $c$. 
Let $\mathcal{F}(D)$ be a class of tuples of the form $(f,\rho)$, 
where $\rho\colon D\to [0,\infty)$ is a possibly unnormalized density
with $\int_D \rho(x) \,\dint x >0$ and for $f$ we assume that  
$f\cdot\rho$ is integrable with respect to the Lebesgue measure. 
Then the goal is to compute 
\begin{equation}  \label{S_f_rho}
	S(f,\rho)= \frac{\int_{D} f(x) \rho(x)\, \dint x}{\int_{D} \rho(x)\, \dint x},
\quad\text{for} \quad (f,\rho)\in \mathcal{F}(D).
\end{equation}
The solution operator $S$ is linear in $f$ but not in $\rho$. 
Hence $S$ is a nonlinear functional.\\

We assume that there are two procedures, $\text{Or}_f$ and $\text{Or}_\rho$, 
which provide information of $f$ and $\rho$, respectively. These procedures are considered as ``black boxes'' 
and we call them oracles. Let the oracle $\text{Or}_f$ be
a procedure which returns for an input $x\in D$ the function value $f(x)$, i.e. $\text{Or}_f(x)=f(x)$.   
Unless stated otherwise we also assume that
the oracle $\text{Or}_\rho$ provides for $x\in D$ the function value of $\rho(x)$, i.e. 
$\text{Or}_\rho(x)=\rho(x)$.
We assume that the cost of an oracle call is much more expensive 
than the cost of arithmetic operations.
Hence we count the total number of oracle calls which are needed to approximate $S(f,\rho)$.\\

Let $\text{{\bf A}lg}_{n}$ be the class of all randomized algorithms which at most use $n$ calls of the  
oracle $\text{Or}_f$
and $n$ calls of the oracle $\text{Or}_\rho$. 
More precisely $A_{n}\in \text{{\bf A}lg}_{n}$ 
is a mapping described by a function $\phi_{2n} :\R^{2n} \to \R$ such that
\[
A_{n}(f,\rho)=\phi_{2n}(\text{Or}_f(X_1),\dots,\text{Or}_f(X_{n}),
    \text{Or}_\rho(X_1),\dots,\text{Or}_\rho(X_{n})).
\]
The sample $(X_1,\dots,X_{n})\in D^n$ is determined as follows: 
Let $\omega=(\omega_1,\dots,\omega_{n})$ 
be a random element with some distribution $W$.
Then 
\begin{align*}
	X_1 & =	X_1(\omega_1), \\
	X_i & = 	X_i(\text{Or}_f(X_1),\dots,\text{Or}_f(X_{i-1}),\text{Or}_\rho(X_1),\dots,\text{Or}_\rho(X_{i-1}),\omega_i),	\quad i=2,\dots,n.
\end{align*}

The \emph{individual error of} $A_{n}\in\text{{\bf A}lg}_{n}$ applied to $(f,\rho) \in \mathcal{F}(D)$ is, 
as in the previous chapters, 
measured in the mean square sense, such that
\[
  e(A_{n},(f,\rho)) = (\expect \abs{S(f,\rho)-A_{n}(f,\rho)}^2)^{1/2},
\]
where the expectation is taken with respect to $W$.
The \emph{overall error} on $\mathcal{F}(D)$ is 
\[
	e(A_{n}, \mathcal{F}(D))= \sup_{(f,\rho)\in \mathcal{F}(D)} e(A_{n},(f,\rho)).
\]
The \emph{complexity} of the problem \eqref{S_f_rho} on $\mathcal{F}(D)$ is given by
\[
	\comp (\e,d,\mathcal{F}(D))=\min\set{n \mid \text{there exists}\; A_n\in\text{{\bf A}lg}_{n}\;\text{with}\; e(A_{n}, \mathcal{F}(D))\leq \e}.
\]
Note that $d$ is the dimension of the domain $D$.
We want to quantify the complexity of a problem
with respect to the dimension $d$.
The integration problem \eqref{S_f_rho} for the class $\mathcal{F}(D)$ is called 
\emph{polynomially tractable} 
if there exist non-negative numbers $c,q_1$ and $q_2$ such that
\[
	\comp (\e,d,\mathcal{F}(D)) \leq c \, \e^{-q_1} d^{q_2}  \quad \text{for all}\; d\in \N,\, \e \in(0,1).
\] 
Roughly spoken it says that the complexity for computing \eqref{S_f_rho} increases at most polynomially 
in the precision 
$\e^{-1}$ and the dimension $d$. 
For details of the concept of tractability let us refer to 
Novak and Wo{\'z}niakowski \cite{trac_vol1,trac_vol2}.\\

Let us provide a result which motivates an additional term of tractability.
We consider the following class of functions
\[
\fc(D) = \{ (f, \rho ) \mid \Vert f \Vert_\infty \le 1, \ 
\frac{\sup \rho}{\inf \rho}  \le C \} .
\]
In some applications $C$ can be very large, such as $C=10^{20}$.  
Observe that always $S(\fc(D) ) = [-1, 1]$, hence the problem 
is scaled properly. 
In \cite{novak} Math{\'e} and Novak proved  
a lower error bound, see \cite[Theorem~1, p.~678]{novak}. 

\begin{theorem}  \label{Nov_low}
For any $A_n\in\text{{\bf A}{\rm lg}}_n$ one obtains
\[
e(A_{n}, \mathcal{F}_C(D))  \ge  \frac{\sqrt 2}{6} 
\begin{cases}
\sqrt{\frac{C}{2n}}, &  2n\geq C - 1, \\  
\frac{3 C}{C+2n-1}, & 2n < C -1.
\end{cases}
\]
\end{theorem}

For an upper error bound Math{\'e} and Novak consider the simple Monte Carlo algorithm: 
Evaluate the numerator and denominator on a common independent
sample according to the uniform distribution, 
say $(X_1,X_2,\dots,X_n)\in D^n$,  and compute
\[ 
 \vtn(f,\rho)= \frac{\sum_{j=1}^n
 f(X_j)\rho(X_j)}{\sum_{j=1}^n\rho(X_j)}. 
\] 
Note that every $X_j$
is uniformly distributed. 
It is essential that one can sample with respect to the uniform distribution on $D$.
This might be a restrictive assumption.
In \cite[Theorem~2, p.~680]{novak} the following upper error bound is proven.

\begin{theorem} \label{Nov_upp}
For all $n\in\N$ we have
\[
e(\vtn,\fc(D))\leq 2\, \min\set{1,  \sqrt{\frac{2C}{n}}} .  
\] 
\end{theorem}

From Theorem~\ref{Nov_low} and Theorem~\ref{Nov_upp} one obtains that the complexity $\comp(\e,d,\fc(D))$
of \eqref{S_f_rho} is linear in $C$ and 
$\vtn$ is almost optimal, for all $\e \in (0, \frac{1}{2\sqrt{2}})$ it follows that
 \[
  0.02\, C \e^{-2} \leq \comp (\e,d,\fc(D)) \leq  8\, C \e^{-2}.
\]

Hence all algorithms are bad if $C=10^{20}$.
Math{\'e} and Novak suggest to consider a smaller class of densities.
The main goal is to have also tractability with respect to $C$ on a class of functions, say 
$\widetilde{\mathcal{F}}_C(D)$, where the possibly unnormalized densities satisfy 
$\frac{\sup \rho}{\inf \rho}\leq C$. More precisely, 
the integration problem $\eqref{S_f_rho}$ is called \emph{tractable also with respect to C}
if there exist non-negative numbers $c,q_1,q_2$ and $q_3$ such that
\begin{equation} \label{c_trac}
  \comp (\e,d,\widetilde{\mathcal{F}}_C(D)) \leq c\,\e^{-q_1}d^{q_2}[\log C]^{q_3}
\end{equation}
for all $\e\in(0,1)$, $d\in \N$ and $C>1$, see \cite[p.~541]{trac_vol2}.\\
 
With Markov chain Monte Carlo algorithms one can achieve this goal on certain classes of functions.
Let $(X_n)_{n\in\N}$ be a Markov chain with transition kernel $K$ and initial 
distribution $\nu$. Assume that the transition kernel has stationary distribution $\pi_\rho$, where
\[
  \pi_\rho(A)=\frac{\int_A \rho(x)\, \dint x}{\int_D \rho(x) \,\dint x}, \quad A\in \Borel(D), 
  \quad\text{so that}\quad
  S(f,\rho)= \int_D f(x)\, \pi_{\rho}(\dint x).
\]
Under suitable assumption on the Markov chain and on $(f,\rho)\in \widetilde{\mathcal{F}}_C(D)$ one has that
the algorithm 
\[
  S_{n,n_0}(f,\rho)=\frac{1}{n}\sum_{j=1}^n f(X_{j+n_0})
\]
is an approximation of $S(f,\rho)$.
Suppose that for each step of the Markov chain we use a single oracle call of $\text{Or}_\rho$.
Then it follows that $S_{n,n_0}$ needs $n+n_0$ oracle calls of $\text{Or}_\rho$ and $n$ oracle calls of $\text{Or}_f$.
Consequently $S_{n,n_0}\in \text{{\bf A}lg}_{n+n_0}$.


 
\section{Integration with respect to log-concave densities} 
\label{sec_S_f_rho}
Let $r>0$ and let $B(x,r)$ be the $d$-dimensional 
Euclidean ball with radius $r$ around $x\in \R ^d$.
Furthermore let $\ball=B(0,1)$ and $r\ball=B(0,r)$. 
The goal is to compute 
\begin{equation}  \label{S_f_rho_ball}
	S(f,\rho)
	= \frac{ \int_{r\ball} f(x)\rho(x)\, \dint x}{ \int_{r\ball} \rho(x) \,\dint x },
\end{equation}
for $(f,\rho)$ which belong to a certain class of functions.
Let us define the class of functions on a convex body $D\subset \R^d$ rather than on $r\ball$.
We assume that the state space $D$ is equipped with the Borel $\sigma$-algebra $\Borel(D)$.
We consider functions $(f,\rho)$ with the following properties:
\begin{itemize}
 \item Let $\rho$ be strictly positive and log-concave, 
	i.e. for all $x,y \in D$ and $0<\lambda <1 $ one has
      \[
	  \rho( \lambda x + (1-\lambda) y) \geq \rho(x)^\lambda \cdot \rho(y)^{1-\lambda}.
      \]
 \item The logarithm of $\rho$ is Lipschitz continuous, i.e. there exists an $\text{\rm L} \geq 0$
       such that
      \[
	 \abs{ \log \rho(x) - \log \rho(y)  } \leq \text{\rm L} \norm{x-y}{\text{\rm E}}, \quad x,y\in D,
      \]
      where $\norm{\cdot}{\text{\rm E}}$ denotes the Euclidean norm. 
 \item The integrand $f$ satisfies $\norm{f}{p}\leq 1$.
\end{itemize}
For $D=r\ball$ one obtains that $\frac{\sup \rho}{\inf \rho} \leq e^{ 2\text{\rm L} r }$.
Hence $C= e^{2\text{\rm L} r}$ and to have tractability also with respect to $C$, 
see \eqref{c_trac}, the goal is to show an error bound which depends polynomial on $\text{\rm L}\, r$.
In general one has the following classes of functions
\[
    \mathcal{F}^{\text{\rm L}}_p(D) = \set{ (f,\rho)\mid \rho \in \mathcal{R}^\text{\rm L}(D),\; \norm{f}{p}\leq 1 },
\]
where
\[
\mathcal{R}^\text{\rm L}(D) = \set{ \rho>0\mid \rho\; 
		\text{is log-concave},\, \abs{ \log \rho(x) - \log \rho(y)  } \leq \text{\rm L} \norm{x-y}{\text{E}} }.
\]

The idea is to apply the Metropolis algorithm to obtain a Markov chain with
stationary distribution $\pi_\rho$, see Section~\ref{ex_gen}. 
The proposal transition kernel on $(D,\Borel(D))$ is given by the ball walk.
This random walk is used in \cite{novak,expl_error} and studied in
different references of volume computation, see e.g. \cite{lova_simo1,vempala}.\\

The transition kernel of the \emph{$\d$ ball walk} is given by
\[
Q_\d(x,A)=\frac{\vol_d(B(x,\d)\cap A)}{\vol_d(\d \ball)}+
	  \left( 1-\frac{\vol_d(B(x,\d)\cap D)}{\vol_d(\d \ball)} \right)\mathbf{1}_A(x), \quad x\in D,\;A\in \Borel(D),
\]
where $\d>0$ and $\vol_d(A)$ denotes the $d$-dimensional Lebesgue measure of $A\in\Borel(D)$. 
Schematically, a single step of the 
$\d$ ball walk from state $x$ may be viewed as in the procedure~\ref{ball_walk}($x,\d$).

\IncMargin{1em}
\begin{procedure}[htb]

\SetKwFunction{rand}{rand}
\SetKwInOut{Input}{input}\SetKwInOut{Output}{output}
\BlankLine
\Input{	current state $x$,
	radius $\d$.}
\Output{ next state $y$. }
\BlankLine
\BlankLine
Choose $y$ uniformly distributed in $B(x,\d)$\;

\eIf{$y\in D$}{
Return $y$\;
}
{
Return $x$\;  					
}
\caption{Ball-Walk($x,\d$)}
\label{ball_walk}
\end{procedure}
\DecMargin{1em}

Let us state some well known properties.

\begin{lemma}
	The transition kernel $Q_\d$ is reversible with respect to the uniform distribution on $D$. 

\end{lemma}
\begin{proof}
See \cite[Proposition~1, p.~685]{novak}. 
\end{proof}
%

The \emph{local conductance} of the ball walk is defined by
\[
l(x)=\frac{\vol_d(B(x,\d)\cap D)}{\vol_d(\d \ball)}, \quad x\in D.
\]
We call $l$ \emph{a lower bound of the local conductance}, if $l(x)\geq l$ for all $x\in D$.
Note that $l$ might be very small. For $D=[0,1]^d$, the $d$-dimensional unit cube, one obtains even for small $\delta$
that $l=2^{-d}$. However, one can show for $D=r\ball$ 
and $\d \leq r/\sqrt{d+1}$ that $l= 0.3$ is a lower bound of the local conductance. 

\begin{lemma}  \label{local_conduct}
	Let $Q_\d$ be the transition kernel of the ball walk on $D=r\ball$ for $r>0$.
	If $\d\leq r/\sqrt{d+1}$, then $l= 0.3$ is a lower bound of the local conductance of the ball walk.
\end{lemma}
\begin{proof}
The assertion follows by the same arguments as in \cite[Lemma~7, p.~687]{novak}, see also \cite{rudolf_diploma}. 
The only difference is that $r\ball$ is a ball with radius $r$ instead of being the unit ball.
\end{proof}
The Metropolis transition kernel based on the $\d$ ball walk is
\[
K_{\rho,\d}(x,A)=\int_A \theta(x,y)\, Q_\d(x,\dint y)  + \mathbf{1}_A(x)
	\left( \int_D (1-\theta(x,y))\, Q_\d(x,\dint y)\right), 
\]
where the acceptance probability is $\theta(x,y) = \min \set{ 1,\frac{\rho(y)}{\rho(x)} }$ for
$x,y \in D$ and $A\in \Borel(D)$. The lazy version of $K_{\rho,\d}$ is 
denoted by $\widetilde{K}_{\rho,\d}$. The transition kernel $\widetilde{K}_{\rho,\d}$ is reversible with respect to $\pi_\rho$.
In Algorithm~\ref{mcmc_metropolis} we present the integration algorithm $S_{n,n_0}^{\d}$ 
which uses the lazy version of the Metropolis transition kernel with proposal transition kernel $Q_\d$.

\IncMargin{1em}
\begin{algorithm}[htb]

\SetKwFunction{rand}{rand}
\SetKwInOut{Input}{input}\SetKwInOut{Output}{output}
\BlankLine
\Input{ $n$, $n_0$, $\d$, $(f,\rho)$.}
\Output{ $S_{n,n_0}^{\d}(f,\rho)$. }
\BlankLine
\BlankLine
Choose $X_1$ uniformly distributed in $D$\;
\For{$k=1$ \KwTo $n+n_0$}{

\eIf{$\rand{}>0.5$}{
$X_{k+1}:=X_{k}$\;
}
{
$Y:=$\ref{ball_walk}($X_k,\,\d$)\;	
\uIf{$\rho(Y)/\rho(X_i)\geq\rand{}$}
{
$X_{i+1}:=Y$\;}
\Else{
$X_{i+1}:=X_i$\;
}
}
}
Compute
\[
S_{n,n_0}^{\d}(f,\rho) := \frac{1}{n} \sum_{i=1}^n f(X_{i+n_0}).
\]
\caption{$S_{n,n_0}^{\d}$}
\label{mcmc_metropolis}
\end{algorithm}
\DecMargin{1em}
It is convenient to use the notation $P_K=P$, $\beta_K = \beta$ and $\Lambda_K=\Lambda$
to indicate the transition kernel $K$.
The following lemma provides a lower bound of the
$L_2$-spectral gap of $P_{\widetilde{K}_{\rho,\d}}$.
The lemma follows from a result of Math{\'e} and Novak 
presented in \cite[Theorem~4, p.~690]{novak}, where an estimate of the conductance
of ${K}_{\rho,\d}$ is shown.
\begin{prop}  \label{spec_gen_metro}
For $r>0$ let $D\subset \R^d$ be a convex body with 
 \[
    \text{\rm {\bf d}iam}(D)=\sup\set{\norm{x-y}{\text{\rm E}}\mid x,y \in D} \leq 2r.
  \]
 Let $l$ be a lower bound of the local conductance of the $\d$ ball walk.
 Then, for all $\rho \in \mathcal{R}^\text{\rm L}(D)$ one has for the lazy version of the Metropolis transition kernel based on
a $\d$ ball walk, given by $\widetilde{K}_{\rho,\d}$, that
  \[
    1-\beta_{\widetilde{K}_{\rho,\d}} \geq 
    \frac{l^2 e^{-2\text{\rm L}\d}}{256} \min\set{ \frac{\pi}{8} \frac{l^2 \d^2}{r^2 (d+1)},1 }. 
  \]
\end{prop}
\begin{proof}
One has
$\beta_{\widetilde{K}_{\rho,\delta}}=\Lambda_{\widetilde{K}_{\rho,\delta}}= \frac{1}{2}(1+\Lambda_{K_{\rho,\delta}})$.
The conductance of ${K}_{\rho,\delta}$ is defined by
\[
\phi({K}_{\rho,\delta},\pi_\rho)=
\inf_{0<\pi_\rho(A)\leq \frac{1}{2}}  \frac{\int_A {K}_{\rho,\delta}(x,A^c)\, \pi_\rho(\dint x)}{\pi_\rho(A)}.
\]
One can use the Cheeger inequality, see Proposition~\ref{cheeger}.
It states that
\[
1-\Lambda_{{K}_{\rho,\delta}} \geq \frac{\phi({K}_{\rho,\delta},\pi_\rho)^2}{2}.
\]
Altogether one obtains
\begin{equation} \label{beta_lazy_metro}
	1-\beta_{\widetilde{K}_{\rho,\delta}} = \frac{1}{2}(1-\Lambda_{K_{\rho,\delta}})
      \geq \frac{\phi(K_{\rho,\delta},\pi_\rho)^2}{4}.
\end{equation}
    In \cite[Theorem~4, p.~690]{novak} it is shown that 
    \[
      \phi(K_{\rho,\d},\pi_\rho) \geq \frac{l e^{-\text{\rm L}\d}}{8} \min\set{\sqrt{\frac{\pi}{8}} \frac{l\d}{r\sqrt{d+1}},1}.
    \]
    This lower bound is plugged into \eqref{beta_lazy_metro} and the assertion is proven.
\end{proof}

In the previous result one can see that the lower bound 
of the local conductance is crucial. 
This motivates that we consider $D=r\ball$, 
since by Lemma~\ref{local_conduct} a lower bound of the local conductance is provided. 
An immediate consequence of the last proposition follows.

\begin{coro}  \label{metro_spec_absch}
 For $r>0$ let $D=r\ball$, assume that $\rho \in \mathcal{R}^\text{\rm L}(r\ball)$ and  
 set 
$\d^*= \min \set{\frac{1}{\text{\rm L}}, \frac{r}{\sqrt{d+1}}}$.
 Then we have 
  \[
    1-\beta_{\widetilde{K}_{\rho,\d^*}} \geq 
     \frac{1.69 \cdot 10^{-6}}{d+1}\, \min\set{ \frac{1}{r^2\,\text{\rm L}^2} , \frac{1}{d+1} }.
  \]

\end{coro}
\begin{proof}
 The assertion is implied by Proposition~\ref{spec_gen_metro} and Lemma~\ref{local_conduct}.
\end{proof}

In particular one obtains that the lazy version of the ball walk has an $L_2$-spectral gap, since one can
consider constant densities where $\text{\rm L}=0$.

\begin{coro}
 For $r>0$ let $D=r\ball$ and
 let $\d = \frac{r}{\sqrt{d+1}}$. 
 Then the lazy version $\widetilde{Q}_{\d}$ of the transition kernel of the ball walk obeys
  \[
    1-\beta_{\widetilde{Q}_{\d}} \geq 
    \frac{1.69 \cdot 10^{-6}}{(d+1)^2}.
  \]
\end{coro}

Now we can apply the error bounds of Section~\ref{err_bound_gen}.
The next theorem states an error bound for $S_{n,n_0}^{\d^*}(f,\rho)$ 
where $(f,\rho)\in \mathcal{F}^\text{\rm L}_p(r\ball)$.

\begin{theorem} \label{err_S_f_rho}
For $r>0$ let $D=r\ball$ and let $\nu$ be the uniform distribution on $r\ball$.   
Let
  $\rho \in \mathcal{R}^\text{\rm L}(r\ball)$ and 
  $\delta^*=\min \set{\frac{1}{\text{\rm L}}, \frac{r}{\sqrt{d+1}}}$. 
  Let $(X_n)_{n\in\N}$ be a Markov chain with transition kernel $\widetilde{K}_{\rho,\delta^*}$ and initial distribution $\nu$.
    The approximation
  of $S(f,\rho)$ is
  \[
      S_{n,n_0}^{\d^*}(f,\rho) = \frac{1}{n} \sum_{i=1}^n f(X_{i+n_0}).
  \]
  For $p\in(2,\infty]$ recall that
  \[
      \mathcal{F}^\text{\rm L}_p(r\ball)=\set{ (f,\rho) \mid \rho\in\mathcal{R}^\text{\rm L}(r\ball),\, \norm{f}{p}\leq1}.
  \]
  Let $n_0(p)$ be the smallest natural number (including zero) greater than or equal to
  \[
    5.92\cdot 10^6\,(d+1) \max\set{r^2\,\text{\rm L}^2, d+1} 
    \cdot\begin{cases}
      \frac{p}{(p-2)}\,( \text{\rm L}\, r + 0.5\log\frac{32p}{p-2}) ,	& 	p\in(2,4),\\
      2 \text{\rm L}\, r + 4.16,			&	p\in[4,\infty].
    \end{cases}
  \]
  Then
  \begin{align*}
    e(S_{n,n_0(p)}^{\d^*},\mathcal{F}^\text{\rm L}_p(r\ball)) & \leq 
     \frac{1089}{\sqrt{n}} \sqrt{d+1} \max\{r\,\text{\rm L},\sqrt{d+1}\,\} \\ 
    & \qquad \qquad+\frac{8.38\cdot 10^5}{n} (d+1)\max\{r^2\,\text{\rm L}^2,(d+1)\}.
  \end{align*}
\end{theorem}
  \begin{proof}
    The initial distribution obeys
    \[
      \nu(A)= \frac{\vol_d(A)}{\vol_d(r\ball)}
	=\frac{1}{\vol_d(r\ball)} \int_A \int_{r\ball} \frac{\rho(y)}{\rho(x)}\, \dint y\; \pi_\rho(\dint x),
	\quad A\in \Borel(r\ball).
    \]
    Since $ \log \rho$ is Lipschitz continuous with Lipschitz constant $\text{\rm L}$ we obtain  
    \[
	e^{-2\text{\rm L} r} \leq \frac{\rho(y)}{\rho(x)} \leq e^{2\text{\rm L} r}, \quad x,y\in r\ball,
    \]
    so that
    \[
    \norm{\frac{d\nu}{d\pi_\rho}-1}{p} \leq \norm{\frac{d\nu}{d\pi_\rho}-1}{\infty} 
\leq \max\set{1 , e^{2\text{\rm L} r}} = e^{2\text{\rm L} r}.
    \]
    By Corollary~\ref{metro_spec_absch} we have the crucial 
    lower bound for the spectral gap $1-\beta_{\widetilde{K}_{\rho,\delta^*}}$ and consequently
    Theorem~\ref{main_coro_gen}\,\eqref{zweitens_burn_in} can be applied which
    proves the assertion.
  \end{proof}

Note that $p\in(2,\infty]$ is necessary to apply Theorem~\ref{main_coro_gen}\,\eqref{zweitens_burn_in}.
An essential consequence of the last theorem is the following result 
concerning the tractability of \eqref{S_f_rho_ball}.

\begin{theorem}  \label{theo_S_f_rho}
For the integration problem $S(f,\rho)$ defined over $\mathcal{F}^\text{\rm L}_p(r\ball)$ with
$r>0$ and $p>2$ we have
\begin{align*}
& \comp (\e,d,\mathcal{F}^\text{\rm L}_p(r\ball)) 
    \leq (d+1)\max\set{r^2\, \text{\rm L}^2,d+1} \\
& \qquad \qquad \qquad \qquad 
    \cdot \left[ 4.8\cdot 10^6\, \e^{-2}+
    1.2\cdot 10^{6} \cdot
     \begin{cases}
      \frac{p}{(p-2)}\,( \text{\rm L}\, r + 0.5\log \frac{32p}{p-2} ),	& 	p\in(2,4)\\
      2 \text{\rm L}\, r + 4.16,			&	p\in[4,\infty]
    \end{cases}
  \right]
\end{align*}
for all $\e\in(0,1)$ and $d\in \N$.
\end{theorem}

The last theorem states that the problem \eqref{S_f_rho_ball} is polynomially tractable.
Roughly spoken, for fixed $p$ one obtains
\[
 \comp (\e,d,\mathcal{F}^\text{\rm L}_p(r\ball)) \prec d \max\set{r^2\, \text{\rm L}^2,d} (\e^{-2} + \text{\rm L}\, r ),
\]
so that the dependence on $\text{\rm L}$, on the precision $\e$, dimension $d$ and $r$ is polynomial.
We have tractability also with respect to $C=e^{2r{\rm L}}$, inequality \eqref{c_trac} holds 
with $q_1=2$, $q_2=2$ and $q_3=3$.
For $p\in [4,\infty]$ the complexity can be bounded independently of $p$.
If $p$ converges to $2$ the result is restrictive.
However, for fixed $p\in(2,\infty]$ we showed that the integration problem on $\mathcal{F}^\text{\rm L}_p(r\ball)$
is polynomially tractable in the sense of \eqref{c_trac}.  

\section{Integration over a convex body}  \label{int_conv}
The goal is to compute 
\begin{equation}  \label{S_f_A}
	S(f,A)
	    = \frac{1}{\vol_d(A)} \int_A f(x)\, \dint x, 
\end{equation}
with $A\subset \R^d$.
In other words, $S(f,A)$ is the expectation of $f$ with respect 
to the uniform distribution, say $\mu_A$, on $A\subset \R^d$.
The domain $A$ and the function $f$ are the input quantities.
It fits in the class of problems described by $\eqref{S_f_rho}$
if we assume that $A\subset D$. Then $\mu_A$ might be considered as
given by a density which is an indicator function.\\

For some domains $A$ it is indeed simple to generate uniformly distributed random points, e.g. 
the Euclidean unit ball or the unit cube. 
Then one can approximate $S(f,A)$ by Monte Carlo methods with an i.i.d. sample.
However, here $A$ is part of the input to the algorithm, thus
the problem $S(f,A)$ shall be solved uniformly for a class of state spaces, where we 
cannot assume that sampling with respect to the uniform distribution is possible.\\ 

Let $r\geq1$ and let 
  \[
    \mathcal{S}_d(r)=\set{ A\subset \R^d\; \text{convex} \mid \ball\subset A \subset r\ball}.
  \]
If $A \in \mathcal{S}_d(r)$ then $A$ is a convex bounded set with non-empty interior which contains the origin.
The class of input parameters is given by
\[
  \mathcal{F}_p(r,d)=\set{ (f,A)\mid \norm{f}{p}\leq 1,\; 
			 A\in \mathcal{S}_d(r)}.
\]
We assume that for any $A \in \mathcal{S}_d(r)$ there exists an oracle $\text{Or}_A(\ell)$ 
which returns for an arbitrary line $\ell$ 
a uniformly distributed random point on $A\cap\ell$.\\

Let us comment this assumption. 
Assume that we have a membership oracle of $A\in \mathcal{S}_d(r)$ 
which is given by 
$\widetilde{\text{Or}}_{A}(x)=\mathbf{1}_A(x)$
for any $x\in r \ball$.
The oracle $\text{Or}_A$
can be implemented by using the membership oracle. 
Let $[x,y]=\set{x+ty\mid t\in[0,1]}$ be the segment of $x,y\in\R^d$ with Euclidean distance $\norm{x-y}{\text{E}}$.
By the convexity of $A$
it follows that $A\cap \ell$ is a single segment, hence there exist $a_1,a_2\in\R^d$ 
such that $[a_1,a_2]=A\cap\ell$. Suppose that $\ell=\set{\tilde{x}+t\, \text{{\bf d}ir}\mid t\in \R}$ with $\tilde{x}\in A$
and assume that there is a positive number $\e_0$ such that $\norm{a_1-a_2}{\text{E}}\geq\e_0$.
We use that $A\in \mathcal{S}_d(r)$ and $\tilde{x}\in A$.
By a bisection method
one can find 
with at most $3\log(\frac{2r}{\e_0})+2$ calls of 
the membership oracle $\widetilde{\text{Or}}_{A}$, a segment $[b_1,b_2]$ with $b_1,b_2\in\R^d$ and $[a_1,a_2]\subset [b_1,b_2]$ such that
\[
  \frac{1}{6}\norm{b_1-b_2}{\text{E}} 
\leq \norm{a_1-a_2}{\text{E}} 
\leq \norm{b_1-b_2}{\text{E}}.
\] 
Then, choose a uniformly distributed random point in $[b_1,b_2]$ and accept it, if it is in $A$,
otherwise reject it and repeat the acceptance rejection procedure. 
This procedure gives a uniformly distributed random point in $A\cap \ell$ and
works reasonably fast, since the acceptance probability is $1/6$.
Altogether an oracle call of $\text{Or}_A$ requires 
at most an expected number of $3\log(\frac{2r}{\e_0})+8$ oracle calls of $\widetilde{\text{Or}}_{A}$. 
In the analysis of the error we count the calls of the oracle $\text{Or}_A$ 
and the function evaluations of $f$, i.e. the calls of the oracle $\text{Or}_f$.\\

Now let us provide a Markov chain on the measurable space $(A,\Borel(A))$ with stationary distribution $\mu_A$. 
We consider the hit-and-run algorithm, also called hypersphere directions algorithm, see \cite{hit_smith}. 
The algorithm is studied and analyzed in \cite{hit_and_run_fast,hit_and_run_from_corner}.
The work of Vempala \cite{vempala} provides an introduction to geometric random walks.\\ 

The algorithm is as follows.
Suppose that the current position is $X_i\in A$ with $i\in\N$. 
Then choose a uniformly distributed direction, say $\text{{\bf d}ir}_i$, and consider the 
line which is defined by $\ell^{(i)}=\set{X_i+t\,\text{{\bf d}ir}_i\mid t\in\R}$. 
Apply $\text{Or}_A(\ell^{(i)})$, which gives 
the next state $X_{i+1}$ chosen uniformly distributed in $\ell^{(i)}\cap A$.
Then, again, a uniformly distributed direction, say $\text{{\bf d}ir}_{i+1}$, is generated and the 
next state is chosen uniformly distributed on $\ell^{(i+1)}\cap A$ by $\text{Or}_A(\ell^{(i+1)})$.
Two consecutive steps of the hit-and-run algorithm are illustrated in Figure~\ref{h_a_r}.
\begin{figure}[htb]
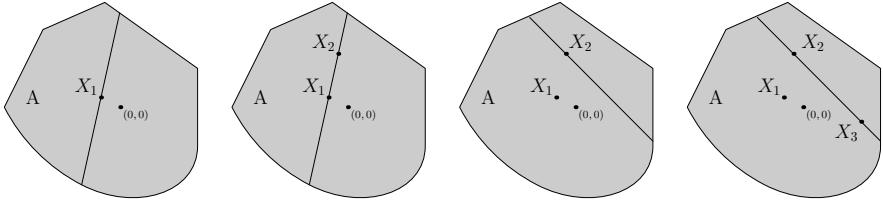
 
\centering
  \includegraphics[width=2.6cm]{graphics.3}
  \hspace{0.2cm}
  \includegraphics[width=2.6cm]{graphics.4}
  \hspace{0.2cm}
  \includegraphics[width=2.6cm]{graphics.6}
  \hspace{0.2cm}
  \includegraphics[width=2.6cm]{graphics.7}
\vspace*{0.5ex}
  \caption{
Illustration of the generation of $X_3$ and $X_2$ by the hit-and-run algorithm given state $X_1$. 
}
\label{h_a_r}
\end{figure}
Recall that the Euclidean unit ball is denoted by $\ball$ 
and its boundary is denoted by $\partial \ball$.
Schematically, a single step of the hit-and-run algorithm 
from $x\in A$ is presented in the Procedure~\ref{hit_and_run}($x$).\\
\IncMargin{1em}
\begin{procedure}[H]
\SetKwFunction{rand}{rand}
\SetKwInOut{Input}{input}\SetKwInOut{Output}{output}
\BlankLine
\Input{	current state $x$.}
\Output{ next state $y$. }
\BlankLine
\BlankLine
Choose a direction $\text{{\bf d}ir}$ uniformly distributed on $\partial \ball$\;
Choose $y$ uniformly distributed on
\[
 A \cap \set{ x+t\, \text{{\bf d}ir} \mid t \in \R };
\]
Return $y$.					
\caption{Hit-and-Run($x$)}
\label{hit_and_run}
\end{procedure}
\DecMargin{1em}
The transition kernel of the hit-and-run algorithm follows.
For any $x,y\in \R^d$ let 
\[
	\text{{\bf I}nt}(x,y)=\set{\lambda \in \R \mid x + \lambda \frac{y-x}{\norm{y-x}{\rm E}}\in A}.
\]
Since $A$ is convex, $\text{{\bf I}nt}(x,y)$ is an interval. 
Let
\[
	\lambda_1(x,y)= \min\set{\a \mid \a\in \text{{\bf I}nt}(x,y)}
\quad \text{and} \quad
	\lambda_2(x,y)= \max\set{\a \mid \a\in \text{{\bf I}nt}(x,y)},
\]
which implies that $\text{{\bf I}nt}(x,y)=[\lambda_1(x,y),\lambda_2(x,y)]$.
The length of the chord $\text{{\bf I}nt}(x,y)$ is given by $\ell(x,y)=\lambda_2(x,y)-\lambda_1(x,y)$.
Let $U(x,y)$ be a uniformly distributed random variable in the interval $\text{{\bf I}nt}(x,y)$.
Then the \emph{hit-and-run transition kernel} $H$ of the hit-and-run algorithm is 
\begin{align}
	H(x,C) 
	& = \frac{\int_{\partial \ball}	\notag
		\Pr[x+U(x,x+\theta) \theta \in C] \, \dint \theta }
		{\vol_{d-1}(\partial \ball)}\\
\displaybreak
	& = \frac{1}{\vol_{d-1}(\partial \ball)}  \notag
		\int_{\partial \ball} \int_{\lambda_1(x,x+\theta)}^{\lambda_2(x,x+\theta)} 
		\frac{\mathbf{1}_C(x+\lambda \theta)}
		{ \ell(x,x+\theta) }
		\,\dint \lambda\, \dint \theta \\
	& = \frac{1}{\vol_{d-1}(\partial \ball)}  \notag 
		\int_{\partial \ball} \int_{\lambda_1(x,x+\theta)}^{0} 
		\frac{\mathbf{1}_C(x+\lambda \theta)}{\ell(x,x+\theta) }
		\,\dint \lambda \,\dint \theta \\	
	& \qquad+ \frac{1}{\vol_{d-1}(\partial \ball)} \notag
		\int_{\partial \ball} \int_{0}^{\lambda_2(x,x+\theta)} 
		\frac{\mathbf{1}_C(x+\lambda \theta)}{ \ell(x,x+\theta) }
		\,\dint \lambda \,\dint \theta \\							 
	& = \frac{2}{\vol_{d-1}(\partial \ball)}  \label{hit_eq}
	        \int_C \frac{1 \,\dint y}{\ell(x,y) \norm{x-y}{\text{E}}^{d-1}},
\end{align}
where $x\in A$ and $C\in \Borel(A)$.
The last equality follows by the integral transformation formula 
\[
	\int_{\R^d} h(y)\, \dint y 
= \int_{\partial \ball} \int_0^\infty h(g(\lambda, \theta))\, \lambda^{d-1}\, \dint \lambda\, \dint \theta
\]
with
\[
h(y)=\frac{\mathbf{1}_C(y)}{\ell(x,y)\norm{x-y}{\text{E}}^{d-1}} 
\]
and either
 $g(\lambda, \theta)=x+\lambda \theta$ or $g(\lambda, \theta)=x-\lambda \theta$.
\begin{lemma} \label{hit_and_run_reversible}
The hit-and-run transition kernel $H$, given by \eqref{hit_eq},
is reversible with respect to $\mu_A$ on $A$.
\end{lemma}
\begin{proof}
Let $k(x,y)$ be a symmetric transition density of a transition kernel $K$,
i.e. $k(x,y)=k(y,x)$ for all $x,y\in A$.
Then it follows by Fubini's theorem that
  \begin{align*}
			\int_B K(x,C) \,\mu_A(\dint x)  = \int_B \int_C k(x,y)\, \mu_A(\dint y)\, \mu_A(\dint x) 
			& = \int_C \int_B k(x,y)\, \mu_A(\dint x)\, \mu_A(\dint y) \\
			 = \int_C \int_B k(y,x) \,\mu_A(\dint x)\, \mu_A(\dint y)
			& = \int_C K(x,B)\, \mu_A(\dint x),\;\, B,C \in \Borel(A).
		\end{align*}
		Hence the transition kernel $K$ is reversible with respect to $\mu_A$.
		Since $\ell(x,y)=\ell(y,x)$, one obtains that the transition kernel $H$ has a symmetric density 
		and this implies that it is reversible with respect to $\mu_A$.	
\end{proof}
The lazy version of $H$ is denoted by $\widetilde{H}$.
In Algorithm~\ref{mcmc_hit_and_run} 
we present the integration algorithm $S_{n,n_0}^{\text{har}}$ 
which uses the lazy version of the hit-and-run transition kernel.  
\IncMargin{1em}
\begin{algorithm}[htb]
\SetKwFunction{rand}{rand}
\SetKwInOut{Input}{input}\SetKwInOut{Output}{output}
\BlankLine
\Input{ $n$, $n_0$, ($f$, $A$).}
\Output{ $S_{n,n_0}^{\text{har}}(f,A)$. }
\BlankLine
\BlankLine
choose $X_1$ uniformly distributed in $\ball$\;
\For{$k=1$ \KwTo $n+n_0$}{

\eIf{$\rand{}>0.5$}{
$X_{k+1}:=X_{k}$\;
}
{
$X_{k+1}:=$\ref{hit_and_run}($X_k$)\;		
}
}
Compute
\[
S_{n,n_0}^{\text{har}}(f,A) := \frac{1}{n} \sum_{i=1}^n f(X_{i+n_0}).
\]
\caption{$S_{n,n_0}^{\text{har}}$}
\label{mcmc_hit_and_run}
\end{algorithm}
\DecMargin{1em}
We use the notation $P_K=P$, $\beta_K = \beta$ and $\Lambda_K=\Lambda$
to indicate the transition kernel $K$.
The following lemma provides a lower bound of the $L_2$-spectral gap of $P_{\widetilde{H}}$.
The lemma is a straightforward implication of a result of Lov{\'a}sz and Vempala 
presented in \cite[Theorem~4.2, p.~993]{hit_and_run_from_corner}.
Lov{\'a}sz and Vempala show an estimate of the conductance of $H$.

\begin{prop}  \label{absch_beta}
Let $r\geq1$. Then, for all $A\in \mathcal{S}_d(r)$ one has for the lazy version 
of the hit-and-run transition kernel, given by $\widetilde{H}$, that 
\[
1-\beta_{\widetilde{H}} \geq 2^{-52}  (d r)^{-2}. 
\]
\end{prop}
\begin{proof}
In \cite[Theorem~4.2, p.~993]{hit_and_run_from_corner} it is proven that
\[
	\phi(H,\mu_A) \geq 2^{-25} (d r)^{-1}.
\] 
Then the proof follows by the same arguments as the proof of Lemma~\ref{spec_gen_metro}.
\end{proof}
Now we can apply the error bounds of Section~\ref{err_bound_gen} and obtain the following.

\begin{theorem}  \label{err_S_f_A}
Let $\nu$ be the uniform distribution on $\ball$.
Let $(X_n)_{n\in\N}$ be a Markov chain with transition kernel $\widetilde{H}$ and initial distribution $\nu$.
The 
approximation of $S(f,A)$ is
\[
   	S_{n,n_0}^{\text{{\rm har}}}(f,A) = \frac{1}{n} \sum_{i=1}^n f(X_{i+n_0}).
\]
For $r\geq1$ and $p>2$ recall that
\[
  \mathcal{F}_p(r,d)=\set{ (f,A)\mid \norm{f}{p}\leq 1,\; 
			 A\in \mathcal{S}_d(r)}.
\] 
Let $n_0(p)$ be the smallest natural number (including zero) greater than or equal to 
\[
  4.51\cdot 10^{15}\, d^2\, r^2  \cdot 
  	\begin{cases}
        	\frac{p}{2(p-2)} ( d \log r +\log\frac{32p}{p-2}) ,  	& p\in(2,4),\\
  		d \log r+ 4.16,			& p\in[4,\infty].
  	\end{cases}
\]
Then
\[
  e(S_{n,n_0(p)}^{\text{\rm har}},\mathcal{F}_p(r,d)) 
		\leq 9.5 \cdot 10^{7}\, \frac{ d r}{\sqrt{n}}+ 6.4 \cdot 10^{15} \,\frac{d^2\, r^2}{n}.
\]	
\end{theorem}
\begin{proof}
Note that the initial distribution $\nu$ is well defined, since for $A\in \mathcal{S}_d(r)$
one has $\ball \subset A \subset r\ball$.
Furthermore, it follows that
\[
	\nu(C)= \frac{1}{\vol_d(\ball)} \int_C \mathbf{1}_{\ball}(x) \, \dint x 
	      =	\frac{1}{\vol_d(A)} \int_C \mathbf{1}_{\ball}(x)\frac{\vol_d(A)}{\vol_d(\ball)}\, \dint x
			,\quad
				C\in \Borel(A).
\] 
One obtains
\begin{align*}
	\norm{\frac{d\nu}{d\mu_A}-1}{p} 
      & \leq \norm{\frac{d\nu}{d\mu_A}-1}{\infty} 
%
%
	\underset{r\in[1,\infty)}{\leq} \frac{\vol_d(r\ball)}{\vol_d(\ball)}= r^d.
\end{align*}
By Lemma~\ref{absch_beta} we have the crucial lower bound for the spectral gap $1-\beta_{\widetilde{H}}$
and consequently Theorem~\ref{main_coro_gen}~\eqref{zweitens_burn_in} can be applied.
Hence the assertion is proven.
\end{proof}

Note that $p>2$ is necessary to apply Theorem~\ref{main_coro_gen}\,\eqref{zweitens_burn_in}.
A consequence of the last theorem is the following result 
concerning the tractability of the integration problem \eqref{S_f_A}.

\begin{theorem}  \label{comp_S_f_rho}
For the integration problem $S(f,A)$ defined over $\mathcal{F}_p(r,d)$ with
$r\geq1$ and $p>2$ we have
\begin{align*}
 \comp (\e,\mathcal{F}_p(r,d)) &
    \leq
      d^2 r^2 \left[ 4\cdot 10^{16}\, \e^{-2} +
         5\cdot 10^{15}  	
        \begin{cases}
        	\frac{p}{2(p-2)} ( d \log r +\log\frac{32p}{p-2} ),  	& p\in(2,4)\\
  		d \log r+ 4.16,			& p\in[4,\infty]
  	\end{cases}
     \right]
\end{align*}
for all $\e\in(0,1)$ and $d\in \N$.
\end{theorem}

The last theorem states that \eqref{S_f_A} is polynomially tractable.
Roughly spoken for fixed $p$ one obtains
\[
 \comp (\e,\mathcal{F}_p(r,d)) \prec d^2 r^2 (\e^{-2} + d \log r ),
\]
so that the dependence on the precision $\e$, dimension $d$ and $r$ is polynomial.
For $p\in [4,\infty]$ the complexity can be bounded independently of $p$.
If $p$ converges to $2$ the result is restrictive.
However, for fixed $p>2$ we showed that the integration problem is polynomially tractable on $\mathcal{F}_d(r,p)$. 

\section{Notes and remarks}
Let us briefly summarize the features of the last sections and provide additional results of the literature.
In Section~\ref{sec_S_f_rho} elementary state spaces were considered, namely balls, 
and the distribution $\pi_\rho$ determined by $\rho$ could be complicated. 
In Section~\ref{int_conv} the distribution of interest was simple, namely the uniform one, and
the state space was possibly complicated.\\

The problem of integration \eqref{int_f_rho}, stated in the form 
\begin{equation*}  
	S(f,\rho)= \frac{\int_{D} f(x) \rho(x)\, \dint x}{\int_{D} \rho(x)\, \dint x}
\end{equation*}
 is formulated as in the work of Math{\'e} and Novak \cite{novak}. 
There the authors also proved an asymptotic error bound of the Metropolis algorithm based on the ball walk 
on $\mathcal{F}^\text{\rm L}_2(\ball)$.
They studied the algorithm $S_{n,0}^{\d^*}$ 
and for $\d^*=\min\set{(d+1)^{-1/2},\text{\rm L}^{-1}}$  it is shown in \cite[Theorem~5, p.~693]{novak} that
\[
  \lim_{n \to \infty}  n \cdot
  e(S_{n,0}^{\d^*},\mathcal{F}^\text{\rm L}_2(\ball))^2 
  \leq 594700 \cdot (d+1) \max\set{d+1,\text{\rm L}^2}. 
\]
The first non-asymptotic error bound is proven in \cite{expl_error} 
for the class $\mathcal{F}_\infty^{\text{\rm L}}(\ball)$.
It states that for $n_0 \geq 1.28\cdot10^6\cdot \text{\rm L} (d+1) \max\set{d+1,\text{\rm L}^2}$
the error obeys
  \[
    e(S_{n,n_0}^{\d^*},\mathcal{F}^\text{\rm L}_{\infty}(\ball))
\leq \frac{8000}{\sqrt{n}}\sqrt{d+1} \max\set{ \sqrt{d+1},\text{\rm L}}.
  \] 
Theorem~\ref{err_S_f_rho} extends the result. 
The integrands $f$ belong to $L_p$
for $p>2$ and we considered the domain $r\ball$. 
The constants in the error bound are of the same magnitude and the 
dependence on the dimension $d$, the Lipschitz constant $\text{\rm L}$ and the precision $\e$ is the same.
The problem is tractable in the sense of \eqref{c_trac}.\\

Apart of the asymptotic result of \cite[Theorem~5, p.~693]{novak} it is always assumed that
the integrand $f$ belongs to $L_p$ for $p>2$.
The case of $f\in L_2$ is not covered so far.
To apply Theorem~\ref{main_thm_unif} it is sufficient to have a transition kernel which is 
reversible with respect to the desired distribution and uniformly ergodic with $(\a,M)$.
It is well known that the ball walk, the Metropolis algorithm based on the ball walk
and the hit-and-run algorithm
are uniformly ergodic, see \cite{hit_smith,KaSm,novak}. 
However, as far as we know there is no estimate of the numbers $\a \in [0,1)$ and $M<\infty$,
of the uniform ergodicity with $(\a,M)$, to obtain polynomial tractability.
We get polynomial tractability 
if there exist non-negative numbers $c$ and $q$, such that  $(1-\a)^{-1} \leq c\, d^q $.
One can prove the following.
%
%
%
Let $D=\ball$ and $\d= 2/\sqrt{d+1}$.
Then the ball walk $Q_\d$ is uniformly ergodic with $(\a,M)$, where 
\[
\a = 1-\frac{0.15}{\sqrt{d+1}((d+1)2^{d+1})^{\sqrt{d+1}}}
\quad  \text{and} \quad 
M=100.
\]
Unfortunately the crucial quantity $(1-\a)^{-1}$ is exponentially bad in $d$. 
Hence, this is not enough to prove polynomial tractability.
It is not clear if one can get a significantly better $\a$.\\

The hit-and-run algorithm is studied in different references of volume computation
and optimization. However, as far as we know it was not yet applied to 
integration problems of the form of \eqref{S_f_A}.
There is an immediate generalization of the hit-and-run algorithm which can be used to sample a
distribution given by a log-concave density, see for example \cite[p.~987]{hit_and_run_from_corner}.
This might be used to obtain further error bounds for other classes of functions.
\\

\begin{appendix}


\chapter{Appendix}
Some aspects of Functional Analysis are fundamental for the understanding of
the error of Markov chain Monte Carlo. 
We present the Spectral Theorem for linear, bounded and self-adjoint operators.
Then we state the Interpolation Theorem of Riesz-Thorin for operators acting on $L_p$.
Afterwards the conductance and the Cheeger inequality is introduced.

\section{Spectral Theorem} \label{sec_spec}
We state the Spectral Theorem for linear, self-adjoint and bounded operators.
For further reading, proofs and details we refer to \cite{kirillov,rudin,triebel}. 
For an introduction see \cite{kreyszig}.\\

Let $H$ be a real or complex Hilbert space and let $\Le(H)$ 
be the space of all linear and bounded operators mapping from $H$ to $H$. 
Let $\Borel(\R)$ be the Borel $\sigma$-algebra over $\R$.
\begin{defin}[\index{spectral measure}spectral measure]
A \emph{spectral measure} or a \emph{projection-valued measure} is a mapping 
$E \colon \Borel(\R) \to \Le(H)$ with the following properties:
\begin{enumerate}[(i)]
\item for all $A\in\Borel(\R)$ the operator $E_A$ is an orthogonal projection,
\item $E_\emptyset=0,\; E_\R=I$, where $I$ is the identity,
\item for pairwise disjoint $A_1,A_2,\dots\in\Borel(\R)$ we have for any $g\in H$ that
	\[
		\sum_{i=1}^\infty E_{A_i}(g)=E_{\bigcup_{i=1}^\infty A_i}(g).
	\]
\end{enumerate}
If there exists a compact set $K\in\Borel(\R)$ with $E_K=I$,
then we say that the spectral measure has \emph{compact support}.
\end{defin}
For $f,g\in H$ a signed measure is defined on $(\R,\Borel(\R))$ by  
\[
\omega(A)=\scalar{E_A f}{g}, \quad A\in\Borel(\R).
\]
If $f=g$, then the measure $\omega$ is non-negative. 
Let $P\in \Le(H)$ be a self-adjoint operator and
let us denote the spectrum of $P\colon H\to H$ by $\spec(P|H)$. 
Furthermore let 
\[
\lambda 
=\inf_{\norm{g}{}=1} \scalar{Pg}{g} 
\quad\mbox{and}\quad
 \Lambda 
=\sup_{\norm{g}{}=1} \scalar{Pg}{g}.
\]
The spectrum of $P$ is closed and $\spec(P| H)\subset [\lambda,\Lambda]$. 
Additionally one has $\lambda,\Lambda \in\spec(P|H)$, thus 
\[
\lambda = \inf\set{\a \mid \a\in\spec(P| H)}
\quad\mbox{and}\quad
\Lambda = \sup\set{\a \mid \a\in\spec(P| H)}.
\]
Now we state the Spectral Theorem for linear bounded self-adjoint operators.
It is an analogon to the finite dimensional Spectral Theorem for matrices.
\begin{prop}[Spectral Theorem] \label{spec_thm}
Let $P\in\Le(H)$ be self-adjoint and $k\in\N$. 
Then there exists a uniquely determined spectral measure $E$ 
with compact support $\spec(P|H)$ such that
\begin{equation} \label{leb_stielt1}
\scalar{P^k f}{g} = \int_\lambda^\Lambda \a^k \;\dint\scalar{E_{\set{\a}}f}{g}, \quad f,g\in H.
\end{equation}
Let $F\colon [\lambda,\Lambda]\to \R$ be a continuous function.
Then one has by the continuous functional calculus a self-adjoint operator $F(P)\in\Le(H)$ with
\begin{equation} \label{leb_stielt2}
\scalar{F(P)f}{g}= \int_\lambda^\Lambda F(\a) \;\dint\scalar{E_{\set{\a}} f}{g}	\quad f,g\in H,
\end{equation}
and 
\begin{equation*} \label{spec_norm}
\norm{F(P)}{H\to H}= \max_{\a \in \spec(P|H)} \abs{F(\a)}.
\end{equation*}
\end{prop}
\begin{remark}
Mostly in the literature the case where $H$ is a complex Hilbert space is considered. 
In \cite{kirillov} they handle both, real and complex Hilbert spaces.  
Note that the integral in \eqref{leb_stielt1} and \eqref{leb_stielt2} 
is defined with respect to a signed measure. 
\end{remark}

\section{Interpolation Theorem}
We state a version of the Theorem of Riesz-Thorin. 
For a proof and further details let us refer to \cite{BeLof,bennett_sharpley}.
Let $L_p=L_p(D,\pi)$ for a probability measure $\pi$ on a measurable space $(D,\D)$.
\begin{prop}[Theorem of Riesz-Thorin] \label{riesz_thorin}
Let $1\leq p, q_1, q_2 \leq \infty$. 
We assume that $\theta \in (0,1)$ and 
\[
\frac{1}{p} = \frac{1-\theta}{q_1}+\frac{\theta}{q_2}.
\]
Further let $T$ be a linear operator from $L_{q_1}$ to $L_{q_1}$ 
and at the same time from $L_{q_2}$ to $L_{q_2}$ with
\begin{align*}
\norm{T}{L_{q_1} \to L_{q_1}} \leq M_1 \quad \text{and} \quad 
\norm{T}{L_{q_2} \to L_{q_2}} \leq M_2.
\end{align*}

Then 
\[
\norm{T}{L_p\to L_p} \leq 2 M_1^{1-\theta} M_2^\theta.
\]
\end{prop}  

\begin{remark}
We can substitute the function spaces $L_p,\,L_{q_1},\,L_{q_2}$ in the last proposition  
by the sequence spaces $\ell_p,\,\ell_{q_1},\,\ell_{q_2}$ and the result remains the same.
\end{remark}
\begin{remark}
Note that we consider real-valued functions. If we would study functions which map into
the complex numbers, then the same 
result holds true. In particular, the additional factor of two in the assertion is not needed.
\end{remark}

\section{Conductance and the Cheeger inequality}  \label{cond_concept}


Let $(D,\D)$ be a measurable space. Assume $K$ is a transition kernel defined on $(D,\D)$ which is reversible 
with respect to a probability measure $\pi$.
 The \emph{conductance} of the transition kernel $K$ is defined by
\[
\phi(K,\pi) =\inf_{0<\pi(A)\leq\frac{1}{2}}   
            \frac{\int_A K(x,A^c)\, \pi(\dint x)}{\pi(A)}.
\]
Let $(X_n)_{n\in\N}$ be a Markov chain with transition kernel $K$ and initial distribution $\pi$.
Then the numerator of the ratio within the definition of the conductance is the probability 
of $X_1 \in A$ and $X_2 \in A^c$. Hence one has
\[
\Pr(X_2\in A^c \mid X_1\in A) = \frac{\int_A K(x,A^c)\, \pi(\dint x)}{\pi(A)}.
\]
The conductance of $K$ is the infimum over sets $A\in\D$, with $0<\pi(A)\leq1/2$, 
of the probability that $X_2\in A^c$ under the condition that $X_1\in A$. 

The Markov operator $P$, given by $Pf(x)=\int_D f(y)\, K(x,\dint y)$, is self-adjoint on $L_2=L_2(D,\pi)$.
For $f\in L_2$ let $S(f)=\int_D f(x)\,\pi(\dint x)$ and let
\[
L_2^0=\{f\in L_2\mid S(f)=0\} .
\]
Furthermore we define
\[
  \Lambda = \sup \set{\a \mid \a \in \spec(P|L_2^0)}.
\]
The Cheeger inequality provides a relation between $\Lambda$ and the conductance $\phi(K,\pi)$.

\begin{prop}[Cheeger inequality] \label{cheeger}
Let the transition kernel $K$ be reversible with respect to a probability measure $\pi$. 
Then
\begin{equation} \label{cheeger_eq}
1-\Lambda \geq \frac{\phi(K,\pi)^2}{2}.
\end{equation}  
\end{prop}

For a proof of the inequality on finite state spaces  
we refer to \cite[Theorem~11.3, p.~93]{behrends}.
The Cheeger inequality for general state spaces is proven by Lawler and Sokal in \cite[Theorem~3.5, p.~570]{lawler_sokal} 
and by Lov{\'a}sz and Simonovits in \cite[Lemma~1.7, p.~374]{lova_simo1}.
Lawler and Sokal provide different types of inequalities for Markov chains and Markov processes.

\end{appendix}

\newcommand{\etalchar}[1]{$^{#1}$}
\providecommand{\bysame}{\leavevmode\hbox to3em{\hrulefill}\thinspace}
\providecommand{\MR}{\relax\ifhmode\unskip\space\fi MR }
\providecommand{\MRhref}[2]{%
  \href{http://www.ams.org/mathscinet-getitem?mr=#1}{#2}
}
\providecommand{\href}[2]{#2}

\end{document}